\documentclass[12pt]{amsart}

\usepackage{amsfonts}
\usepackage{amssymb}
\usepackage{amsmath,textcomp}
\usepackage{bm,nicematrix}
\usepackage{color}

\usepackage[pagebackref,breaklinks=false,pdftoolbar=false, colorlinks=true, citecolor=b,linkcolor=a]{hyperref}

\definecolor{a}{rgb}{0.2, 0.2, 0.6}
\definecolor{b}{rgb}{0.29, 0.33, 0.13}

\newtheorem{theorem}{Theorem}[section]
\newtheorem*{theorem-non}{Theorem}
\newtheorem{lemma}[theorem]{Lemma}
\newtheorem{remark}[theorem]{Remark}
\newtheorem{proposition}[theorem]{Proposition}
\newtheorem{corollary}[theorem]{Corollary}
\theoremstyle{definition}
\newtheorem{definition}[theorem]{Definition}
\newtheorem{example}[theorem]{Example}
\numberwithin{equation}{section}

\begin{document}
\title{Some new perspectives on $d$-orthogonal polynomials}
	\author{Abdessadek Saib}
\address{Department of Mathematics, University of Tebessa 12022, Algeria}
\email{sad.saib@gmail.com}

\begin{abstract}
The aim of this paper is twofold. The first part is concerned with the associated and the so-called co-polynomials, i.e. new sequences obtained when finite perturbations of the recurrence coefficients are considered. 
	In the second part we present nice Casorati determinants with co-polynomials entries. We look at Darboux factorization of lower Hessenberg matrices and the corresponding polynomials, then combine it with totally nonnegative matrices to find out sufficient conditions for zeros to be real and distinct. 
	In addition, kernel polynomials, $d$-symmetric sequences as well as quasi-orthogonality from linear combination are discussed as well. Therefore, new characterization of the $d$-quasi-orthogonality which corresponds to the first structure relation in the standard case is constructed. 
\end{abstract}

\maketitle


\section{Introduction}

This paper deals with the theory of $d$-orthogonal polynomials, that was initially introduced and discussed in the thesis of Van Iseghem \cite{Iseghem} in the study of vectorial Pad\'e approximations.
This study has been the building block which led to such observation that common denominator polynomials satisfy a $(d+2)$-term recurrence relation.

Two years later, this new theory was further discussed by Maroni \cite{Maronid}. Maroni's new algebraic approach figured out many interesting characterizations of the $d$-orthogonality relying almost all on the orthogonality's vectorial form. 
Besides, he introduced the $d$-quasi-orthogonality's notion in the same paper.
Since then, some attempts were given in order to improve, as well as, to understand this new theory.  
The challenge tackled by Douak and Maroni was Hahn's property. 
They succeeded one time by characterizing the class of $d$-orthogonal polynomials with $d$-orthogonal derivatives in terms of Pearson equation. 
The latter class (orthogonal polynomials sequence whose sequence of derivatives is also orthogonal) was referred to as Hahn classical $d$-orthogonal polynomials, according to Hahn's work on the characterization of classical orthogonal polynomials from algebraic point of view \cite{HahnClassic}.
Other researchers have also analyzed some characterization's problems which led to 
construct many  $d$-analogue of classical families of polynomials and discover some new ones.  

Classical orthogonal polynomials sequences (OPS in short) constitute significant class of special functions with wide range of their applications mainly in numerical analysis, probability and statistics, stochastic processes, combinatorics, number theory, potential theory, scattering theory, physics, biology, automatic control ... and more.  
Starting from classical OPS, there were many ways to formulate and introduce analogous problems in $d$-orthogonality context and the construction of new $d$-OPS families is then quite fruitful. Hence, several examples of $d$-OPS families which reduce to the classical cases are provided in literature.
In contrast to the classical OPS case, to the best of our knowledge, Pearson equation is the only characterization  which was extended to Hahn classical $d$-OPS. Although, Maroni has pointed out another one in the case $d$=2 \cite[Prop. 6.2]{MaroniFT}. 
In this paper, due to new characterizations of the $d$-quasi-orthogonality in terms of linear combinations (see proposition \ref{T16} and theorem \ref{NCQO2}), we are going to present two new characterizations of Hahn classical $d$-OPS in terms of linear combinations as well. We would like to emphasize that one of these latter gives us ideas to construct new polynomial families which possess Hahn's property. First results in this context could be found in \cite{Sadek6}.

Shohat used Pearson equation in his classical OPS's study 
which led later to discover the semi-classical OPS. He also coined the quasi-orthogonality's notion that was first introduced by Riesz in 1923. Later on, Al-Salam and Chihara stated in \cite{Al-Salam2} that classical OPS could be further characterized in terms of linear combination of OPS, i.e. with the aid of quasi-orthogonality. The latter result, led to significant development since the late of 1980's. 
Among the generalizations of classical OPS, an important problem was raised by Askey (see \cite[p.69]{Al-Salam2}) 
to characterize sequences of OPS whose derivatives are quasi-orthogonal.
Many authors were interested in studying Askey's problem by using different approaches to resolve the problem. One of the most successful of these was done by Maroni who presented in \cite{MaroniProlg} an algebraic theory of semi-classical OPS including very interesting  characterizations. Ever since, quasi-orthogonality has been aptly motivated. 
Indeed, quasi-orthogonality and linear combinations of OPS  as well as modifications of orthogonality's measures are a powerful tool in constructing new families of OPS. 

Riesz and Shohat discussed further the zeros of linear combinations of OPS. 
For a positive definite linear form, it is well known that zeros of the corresponding OPS are real and simple. Notice that the interlacing property is by no means sufficient to ensure orthogonality. 
Nevertheless, a necessary and sufficient conditions on the interlacing of zeros of two and three arbitrary polynomials that can be embedded in an OPS were discussed respectively by Wendroff in \cite{Wendroff} and in \cite{DriverZeroLinComb}. Furthermore, many recent results derive sufficient conditions for a linear combination of OPS to have simple zeros might be found. For the multiple and the $d$-OPS, it seems no longer true at the time! 

We are interested in the zeros of OPS because of their important role mainly in interpolation and approximation theory, Gauss-Jacobi quadrature, spectral theory, and in image analysis and pattern recognition. 
The first real investigation of the zeros of multiple OPS was not performed till 2011, when Haneczok with Van Assche gave sufficient conditions for zeros of the latter class to be real and distinct \cite{AsscheZero}, a result we shall build upon and generalize to the $d$-orthogonality with the aid of oscillation matrices which involves some interlacing properties.

However, the general theory is far from being complete, and many questions remain unanswered or have only limited explanations.  Intuitively, this is  why the theory is not so deeply yet understood. 
This study aims to significantly answer some open questions mainly for Hahn's property and the multiplicity as well as the nature of zeros. 
We desperately hope that our results could give some help to find new applications as for the usual OPS (standard case) which are ubiquitous in several areas. 
Our approach to study Hahn's property is very simple and based on new characterizations of the quasi-orthogonality from linear combination of polynomials. 
Although, one of these two new characterizations shows that the sequence of its normalized derivative possesses again Hahn's property. 
Thus, without doubt, this is an affirmative answer to a first open question.

The layout of the paper is as follows. After recalling the definition and some characterizations of the $d$-orthogonality, \autoref{sec:3} provides a deep study of the so-called co-polynomials reached by some modifications and perturbations in the recurrence coefficients. 
In other words, we discuss the associated as well as anti-associated polynomials, co-recursive polynomials, co-dilated polynomials and finally the co-modified polynomials where the solutions of these new families are explicitly singled out in terms of the fundamental set of solutions of a $(d+2)$-term recurrence relation. Some properties of the formal Stieltjes functions are also presented. 
In \autoref{sec:determ} we shed a spotlight on Casorati determinants. In other words, we consider some Casorati determinants which entries are d-OPS. Many identities satisfied by these type of determinants
with $d$-OPS and co-polynomials entries are given. Almost all our proofs in that section are based 
on one or two $(d+2)$-recurrence relation connects $d+2$ consecutive polynomials in level of association. 
In \autoref{sec:ChrDarbForm}, we deduce that some Casorati determinants also provide partial generalizations of Christoffel-Darboux type formulas for $d$-OPS.

\autoref{sec:lu} deals with Darboux transformations. 
We started by looking for the existence and the sequences generated 
by $LU$ and $UL$ factorization of lower Hessenberg matrix associated with a $d$-OPS. 
Fortunately, our study has been motivated by the fact that $d$-OPS generated by the first factorization  
is $d$-quasi-orthogonal of order one with respect to that generated by the second one. This is discussed in \autoref{sec:kern}. 
In fact, we have discovered a $d$-analogue of the first structure relation termed out by Maroni 
in his study of semi-classical polynomials (for usual orthogonality)  \cite{MaroniProlg}.
Moreover, some properties  have emerged in this section,  characterize 
kernel polynomials in the case $d=1$ from the quasi-orthogonality point of view. 
For this end, we have also looked at Uvarov transformation in order 
to find out further properties of kernel polynomials, since they were 
the crucial point in characterizing such transformations. 
However, for $d>1$, this seems no longer true. Nevertheless, in \autoref{sec:sym}, an extended investigation on the $(d+1)$-decomposition of a $d$-symmetric sequence, 
reveals that the latter fact could be a starting point for many studies, mainly for  zeros, and to 
construct new families of polynomials using products of bi-diagonal matrices. 
Roughly speaking, we have explicitly characterized Hahn classical $d$-OPS and showed further that the derivative sequence of any order 
of Hahn classical $d$-OPS always possesses Hahn property.
However our approach to investigate zeros, in \autoref{sec:zero}, is to appeal the theory of totally nonnegative (TN) matrices.
We started in a first time (see \autoref{sec:sym}) by considering higher order 
three term recurrence relation (all recurrence coefficients are zero except the last one), 
i.e., $d$-symmetric case  (\ref{SD3}), and then regarding for all nonsymmetric 
sequences in which their recurrence coefficients are expressed in terms of the 
nonzero parameter in $d$-symmetric sequence (\ref{SD3}) above.
Although this higher order recurrence relation has been investigated many times, 
see for instance \cite{Higher3term,GeneticSum,VargaZeroFaber}, our idea to find out some zero's 
properties  is based on previous reasonings and techniques, but technically 
more detailed, extremely intricate and somehow different of the approaches used 
in the aforecited references.  
Indeed, it is well known in this case, that there are $(d+1)$ nonsymmetric 
$d$-OPS families called components  of the $d$-symmetric sequence defined by 
this higher order three term recurrence relation (\ref{SD3}) \cite{Maroni2classic} 
(see also \cite{VargaZeroFaber} for a quite different construction). 
As a result, the recurrence coefficients of the above components are in fact the symmetric functions. 
Which  inevitably leads to think about TN matrices. Moreover, lower Hessenberg matrices of the components are product of 
$d$ lower bidiagonal TN matrices and one upper bidiagonal TN matrix whenever the coefficient of the $d$-symmetric sequence is strictly positive. Furthermore, a sufficient condition to be done for the oscillation is  that the recurrence coefficients should be strict positive. Consequently, this is a much simpler proof than that presented in \cite{VargaZeroFaber}.
Next, we have only showed that Hessenberg matrices are oscillation matrices whenever all the recurrence coefficients are strictly positive. Unfortunately, these conditions are sufficient but not necessary as $d$-Laguerre polynomials show.  
The outcome of interlacing properties could be deduced by appealing oscillation matrix's tools.

\section{Basic background}

Let $\left\{ P_{n}\right\} _{n\geq 0}$ be a monic  sequence  in the space of polynomials $\mathcal{P}$
with $\deg P_{n}=n$, $n\geq 0$. By the euclidean division, there always
exist complex sequences $\left\{ \beta _{n}\right\} _{n\geq 0}$, $\left(
\chi _{n,v}\right)$, $0\leq v\leq n$ such that 
\begin{equation*}
\begin{array}{l}
P_{0}\left( x\right) =1,\ \ \ P_{1}\left( x\right) =x-\beta _{0},\vspace{0.2cm} \\ 
P_{n+2}\left( x\right) =\left( x-\beta _{n+1}\right) P_{n}\left( x\right)
-\sum\limits_{v=0}^{n}\chi _{n,v}P_{v}\left( x\right) ,\ \ n\geq 0.
\end{array}
\label{Y2}
\end{equation*}
The dual sequence $\left\{ u_{n}\right\} _{n\geq 0}$, $u_{n}\in \mathcal{P}
^{\prime }$ of $\left\{ P_{n}\right\} _{n\geq 0}$ is defined by the duality bracket  denoted throughout as
$\left\langle u_{n},P_{m}\right\rangle :=\delta _{n,m},$ $n,m\geq 0$. 
The latter equality is sometimes called a bi-orthogonality between two sequences.
In particular, we denote by 
$\left( u_r\right)_{n}=\left\langle u_r,x^{n}\right\rangle$, $n\geq0$, the moments of $u_r$. 
Using the definition of dual sequence, it is easy seen that we have 
\begin{equation}
\begin{array}{rl}
\beta _{n} & =\left\langle u_{n},xP_{n}\left( x\right) \right\rangle ,\ \
n\geq 0,\vspace*{0.2cm} \\ 
\chi _{n,v} & =\left\langle u_{v},xP_{n+1}\left( x\right) \right\rangle ,\ \
0\leq v\leq n.%
\end{array}
\label{Y3}
\end{equation}

Now for any polynomial $\pi $ and any $c\in\mathbb{C}$, we can define the
following forms $Du=u^{\prime }$, $\pi u$ and $\delta _{c}$ by 
\begin{equation*}
\left\langle u^{\prime },p\right\rangle :=-\left\langle u,p^{\prime
}\right\rangle ,\quad \left\langle \pi u,p\right\rangle :=\left\langle u,\pi
p\right\rangle ,\quad \left\langle \delta _{c},p\right\rangle :=p\left(
c\right) ,\quad \ p\in \mathcal{P},
\end{equation*}
and for each $\lambda \in \mathbb{C}$ and $s\in \mathbb{N}$, we consider the
operators $\theta _{\lambda }$ and $\sigma$ defined respectively as 
\begin{equation}
\left( \theta _{\lambda }p\right) \left( x\right) =\frac{p\left( x\right)
	-p\left( \lambda \right) }{x-\lambda }, \qquad \left(\sigma_s
p\right)(x):=p(x^s), \quad \ p\in \mathcal{P}.  \label{E4}
\end{equation}

The right-multiplication of a form $u\in \mathcal{P}^{\prime }$ by a polynomial  $p\left( x\right) =\sum\nolimits_{\nu =0}^{m}a_{\nu }x^{\nu }$ is
\begin{equation}
\left( up\right) \left( x\right) =\sum\nolimits_{n=0}^{m}\left(
\sum\nolimits_{\nu =n}^{m}a_{\nu }\left( u\right) _{\nu -n}\right)
x^{n}=\left\langle u,
\frac{xP_{n+1}\left( x\right) -\xi P_{n+1}\left( \xi \right) }{x-\xi }\right\rangle,  \label{E2}
\end{equation}
then $(x-\lambda)^{-1}u$ and the Cauchy product of two forms are defined
\begin{equation}
\left\langle (x-\lambda)^{-1}u,p\right\rangle :=\left\langle u,\theta_{\lambda} p\right\rangle ,\quad 
\left\langle uv,p\right\rangle :=\left\langle u,vp\right\rangle.
\label{E3}
\end{equation}

For a linear form $u$, let $S\left( u\right) $ be its Stieltjes function
defined by 
\begin{equation*}
S\left( u\right) \left( z\right) =-\sum\nolimits_{n\geq 0}\frac{\left(
	u\right) _{n}}{z^{n+1}}=-\frac{1}{z}\left\langle u,\sum\nolimits_{n\geq
	0}\left( \dfrac{x}{z}\right) ^{n}\right\rangle =\left\langle u,\dfrac{1}{x-z}
\right\rangle ,
\end{equation*}
by all means, the functional (linear form) $u$ acts on the variable $x$. In
particular, we have $S\left( \delta \right) \left( z\right) =-1/z$. It is easily seen from the latter definition together with Cauchy product that $S\left( uv\right) \left( z\right) =-zS\left( u\right) \left(
z\right) S\left( v\right) \left( z\right) $. 
Stieltjes function, termed also Stieltjes transform, plays an important role
in the determination of the orthogonality's measure. It is worthwhile to
notice that if we know specifically a generating function $F(x)=\sum_n (u)_n
x^n$ of a moment sequence corresponding to the form $u$, then we could
determine Stieltjes transform explicitly as $zS(z)=-F(1/z)$. We prove, in the
next section, some new algebraic identities satisfied by Stieltjes function.
For this end, some properties are needed

\begin{lemma}\label{lem1}\cite{MaroniAlg}
	For any $p\in \mathcal{P}$ and any 
	$u,v\in \mathcal{P}^{\prime }$, we have	
	\begin{enumerate}
		\item[(a)] $p\left( x\right) \left( uv\right) =\left( p\left( x\right)
		u\right) v+x\left( u\theta _{0}p\right) \left( x\right) v$,
		
		\item[(b)] $S\left( pu\right) \left( z\right) =p\left( z\right) S\left(
		u\right) \left( z\right) +\left( u\theta _{0}p\right) \left( z\right) $.
	\end{enumerate}
\end{lemma}

Interested reader could find elementary proofs of the latter lemma as well as further interesting properties in Maroni's paper such as  \cite{MaroniSDF}.

Before we dive into the $d$-orthogonality, let us briefly 
recall the standard orthogonality in formal essence.  
The sequence $\left\{ P_{n}\right\} _{n\geq 0}$ is said to be orthogonal
with respect to some linear form (called also moment functional) $u$, if 
\begin{equation*}
\left\langle u,P_{m}P_{n}\right\rangle :=r_{n}\delta _{n,m},\quad n,m\geq
0,\quad r_{n}\neq 0,\ n\geq 0.
\end{equation*}
In this case, necessarily, $u=\lambda u_{0}$, $\lambda \neq 0$. Further, we
have $u_{n}=\left( \left\langle u_{0},P_{n}^{2}\right\rangle \right)
^{-1}P_{n}u_{0}$, $n\geq 0$ \cite{MaroniAlg}. In terms of (\ref{Y3}), 
$\left\{ P_{n}\right\} _{n\geq 0}$ is orthogonal if and only if 
\begin{equation*}
\chi _{n,v}=0, \ 0\leq v\leq n-1, \ n\geq 1 \ \ \text{and} \ \ \chi
_{n,n}\neq 0, \ \ n\geq 0.
\end{equation*}

For a generalization of the above standard orthogonality, we will deal with
the concept of $d$-orthogonality (it can be also regarded as the type II multiple OPS at the step line). 
Let us recall the definition and some characterizations which will be needed in the sequel. 
Throughout this paper, all the sequence of polynomials are supposed to be monic.

\begin{definition}
	\label{T4}\cite{Maronid,Maroni2} A sequence of monic polynomials $\left\{
	P_{n}\right\} _{n\geq 0}$ is said to be a $d$-orthogonal polynomial
	sequence, in short a $d$-OPS, with respect to the $d$-dimensional vector of
	linear forms $\mathcal{U}=\left( u_{0},...,u_{d-1}\right) ^{T}$ if 
	\begin{equation}
	\left\{ 
	\begin{array}{l}
	\left\langle u_{t},x^{m}P_{n}\left( x\right) \right\rangle =0,\quad n\geq
	md+t+1,\quad m\geq 0,\vspace{0.2cm} \\ 
	\left\langle u_{t},x^{m}P_{md+t}\left( x\right) \right\rangle \neq 0,\quad
	m\geq 0,%
	\end{array}%
	\right.  \label{B5}
	\end{equation}
	for each $0\leq t\leq d-1.$
\end{definition}

The first and second conditions of (\ref{B5}) are called respectively the
$d$-orthogonality conditions and the $d$-regularity conditions. In this case,
the $d$-dimensional form $\mathcal{U}$ is called regular. Notice further that if 
$d=1$, then we meet again the notion of ordinary orthogonality. 

For this type of orthogonality we also have fundamental  characterizations some of them are in the following theorem in which the (b) constitutes the analogue of Favard's theorem.

\begin{theorem}
	\label{BT2}\cite{Maronid} Let $\left\{ P_{n}\right\} _{n\geq 0}$ be a monic
	sequence of polynomials, then the following statements are equivalent.
	
	\begin{enumerate}
		\item[ (a)] The sequence $\left\{ P_{n}\right\} _{n\geq 0}$ is $d$-OPS with
		respect to $\mathcal{U}=\left( u_{0},...,u_{d-1}\right)$.
		
		\item[ (b)] The sequence $\left\{ P_{n}\right\} _{n\geq 0}$ satisfies a 
		$(d+2)$-term recurrence relation 
		\begin{equation}
		P_{m+d+1}\left( x\right) =\left( x-\beta _{m+d}\right) P_{m+d}\left(
		x\right) -\sum\nolimits_{\nu =0}^{d-1}\gamma _{m+d-\nu
		}^{d-1-\nu}P_{m+d-1-\nu }\left( x\right) ,\ \ m\geq 0,  \label{B6}
		\end{equation}
		with the initial data 
		\begin{equation}
		\left\{ 
		\begin{array}{l}
		P_{0}\left( x\right) =1,\ \ P_{1}\left( x\right) =x-\beta _{0},\vspace{0.15cm} \\ 
		P_{m}\left( x\right) =\left( x-\beta _{m-1}\right) P_{m-1}\left( x\right)
		-\sum\nolimits_{\nu =0}^{m-2}\gamma _{m-1-\nu }^{d-1-\nu}P_{m-2-\nu }\left(
		x\right) ,\ \ 2\leq m\leq d,
		\end{array}
		\right.  \label{B7}
		\end{equation}
		and the regularity conditions $\gamma _{m+1}^0\neq 0$,  $m\geq 0$.
		
		\item[(c)] For each $\left( n,t \right)$, $n\geq 0,$ $0\leq t \leq d-1$,
		there exist $d$ polynomials $\phi _{t ,s }$, $0\leq s \leq d-1$ such
		that 
		\begin{equation*}
		u_{nd+t }=\sum\nolimits_{s =0}^{d-1}\phi _{t,s }u_{s },\ \ n\geq
		0,\ \ 0\leq t \leq d-1,  \label{B8}
		\end{equation*}
		and verifying 
		\begin{equation*}
		\left\{ 
		\begin{array}{lll}
		\deg \phi _{t,t}=n, & 0\leq t \leq d-1, & \text{and if }\ d\geq 1, 
		\vspace{0.2cm} \\ 
		\deg \phi _{t,s}\leq n, & 0\leq s \leq t -1, & \text{if }\ 1\leq
		t \leq d-1,\vspace{0.2cm} \\ 
		\deg \phi _{t,s}\leq n-1, & t+1\leq s \leq d-1, & \text{if }\
		0\leq t \leq d-2.%
		\end{array}
		\right.  \label{B9}
		\end{equation*}
		
		\item[(d)] The dual sequence satisfies
		\begin{equation}
		xu_{n}=u_{n-1}+\beta _{n}u_{n}+\sum\nolimits_{\nu =0}^{d-1}\gamma
		_{n+1}^{d-1-\nu }u_{n+1+\nu },\qquad n\geq 0\qquad(u_{-1}=0).  \label{E11}
		\end{equation}
	\end{enumerate}
\end{theorem}

Now, if we multiply the recurrence of $P_{\left( n+1\right) d+r}$ by $x^{n}$, 
we get under the action of $u_r$ 
\begin{equation*}
\left\langle u_{r},x^{n+1}P_{\left( n+1\right) d+r}\right\rangle =\gamma
_{nd+r+1}^0\left\langle u_{r},x^{n}P_{nd+r}\right\rangle ,
\end{equation*}
and then 
\begin{equation}
\prod\nolimits_{\nu =0}^{n}\gamma _{\nu d+r+1}^0=\frac{\left\langle
	u_{r},x^{n+1}P_{\left( n+1\right) d+r}\right\rangle }{\left\langle
	u_{r},P_{r}\right\rangle },\ \ 0\leq r\leq d-1.  \label{CO1}
\end{equation}

When $r=d-1$ in (\ref{CO1}), and if we set $\left\langle
u_{d-1},P_{d-1}\right\rangle =\gamma _0^0$, we obtain 
\begin{equation}
\prod\nolimits_{\nu =0}^{n}\gamma _{\nu d}^0=\left\langle
u_{d-1},x^{n}P_{\left( n+1\right) d-1}\right\rangle .  \label{CO2}
\end{equation}

Moreover, and in a similar way we have \cite{Maronid} 
\begin{equation}
\begin{array}{rl}
\beta _{\nu }= & \left\langle u_{\nu },xP_{\nu }\right\rangle \text{\ },\ \
0\leq \nu \leq d-1,\vspace{0.2cm} \\ 
\gamma _{\nu }^{\nu +r}= & \left\langle u_{\nu -1},xP_{d-1-r}\right\rangle 
\text{\ },\ \ 1\leq \nu \leq d-1-r,\ \ 0\leq r\leq d-2,\vspace{0.2cm} \\ 
\gamma _{n+1+\nu }^{\nu}= & \left\langle u_{n+\nu },xP_{n+d}\right\rangle 
\text{\ },\ \ 0\leq \nu \leq d-1,\hspace{0.5cm}n\geq 0.%
\end{array}
\label{R1}
\end{equation}

The inspection of $UL$ decomposition of lower Hessenberg matrix bellow, reveals a glimpse on kernel polynomials. Accordingly, the sequence of polynomials generated by the matrix $LU$ is $d$-quasi-orthogonal of order one with respect to that generated by $UL$. We defer further details concerning Darboux transformation and quasi-orthogonality to Section \ref{sec:lu} and \ref{sec:kern}, respectively. 

\begin{definition}
	\label{T10}\cite{Maronid} A sequence $\left\{ P_{n}\right\} _{n\geq 0}$ is
	said $d$-quasi-orthogonal of order $l$ with respect to the form $\mathcal{U}
	=\left( u_{0},...,u_{d-1}\right) ^{T}$, if for every $0\leq t \leq d-1$, there exist $l_{t}\geq 0$ and $s _{t}\geq l_{t }$ integers such that 
	\begin{equation}
	\left\{ 
	\begin{array}{l}
	\left\langle u_{t},P_{m}P_{n}\right\rangle =0,\ \ 
	n\geq \left( m+l_{t}\right) d+t+1,\ \ m\geq 0,\vspace{0.2cm} \\ 
	\left\langle u_{t},P_{s _{t}}P_{\left(s		_{t}+l_{t}\right)d+t}\right\rangle \neq 0,\ \ m\geq 0,
	\end{array}%
	\right.  \label{E12}
	\end{equation}
	for every $0\leq t \leq d-1$. We put $l=\underset{0\leq t \leq d-1}{\max }l_{t}$.
\end{definition}

Unfortunately, some characterizations of the $d$-quasi-orthogonality proved by Maroni \cite{Maronid} subject to some relations only between forms. In the usual orthogonality, we emphasize that an OPS is also
quasi-orthogonal of order $l$ with respect to another regular form, if and only if
it is a linear combination of $l$ terms of the corresponding sequence of the
second form. The latter has been generalized to the $d$-orthogonality as follows

\begin{proposition}
	\label{T16}\cite{Sadek2} For any two $d$-OPS's $\left\{ P_{n}\right\}
	_{n\geq 0}$ and $\left\{ Q_{n}\right\} _{n\geq 0}$ relative to $\mathcal{U}$
	and $\mathcal{V}$, respectively, the following are equivalent
	
	\begin{enumerate}
		\item There exists a matrix polynomial $\Phi=\left( \phi _{r}^{s}\right) $, $0\leq s ,r \leq d-1$
		such that 
		\begin{equation*}
		v_r=\sum\nolimits_{s =0}^{d-1}\phi _r^{s}u_{s },\ \ 0\leq r \leq d-1,
		\label{E15}
		\end{equation*}
		where 
		\begin{equation*}
		\begin{array}{lll}
		\deg \phi _r^{r}=l, &  & \vspace{0.2cm} \\ 
		\deg \phi _{r}^{s }\leq l, & 0\leq s \leq r -1, & \text{if }\ 1\leq r \leq
		d-1,\vspace{0.2cm} \\ 
		\deg \phi _{r}^{s }\leq l-1, & r +1\leq s \leq d-1, & \text{if }\ 0\leq r
		\leq d-2.%
		\end{array}
		\label{E16}
		\end{equation*}
		
		\item there exists a non negative integer $l$ such that 
		\begin{equation}
		P_{n}\left( x\right) =Q_{n}\left( x\right)
		+\sum\nolimits_{i=1}^{dl}a_{n,i}Q_{n-i}\left( x\right) ,\ \ n\geq dl,
		\label{ES}
		\end{equation}
		with $a_{n,dl}\neq 0$.
	\end{enumerate}
\end{proposition}

The above characterization (\ref{ES}) reduces to \cite[p. 294]{MarcFuncAppr} 
first proved for classical OPS ($d$=1). It is worthy to notice here that the length of the expansion (\ref{ES}) depends closely on the $\deg \phi _r^{r}$. Therefore, in the above Proposition it was assumed that the degree is exactly $l$ which is equivalent to saying that the quasi-orthogonality is of order exactly $l$. It took five years after the above characterization to come up with a full description of the connection between the order of the quasi-orthogonality and the length of the latter expansion.
Indeed, as in the definition \ref{T10}, Maroni defines the $d$-quasi-orthogonality when $\left\langle u_{t},P_{n}\right\rangle =0,\ n> dl+t $, the question now is what happens between $d(l-1)+1$ and $dl-1$? In order to discuss this situation we shall distinguish between the $d$-quasi-orthogonality of order exactly $l$ and at most $l$. 
\begin{definition}\label{NDQO1}
	A sequence $\left\{ P_{n}\right\} _{n\geq 0}$ is
	$d$-quasi-orthogonal of order at most $l$ with respect to the form $\mathcal{U}=\left( u_{0},...,u_{d-1}\right) ^{T}$, if there exists an integer $1\leq r \leq d$ such that for every $0\leq t \leq d-1$, there exist $l_{t}\geq 0$ and $s _{t}\geq l_{t }$ integers such that 
	\begin{equation}
	\left\{ 
	\begin{array}{l}
	\left\langle u_{t},P_{m}P_{n}\right\rangle =0,\ \ 
	n\geq (m+l_{t}-1) d+r+t+1,\ \ m\geq 0,\vspace{0.2cm} \\ 
	\left\langle u_{t},P_{s _{t}}P_{\left(s		_{t}+l_{t}-1\right)d+r+t}\right\rangle \neq 0,\ \ m\geq 0.
	\end{array}%
	\right.  \label{NQO2}
	\end{equation}
\end{definition}

In this situation, the matrix polynomial of the latter Proposition is divided into three bands according to the degree of their entries (see Remark \ref{RBand}). When $r=d$ two zones A and B are obtained (zone C in which the degree is less or equals to $l-2$ disappears) and again the characterization of the Proposition \ref{T16} is recovered. The latter helps, from one hand to close the implication between the first and the second structure relation, and from the other hand to see exactly the meaning of the entries of the matrix $\Phi$ at Pearson equation constructed by Douak and Maroni in 1995 \cite[eq. (3.3), p. 183]{DouakCar}. 
Indeed, the latter corresponds to the case $l=2$ and $r=1$ (see \autoref{sec:kern}).  

Besides, in the study of the regularity of linear combinations of polynomials, we are also interested in the determination of the matrix polynomial, i. e. the link between the two vector forms, and this is based on the following useful characterization theorem

\begin{theorem}
	\label{T6}\cite{Maronid} For each sequence $\left\{ P_{n}\right\} _{n\geq 0}$
	$d$-OPS with respect to $\mathcal{U}$, then the following statements are
	equivalent.
	
	\begin{enumerate}
		\item[(i)] There exist $\mathcal{L}\in \mathcal{P}^{\prime }$ and a
		nonnegative integer number $s$ such that 
		\begin{equation*}
		\begin{array}{ccc}
		\left\langle \mathcal{L},P_{n}\right\rangle =0,\ \ n\geq s+1 & \ \ \text{
			and\ \ } & \left\langle \mathcal{L},P_{s}\right\rangle \neq 0.
		\end{array}%
		\end{equation*}
		
		\item[(ii)] There exist $\mathcal{L}\in \mathcal{P}^{\prime }$, 
		a nonnegative integer number $s$, and $d$ polynomials $\phi ^{\alpha }$, 
		$0\leq \alpha \leq d-1$, such that $\mathcal{L}=\sum\nolimits_{\alpha
			=0}^{d-1}\phi ^{\alpha }u_{\alpha }$ with the following properties: if 
		$s=qd+r$, $0\leq r\leq d-1$, we have 
		\begin{equation*}
		\begin{array}{lll}
		\deg \phi ^{r}=q, & 0\leq r\leq d-1, & \text{and if }\ d\geq 2,\vspace{0.2cm}
		\\ 
		\deg \phi ^{\alpha }\leq q, & 0\leq \alpha \leq r-1, & \text{if }\ 1\leq
		r\leq d-1,\vspace{0.2cm} \\ 
		\deg \phi ^{\alpha }\leq q-1, & r+1\leq \alpha \leq d-1, & \text{if }\ 0\leq
		r\leq d-2.
		\end{array}
		\end{equation*}
	\end{enumerate}
\end{theorem}

The most notable moment functionals are those in the positive definite
case. For instance, in this case, the zero of the corresponding sequence of
orthogonal polynomials exhibit special features. We will
back to this context later in section \ref{sec:zero}, by starting the definition of positive definite moment functional, in the sense of Chihara \cite{Chihara}, which can be extended in a natural way into the $d$-orthogonality and that the
positive definiteness may be characterized by the recurrence coefficients.

Actually, an OPS can be seen as the characteristic polynomial of a
certain tridiagonal matrix. Therefore, it is not surprising that quite a few
results on OPS can be verified with tools from matrix theory. 

In this case, it is well known that we can express (\ref{B6}) in terms of matrices as 
$x\mathbb{P}=J_{d}\mathbb{P}:=J\mathbb{P}$ (see (\ref{NOT1}) below)  where 
$J_{d}=\left( a_{i,j}\right) _{i,j=0}^{\infty }$ is a $\left( d+2\right) $
-banded lower Hessenberg matrix, i.e., that is to say
\begin{equation}
\left\{ 
\begin{array}{rl}
a_{i,i+1} & =1,\ \ i\geq 0 \vspace{0.2cm} \\ 
a_{i,i} & =\beta_i,\ \ i\geq 0 \vspace{0.2cm} \\ 
a_{i+r,i} & =\gamma_{i+1}^{d-r},\ \ i\geq 0, \ \ 1\leq r\leq d.
\end{array}
\right.  \label{Y9}
\end{equation}
Matrix (\ref{Y9}) is called the monic lower Hessenberg matrix of the monic $d$-OPS $\left\{P_{n}\right\} _{n\geq 0}$.

To describe our results, we have introduced some notation which will be kept
throughout. For the sake of simplicity, we will often use the following notations 
\begin{equation}
\begin{array}{rlcrl}
\mathbb{A}_{k}^{r} & =\left( A_{k}^{\left( r\right) },A_{k+1}^{\left(
	r\right) },...,A_{k+d}^{\left( r\right) }\right) ^{T}& & 	
\mathbf{A}_{k,-l}^{r} & =\left( A_{k}^{\left( r\right) },A_{k-1}^{\left(
	r+1\right) },...,A_{k-d+l}^{\left( r+d-l\right) }\right)
\vspace{.2cm}\\
\mathbb{A}^{r} & =\left( A_0^{ (r) },A_1^{(r) },... \right) ^{T}    &  
\hspace{-2mm}\mathbf{A}_{k}^{r}&:=\mathbf{A}_{k,0}^{r} & =\left( A_{k}^{\left( r\right)
},A_{k-1}^{\left( r+1\right) },...,A_{k-d}^{\left( r+d\right) }\right).	
\end{array}\label{NOT1}
\end{equation}

\section{Modification of the recurrence coefficients} \label{sec:3}

Some modifications of the recurrence coefficients in equations 
(\ref{B6})-(\ref{B7}), lead to new families of $d$-OPS such as the associated and the
co-recursive polynomials as well as to some interesting Hessenberg matrices.
Indeed, by deleting the first r rows and columns from the lower Hessenberg matrix, the
corresponding OPS are the associated polynomials of order
r, denoted by $P_{n}^{\left( r\right) }$. Instead of deleting rows and
columns, if we add r new rows and columns at the beginning of the lower Hessenberg
matrix, then the corresponding new OPS are called
anti-associated of order r denoted by $P_{n}^{\left( -r\right) }$ \cite{Ronv96}.

The purpose of this section is to discuss the associated and the
anti-associated polynomials in a greater generality manner and further, to introduce
particular perturbation of the coefficients. These families were initialized
for $d\geq 2$ in \cite{Maroni2,Sadek1}, yet, we repeat some main results here in view of uniformity of treatment and for completeness.

\subsection{The associated sequence}

The associated sequence of $\left\{ P_{n}\right\} _{n\geq 0}$ (with respect
to $u_{0}$), is the sequence $\{P_{n}^{\left( 1\right) }\}_{n\geq 0}$
defined by 
\begin{equation}
P_{n}^{\left( 1\right) }\left( x\right) =\left\langle u_{0},
\frac{P_{n+1}\left( x\right) -P_{n+1}\left( \xi \right) }{x-\xi }\right\rangle,\quad n\geq 0.  \label{P1}
\end{equation}
$P_{n}^{\left( 1\right) }$ is a monic polynomial of degree $n$. 
Therefore, from (\ref{E2}), we have
\begin{equation}
P_{n}^{\left( 1\right) }\left( x\right) =\left( u_{0}\theta
_{0}P_{n+1}\right) \left( x\right) ,\quad n\geq 0.  \label{E31}
\end{equation}
Let us
denote by $\{u_{n}^{\left( 1\right) }\}_{n\geq 0}$ the dual sequence of 
$\{P_{n}^{\left( 1\right) }\}_{n\geq 0}$. 
Then, it results straightforwardly
from the left product of a form by polynomial that \cite{Sadek1}
	\begin{equation}
	\begin{array}{l}
	\left\{ 
	\begin{array}{l}
	u_{\nu }^{\left( 1\right) }=x\left( u_{\nu +1}u_{0}^{-1}\right) ,\quad 0\leq
	\nu \leq d-2,\quad d\geq 2,\vspace{0.2cm} \\ 
	\gamma _{1}^{0}u_{d-1}^{\left( 1\right)
	}=-x^{2}u_{0}^{-1}-\sum\nolimits_{\nu =0}^{d-2}\gamma _{1}^{d-1-\nu }x\left(
	u_{\nu +1}u_{0}^{-1}\right) .
	\end{array}
	\right.
	\end{array}
	\label{A3}
	\end{equation}

Accordingly, the successive associated sequences are defined recursively 
\cite{Maroni2} 
\begin{equation*}
P_{n}^{\left( r+1\right) }\left( x\right) =\left( P_{n}^{\left( r\right)
}\left( x\right) \right) ^{\left( 1\right) }\quad \text{and}\quad
u_{n}^{\left( r+1\right) }=\left( u_{n}^{\left( r\right) }\right) ^{\left(
	1\right) },\quad n,r\geq 0,
\end{equation*}
with $P_{n}^{\left( 0\right) }=P_{n}$ and $u_{0}^{\left( 0\right) }=u_{0}$.
That is to say 
\begin{equation}
P_{n}^{\left( r+1\right) }\left( x\right) =\left( u_{0}^{\left( r\right)
}\theta _{0}P_{n+1}^{\left( r\right) }\right) \left( x\right) ;\quad
u_{n}^{\left( r+1\right) }=\left( xu_{n+1}^{\left( r\right) }\right) \left(
u_{0}^{\left( r\right) }\right) ^{-1},\ n,r\geq 0.  \label{P5}
\end{equation}

When $\{P_{n}\}_{n\geq 0}$ is $d$-orthogonal with respect to 
$\mathcal{U}=\left( u_{0},...,u_{d-1}\right) $, it verifies a recurrence relation of
type (\ref{B6}), we deduce immediately that the associated sequence 
$\{P_{n}^{\left( r\right) }\}_{n\geq 0}$ satisfies the following recurrence
relation 
\begin{equation}
P_{m+d+1}^{\left( r\right) }\left( x\right) =\left( x-\beta _{m+d+r}\right)
P_{m+d}^{\left( r\right) }\left( x\right) -\sum\nolimits_{\nu
	=0}^{d-1}\gamma _{m+d+r-\nu }^{d-1-\nu}P_{m+d-1-\nu }^{\left( r\right)
}\left( x\right) \text{\ for }m\geq 0,  \label{P6}
\end{equation}%
with the initial conditions 
\begin{equation}
\begin{tabular}{ll}
$P_{0}^{\left( r\right) }\left( x\right) $ & $=1,\quad P_{n}^{\left(
	r\right) }\left( x\right) =x-\beta _{r},\vspace{0.15cm}$ \\ 
$P_{m}^{\left( r\right) }\left( x\right) $ & $=\left( x-\beta
_{m+r-1}\right) P_{m-1}^{\left( r\right) }\left( x\right)
-\sum\nolimits_{\nu =0}^{m-2}\gamma _{m+r-1-\nu }^{d-1-\nu}P_{m-2-\nu
}^{\left( r\right) }\left( x\right) ,\quad 2\leq m\leq d.$%
\end{tabular}
\label{P7}
\end{equation}

Furthermore, the sequence of polynomials and its corresponding associated
sequence are also connected through \cite{Maroni2}
\begin{equation*}
P_{n}^{\left( r+1\right) }\left( x\right) =\left( u_{r}\theta
_{0}P_{n+r+1}\right) \left( x\right) ,\quad n,r\geq 0.
\end{equation*}

We have discovered in \cite{Sadek1} the following connections between the successive associated sequences 
	\begin{equation}
	u_{n}^{\left( r+1\right) }=x^{r+1}\left( u_{n+r+1}\right) \left(
	u_{0}^{\left( 0\right) }u_{0}^{\left( 1\right) }u_{0}^{\left( 2\right)
	}...u_{0}^{\left( r\right) }\right) ^{-1}.\label{P8} 	\end{equation}
Moreover, it is worthy to notice that the right hand side of the latter equality could be reduced in concise way and also have the following representation which is new in the literature
\begin{proposition}\label{TP2}
We have 
\begin{align}
	u_{n}^{\left( r+1\right) }
	&=\left( xu_{n+r+1}\right)u_{r}^{-1},\quad n,r\geq 0.  \label{P9}
\end{align}
\end{proposition}

\begin{proof}
Starting from (\ref{P8}) we deduce when $n=0$, that 
	\begin{equation*}
x^{r}u_{r}=	u_{0}^{\left( 0\right) }u_{0}^{\left( 1\right) }u_{0}^{\left( 2\right)}...u_{0}^{\left( r\right) }.
	\end{equation*}%
Hence, using (\ref{P8}) once again
	\begin{equation*}
	u_{n}^{\left( r\right) }\left( x^{r-1}u_{r-1}\right) 
	=x^{r}u_{n+r}.
	\end{equation*}
	
	Now the left product of a form by polynomial completes the justification.
\end{proof}

When $r=0$, our results (both formulas (\ref{P8}-\ref{P9}) reduce again to the classical result of Maroni, i.e.	$u_{n}^{\left( 1\right) }=\left( xu_{n+1}\right) u_{0}^{-1}$ \cite[(1.9)]{Maronid}. Furthermore, some formulas corresponding to the Stieltjes function could be readily termed out from the latter Proposition.

When $r>d$, we can express the element of the sequence $\{P_{n}^{\left( r\right) }\}_{n\geq 0}$ 
in terms of the original polynomials and their first $d$ consecutive polynomials 
in association as follows 
\begin{equation*}
P_{n}^{\left( r\right) }=a_{1}\left( x\right) P_{n+r}+a_{2}\left( x\right)
P_{n+r-1}^{\left( 1\right) }+...+a_{d+1}\left( x\right) P_{n+r-d}^{\left(
	d\right) }.
\end{equation*}

As a matter of fact, the polynomials $\left\{ P_{n},P_{n-1}^{\left( 1\right)},...,P_{n-d}^{\left( d\right) }\right\} $ constitute the fundamental system of solution of the linear recurrence (\ref{B6}). They are linearly independent as shown in Proposition \ref{BT10} in below. To compute the coefficients $\left\{
a_{i}\left( x\right) \right\} _{i=1}^{d}$ the initial conditions 
$P_{0}^{\left( r\right) }=1$ and $P_{-n}^{\left( r\right) }=0$ if $1\leq n\leq d$ are used (see also \cite{Assche}).

For the anti-associated $d$-OPS, the corresponding lower Hessenberg
matrix, denoted $J_{d}^{\left( -r\right) }$ \cite{Ronv96}, contains
$\left(d+1\right) r$ new parameters and they satisfy
\begin{equation*}
\left( P_{n+r}^{\left( -r\right) }\left( x\right) \right) ^{\left( k\right)
}=P_{n+r}^{\left( k-r\right) }\left( x\right) .
\end{equation*}

For $r=1$ we have
\begin{equation*}
J_{d}^{\left( -1\right) }=\left( 
\begin{array}{cc}
\beta _{-1} & 1 \\ 
\Gamma & J_{d}
\end{array}
\right) , \ \Gamma^T =\left( 
\gamma _0^{d-1}, \cdots , \gamma _0^{0}
\right) .
\end{equation*}

The associated sequence as well as the anti-associated sequence can be both
defined by (\ref{P1}). Indeed, the anti-associated polynomials are defined
by means of (\ref{P1}) as follows 
\begin{equation*}
P_{n}^{\left( -r\right) }\left( x\right) =\left\langle u_{0},
\frac{P_{n-r}\left( x\right) -P_{n-r}\left( \xi \right) }{x-\xi }\right\rangle,\quad n\geq 0.
\end{equation*}

The latter one allows us to verify that the family of $P_{n}^{\left( -r\right) }$
satisfies the recurrence (\ref{P6}) by shifting $\beta _{m+d+r}$ and $\gamma
_{m+d+r-\nu }^{d-1-\nu}$ to $\beta _{m+d-r}$ and $\gamma _{m+d-r-\nu
}^{d-1-\nu}$, respectively. Moreover, if we denote the corresponding dual sequence of the anti-associated polynomials by $\mathcal{U}^{\left(
	-r\right) }=\left( u_{0}^{\left( -r\right) },...,u_{d-1}^{\left( -r\right)
}\right) $, then following the same idea used in \cite{Maronid}, we recover a closed connection 
between $P_{n}^{\left( -1\right) }$ and $u_{0}^{\left( -1\right) }$ to the original
sequences as well 
\begin{equation*}
P_{n}^{\left( -1\right) }(x)=(u_{0}\theta _{0}P_{n-1})(x)\text{\ \ and\ \ }
u_{n}^{\left( -1\right) }=\left( xu_{n-1}\right) u_{0}^{-1},\ \ n\geq 1.
\end{equation*}

Since the anti-associated polynomials are also $d$-OPS, then we could  expand them as a linear combination of the basic solutions of their ($d+2$)-term recurrence relation as follows 
\begin{equation*}
P_{n}^{\left( -r\right) }=b_{1}\left( x\right) P_{n+r}+b_{2}\left( x\right)
P_{n+r-1}^{\left( 1\right) }+...+b_{d+1}\left( x\right) P_{n+r-d}^{\left(
	d\right) }.
\end{equation*}

The coefficients $\left\{ b_{i}\left( x\right) \right\} _{i=1}^{d}$ can be
determined using the initial conditions $P_{0}^{\left( -r\right) }=1$ and 
$P_{-n}^{\left( -r\right) }=0$ if $1\leq n\leq d$.

\subsection{Finite modifications}

The general modification consists in perturbing some terms of the sequences 
$\left\{ \beta _{n}\right\} _{n\geq 0}$ and $\left\{ \gamma _{n}^{\nu}; 0\leq
\nu\leq d-1\right\} _{n\geq 0}$ of the recurrence by adding or multiplying by
some complex numbers.

\subsubsection{Co-Recursive polynomials}

Initially, let us explain a little bit the construction of 
co-recursive sequences. Given an array $\eta
=\left\{ \eta _{n+1}^{\nu }\right\} _{0\leq n\leq d-1-\nu ,1\leq \nu \leq d-1}$ and a $d$-dimensional
vector $\mu =(\mu _{0},\mu _{1},\dots ,\mu _{d-1})$,  the co-recursive sequence $\left\{ P_{n}^{c}\right\} _{n\geq	0}$ 
where $P_{n}^{c}(x):=P_{n}(x,\mu ,\eta )$, is defined by modifying the initial values
of the sequences $\left\{ \beta _{n}\right\} _{n\geq 0}$ and $\left\{ \gamma
_{n}^{\nu};1\leq \nu\leq d-1\right\} _{n\geq 0}$ as follows \cite{Maroni2,Sadek1,Sadek4} 
\begin{equation}
\begin{tabular}{rllll}
$P_{0}^{c}\left( x\right) $ & $=1,$ & \hspace{0.75cm} & $P_{1}^{c}\left(
x\right) $ & $=x-\alpha _{0},\vspace{0.2cm}$ \\ 
$P_{m}^{c}\left( x\right) $ & \multicolumn{4}{l}{$=\left( x-\alpha
	_{m-1}\right) P_{m-1}^{c}\left( x\right) -\sum\nolimits_{\nu=0}^{m-2}\xi
	_{m-1-\nu}^{d-1-\nu}P_{m-2-\nu}^{c}\left( x\right) ,\quad 2\leq m\leq d,\vspace{0.2cm}$} \\ 
$P_{m+d+1}^{c}\left( x\right) $ & \multicolumn{4}{l}{$=\left( x-\beta
	_{m+d}\right) P_{m+d}^{c}\left( x\right) -\sum\nolimits_{\nu=0}^{d-1}\gamma
	_{m+d-\nu}^{d-1-\nu}P_{m+d-1-\nu}^{c}\left( x\right) ,\ \ m\geq 0,$}
\end{tabular}
\label{P11}
\end{equation}%
where $\xi _{n}^{0}=\gamma _{n}^{0},$ $\forall n\geq 1$ and 
\begin{equation}
\begin{tabular}{rlll}
$\alpha _{n}$ & $=\beta _{n}+\mu _{n}$, & \hspace{0.35cm} & for $0\leq n\leq
d-1,\vspace{0.2cm}$ \\ 
$\xi _{n}^{\nu}$ & $=\gamma _{n}^{\nu}+\eta _{n}^{\nu}$ &  & for $1\leq n\leq \nu,\
\ 1\leq \nu\leq d-1.$
\end{tabular}
\label{P13}
\end{equation}

Further generalizations of this perturbation can be done by translating the
perturbation at level $k\geq 0$. That is to say, we define a generalized
co-recursive polynomials by the following recurrence 
\begin{equation}
\begin{tabular}{l}
$P_{m}^{c}\left( x\right) =P_{m}\left( x\right) ,\quad m\leq k,\quad $with$
\quad P_{m}^{c}\left( x\right) \equiv 0,$ $m<0,$\vspace{0.2cm} \\ 
$P_{m}^{c}\left( x\right) =\left( x-\alpha _{m-1}\right) P_{m-1}^{c}\left(
x\right) -\sum\nolimits_{\nu=0}^{d-1-\nu}\xi _{m-1-\nu}^{d-1-\nu}P_{m-2-\nu}^{c}\left(
x\right) ,\quad k+1\leq m\leq d+k$,
\end{tabular}
\label{P14}
\end{equation}%
\begin{equation}
P_{m+d+1}^{c}\left( x\right) =\left( x-\beta _{m+d}\right) P_{m+d}^{c}\left(
x\right) -\sum\nolimits_{\nu=0}^{d-1}\gamma
_{m+d-\nu}^{d-1-\nu}P_{m+d-1-\nu}^{c}\left( x\right) ,\ \ m\geq k.  \label{P15}
\end{equation}

\begin{proposition}
	\label{TP3} The general solution of the recurrence (\ref{P14})-(\ref{P15}),
	could be given as 
	\begin{equation}
	\begin{array}{cl}
	P_{n}^{c}\left( x\right) & =P_{n}\left( x\right)
	-\sum\limits_{i=1}^{d}A_{i}\left( x\right) P_{n-k-i}^{\left( k+i\right)
	}\left( x\right) ,\quad d\geq 1,\quad n,k\geq 0, \\ 
	P_{n}^{c}\left( x\right) & =P_{n}\left( x\right) ,\quad n\leq k.
	\end{array}
	\label{P17}
	\end{equation}
\end{proposition}

\begin{proof}
We shall use the initial conditions (\ref{P11})-(\ref{P13}) to determine
the coefficients $\left\{ A_{k}\right\} _{k=1}^{d}$ explicitly. 
Indeed, for $m=k+1$, we get 
\begin{equation*}
\begin{tabular}{ll}
$P_{k+1}^{c}$ & $=\left( x-\beta _{k}-\mu _{k}\right) P_{k}^{c}-\left(
\gamma _{k}^{d-1}+\eta _{k}^{d-1}\right) P_{k-1}^{c}-...-\left( \gamma
_{1}^{d-k}+\eta _{1}^{d-k}\right) P_{0}^{c}\vspace{0.2cm}$ \\ 
& $=P_{k+1}-\mu _{k}P_{k}^{c}-\eta _{k}^{d-1}P_{k-1}^{c}-...-\eta
_{1}^{d-k}P_{0}^{c}\vspace{0.2cm}$ \\ 
& $=P_{k+1}-A_{1}$
\end{tabular}
\end{equation*}
i.e.,
\begin{equation*}
A_{1}=\mu _{k}P_{k}+\eta _{k}^{d-1}P_{k-1}+...+\eta _{1}^{d-k}P_{0}.
\label{A1}
\end{equation*}

By induction on $m$, we obtain for $m=k+d$, 
\begin{equation*}
\begin{tabular}{ll}
$P_{d+k}^{c}$ & $=\left( x-\beta _{d+k-1}-\mu _{d+k-1}\right)
P_{d+k-1}^{c}-\left( \gamma _{d+k-1}^{d-1}+\eta _{d+k-1}^{d-1}\right)
P_{d+k-2}^{c}-...\vspace{0.2cm}$ \\ & $
-\left( \gamma _{k+1}^{1}+\eta _{k+1}^{1}\right) P_{k}^{c}-\gamma
_{k}^{0}P_{k-1}^{c},$%
\end{tabular}%
\end{equation*}
then using the expansion (\ref{P17}), we get 
\begin{equation*}
\begin{tabular}{ll}
$P_{d+k}^{c}$ & $=\left( x-\beta _{d+k-1}-\mu _{d+k-1}\right) \left[
P_{d+k-1}-A_{1}P_{d-2}^{\left( k+1\right) }-...-A_{d-1}P_{0}^{\left(
	d+k-1\right) }\right] \vspace{0.2cm}$ \\ 
& $-\left( \gamma _{d+k-1}^{d-1}+\eta _{d+k-1}^{d-1}\right) \left[
P_{d+k-2}-A_{1}P_{d-3}^{\left( k+1\right) }-...-A_{d-2}P_{0}^{\left(
	d+k-2\right) }\right] -...\vspace{0.2cm}$ \\ 
& $-\left( \gamma _{k+2}^2+\eta _{k+2}^2\right) \left[ P_{k+1}-A_{1}P_{0}^{
	\left( k+1\right) }\right] -\left( \gamma _{k+1}^1+\eta _{k+1}^1\right)
P_{k}-\gamma _{k}^0P_{k-1}\vspace{0.2cm}$ \\ 
& $=P_{d+k}-A_{1}P_{d-1}^{\left( k+1\right) }-...-A_{d-1}P_{1}^{\left(
	d+k-1\right) }-A_{d}P_{0}^{\left( d+k\right) }$,
\end{tabular}
\end{equation*}
whence, finally
\begin{equation*}
\begin{array}{cl}
A_{d} & =\mu _{d+k-1}\left[ P_{d+k-1}-A_{1}P_{d-2}^{\left( k+1\right)
}-...-A_{d-1}P_{0}^{\left( d+k-1\right) }\right] \vspace{0.2cm} \\ 
& +\eta _{d+k-1}^{d-1}\left[ P_{d+k-2}-A_{1}P_{d-3}^{\left( k+1\right)
}-...-A_{d-2}P_{0}^{\left( d+k-2\right) }\right] \vspace{0.2cm} \\ 
& +...+\eta _{k+2}^2\left[ P_{k+1}-A_{1}P_{0}^{\left( k+1\right) }\right]
+\eta _{k+1}^1P_{k}.
\end{array}
\label{Ad}
\end{equation*}
\end{proof}

Notice that when $k=0$, we find the definition of the co-recursive polynomials 
\cite{Maroni2,Sadek1,Sadek4}, and when $d=1$ the above results reduce to
those analyzed in \cite{ChiharaCo,Dini,Foup3,Foup4,MarcPert,Ronv90,Slim}, among
others.

Co-recursive polynomials are proved to be useful in many cases. 
In fact, the regularity of linear combination of $d$-orthogonal polynomials 
treated recently in \cite{Sadek4} are given in terms of co-recursive polynomials 
for the inverse problem, and with the help of associated polynomials for the direct problem. 
We shall show for convenience some properties of their dual sequences.

Therefore, as it was remarked in \cite{Sadek1}, we again have the following 
\begin{equation*}
\left( P_{n}^{c}\right) ^{\left( d+k\right) }\left( x\right) =P_{n}^{\left(
	d+k\right) }\left( x\right) ,\quad d\geq 1,\quad n,k\geq 0.
\end{equation*}

That is, if we denote the dual sequence of co-recursive polynomials by 
$\left\{ \mathcal{L}_{n}\right\} _{n\geq 0}$, then there exist constants 
$c _{\nu }$ such that 
\begin{equation*}
\mathcal{L}_{\nu }^{\left( d+k\right) }=c _{\nu }u_{\nu }^{\left(
	d+k\right) }.
\end{equation*}

In order to determine the parameters $c _{\nu }$, we first prove the
following lemma.

\begin{lemma}
	\label{TP4} We have the following equalities
	\begin{equation*}
	\begin{array}{rl}
	\left\langle u_{\nu }^{\left( 1\right) },P_{\nu +2+i}\right\rangle & =\gamma
	_{\nu +2}^{d-1-i},\ \ \ 0\leq \nu ,i\leq d-1,\vspace{0.2cm} \\ 
	\left\langle u_{\nu }^{\left( r\right) },P_{r(d+1)+\nu }\right\rangle & 
	=\prod\nolimits_{i=1}^{r}\gamma _{d(r-\nu )+\nu +r+1}^{0},\ \ 0\leq \nu \leq
	d-1,\ \ r\geq 1.%
	\end{array}%
	\end{equation*}
\end{lemma}

\begin{proof}
	From (\ref{P9}) we have 
	\begin{equation*}
	\left\langle u_{\nu }^{\left( 1\right) },P_{\nu +2+i}\right\rangle
	=\left\langle u_{\nu +1}u_{0}^{-1},xP_{\nu +2+i}\right\rangle =\left\langle
	u_{\nu +1},xP_{\nu +2+i}\right\rangle .
	\end{equation*}
	Next, relations (\ref{CO1}) show the first equality. For the second equality, one
	can use (\ref{P8}). Indeed, by taking 
	$m+d+1=r(d+1)+\nu $ in the recurrence (\ref{B6}), we get using also the
	definition of dual sequence that 
	\begin{equation*}
	\left\langle u_{\nu }^{\left( r\right) },P_{dr+r+\nu }\right\rangle
	=\left\langle u_{\nu +r},x^{r}P_{dr+r+\nu }\right\rangle =\gamma
	_{d(r-1)+\nu +r+1}^{0}\left\langle u_{\nu +r},x^{r-1}P_{d(r-1)+r+\nu
	}\right\rangle .
	\end{equation*}
\end{proof}

Thus, this  lemma shows that $c _{\nu }=1,\ \forall \nu \geq 0$ because of the
following 
\begin{equation*}
\left\langle \mathcal{L}_{\nu }^{\left( d+k\right) },P_{(d+k)(d+1)+\nu
}\right\rangle =c _{\nu }\left\langle u_{\nu
	+d+k},x^{d+k}P_{(d+k)(d+1)+\nu }\right\rangle .
\end{equation*}
It is followed by proposition \ref{TP2} that 
\begin{equation*}
\left( x\mathcal{L}_{\nu }\right) \mathcal{L}_{d+k-1}^{-1}=\left( xu_{\nu
}\right) u_{d+k-1}^{-1},\ \ \nu \geq d,\ \ k\geq 0.
\end{equation*}

\subsubsection{Co-Dilated polynomials}

Generally, we study now further finite modifications of all the 
recurrence coefficients and we start with the so-called co-dilated and then with the co-modified polynomials \cite{Dini,MarcPert,Ronv90}. 
The main idea consists of modifying the recurrence coefficients and keeping unchanged
the regularity conditions in order to preserve the orthogonality according to Favard's theorem.
In doing so, next we shall multiply 
the last terms ($\gamma _{n}^{0}$) in the recurrence by a non-zero complex parameter $\lambda $.
We emphasize that in the previous modification, the regularity is well satisfied.

As customary, the co-dilated of a $d$-OPS $\left\{
P_{n}\right\} _{n}$ denoted by $\left\{ \tilde{P}_{n}\right\} _{n}$, is the
family of polynomials generated by the recurrence formula (\ref{B6}) in
which $\gamma _{1}^{0}$ is replaced by $\lambda \gamma _{1}^{0}$. Hence, by
regarding the initial conditions, this is equivalent to the following
recurrence
\begin{equation}
\begin{array}{cl}
\tilde{P}_{n} & =P_{n},\hspace{0.75cm}n\leq d,\vspace{0.2cm} \\ 
\tilde{P}_{d+1} & =\left( x-\beta _{d}\right) \tilde{P}_{d}-\sum\nolimits_{
	\nu =0}^{d-2}\gamma _{d-\nu }^{d-1-\nu }\tilde{P}_{d-1-\nu }-\lambda \gamma _{1}^{0},
\end{array}
\label{P19}
\end{equation}
\begin{equation*}
\tilde{P}_{n+d+1}=\left( x-\beta _{n+d}\right) \tilde{P}_{n+d}-
\sum\nolimits_{\nu =0}^{d-1}\gamma _{n+d-\nu }^{d-1-\nu }\tilde{P}_{n+d-1-\nu},\quad n\geq 1.
\end{equation*}

Using the initial conditions (\ref{P19}), the general solution of this
recurrence can be written as 
\begin{equation}
\tilde{P}_{n}\left( x\right) =P_{n}\left( x\right) +\gamma _{1}^{0}\left(
1-\lambda \right) P_{n-(d+1)}^{\left( d+1\right) }\left( x\right) ,\quad
d\geq 1,\quad n\geq 0.  \label{P21}
\end{equation}

Remark that we have again $\tilde{P}_{n}^{\left( 1\right) }=P_{n}^{\left(
	1\right) }$, $n\geq 0$. So, if we denote by $\widetilde{\mathcal{U}}
=\left( \tilde{u}_{0},...,\tilde{u}_{d-1}\right) $ the dual sequence of co-dilated sequence of polynomials, then there exist $\omega _{\nu
}$ constants such that 
\begin{equation}
\left( \tilde{u}_{\nu }\right) ^{\left( 1\right) }:=\tilde{u}_{\nu }^{\left(
	1\right) }=\omega _{\nu }u_{\nu }^{\left( 1\right) }.  \label{P22}
\end{equation}

Since the co-dilated are obtained under the modification of only one
parameter, then from lemma \ref{TP4} we see that 
\begin{equation}
\omega _{\nu }=1\text{\ \ for\ \ }0\leq \nu \leq d-1.  \label{P23}
\end{equation}

Apparently, we have the following result

\begin{proposition}
	\label{TP5} Let $\{P_{n}\}_{n\geq 0}$ be $d$-OPS with respect to 
	$\mathcal{U}=\left( u_{0},...,u_{d-1}\right) $, then the co-dilated sequence 
	$\{\tilde{P}_{n}\}_{n\geq 0}$ is $d$-OPS with respect to 
	$\widetilde{\mathcal{U}}=\left( \tilde{u}_{0},...,\tilde{u}_{d-1}\right) $ satisfying 
	\begin{equation}
	\tilde{u}_{\nu }=u_{\nu }\left[ \lambda \delta +\bar{\lambda}\left( 1-\beta
	_{0}x^{-1}\right) u_{0}-\bar{\lambda}\sum\nolimits_{\nu =0}^{d-2}\gamma
	_{1}^{d-1-\nu}\left( x^{-1}u_{\nu +1}\right) \right] ^{-1}  \label{D1}
	\end{equation}
	for $0\leq \nu \leq d-1,\quad d\geq 1$ where $\lambda +\bar{\lambda}=1$.
\end{proposition}

\begin{proof}
	From (\ref{P22})-(\ref{P23}), we deduce by using the second equality in (\ref{A3}), that 
	\begin{equation}
	\tilde{u}_{\nu +1}\tilde{u}_{0}^{-1}=u_{\nu +1}u_{0}^{-1},\quad 0\leq \nu
	\leq d-2,\quad d\geq 2,  \label{D2}
	\end{equation}
	
	together with the third equality in (\ref{A3}), we obtain 
	\begin{equation*}
	x^{2}\left[ \tilde{u}_{0}^{-1}-\lambda u_{0}^{-1}\right] +\bar{\lambda}
	\sum\nolimits_{\nu =0}^{d-2}\gamma _{1}^{d-1-\nu}x\left( u_{\nu
		+1}u_{0}^{-1}\right) =0.
	\end{equation*}
	
	The left product of a regular form by polynomial and \cite[lem. 24]{Sadek1} give 
	\begin{equation*}
	\tilde{u}_{0}^{-1}=\lambda u_{0}^{-1}+\bar{\lambda}\delta +\beta _{0}\bar{
		\lambda}\delta ^{\prime }-\bar{\lambda}\sum\nolimits_{\nu =0}^{d-2}\gamma
	_{1}^{d-1-\nu}x^{-1}\left( u_{\nu +1}u_{0}^{-1}\right) ,
	\end{equation*}
	which can be written using the fact $\delta ^{\prime }u=-x^{-1}u$ \cite[(1.16)]{MaroniAlg} in the following form 
	\begin{equation}
	\tilde{u}_{0}=u_{0}\left[ \lambda \delta +\bar{\lambda}\left( 1-\beta
	_{0}x^{-1}\right) u_{0}-\bar{\lambda}\sum\nolimits_{\nu =0}^{d-2}\gamma
	_{1}^{d-1-\nu}\left( x^{-1}u_{\nu +1}\right) \right] ^{-1}.  \label{D4}
	\end{equation}
	
	Replace (\ref{D4}) in (\ref{D2}) to obtain the desired result.
\end{proof}

In terms of the Stieltjes function, we obtain straightforwardly 

\begin{corollary}
	\label{TP6} The Stieltjes function of the co-dilated sequence satisfies 
	\begin{equation*}
	S\left( \tilde{u}_{\nu }\right) \left( z\right) =\frac{S\left( u_{\nu
		}\right) \left( z\right) }{\lambda -\bar{\lambda}P_{1}\left( z\right)
		S\left( u_{0}\right) \left( z\right) +\bar{\lambda}\sum\nolimits_{\nu
			=0}^{d-2}\gamma _{1}^{d-1-\nu }S\left( u_{\nu +1}\right) \left( z\right) },
	\end{equation*}
	for $0\leq \nu \leq d-1,\quad d\geq 1$.
\end{corollary}

See \cite{Dini} for a similar result in the case $d=1$. 
The extension of co-dilated at level $k\geq 1$ was introduced in 
\cite{MarcPert} for $d=1$. In case of the $d$-orthogonality, we can
multiply the constant $\gamma _{k}^{0}$ by a nonzero complex number $\lambda 
$. In this case, new family is defined by the following recurrence 
\begin{equation}
\begin{array}{rl}
\tilde{P}_{n} & =P_{n},\hspace{0.75cm}n\leq d+k-1,\vspace{0.2cm} \\ 
\tilde{P}_{d+k} & =\left( x-\beta _{d+k-1}\right) \tilde{P}
_{d+k-1}-\sum\nolimits_{\nu =0}^{d-2}\gamma _{d+k-1-\nu }^{d-1-\nu }\tilde{P}
_{d+k-2-\nu }-\lambda \gamma _{k}^{0}\tilde{P}_{k-1},
\end{array}
\label{D5}
\end{equation}
\begin{equation*}
\tilde{P}_{n+d+1}=\left( x-\beta _{n+d}\right) \tilde{P}_{n+d}-
\sum\nolimits_{\nu =0}^{d-1}\gamma _{n+d-\nu }^{d-1-\nu }\tilde{P}_{n+d-1-\nu},\quad n\geq k+1.
\end{equation*}

Using initial conditions (\ref{D5}), the general solution of the latter recurrence could be written as 
\begin{equation*}
\tilde{P}_{n}\left( x\right) =P_{n}\left( x\right) +\gamma _{k}^{0}\bar{
	\lambda}P_{k-1}\left( x\right) P_{n-(d+k)}^{\left( d+k\right) }\left(
x\right) ,\quad d\geq 1,\quad n\geq 0.
\end{equation*}

Remark further that we have $\tilde{P}_{n}^{\left( k\right) }=P_{n}^{\left(
	k\right) }$ for $n\geq 0$, $k\geq 1$, and then again by lemma \ref{TP4} 
\begin{equation*}
\left( \tilde{u}_{\nu }\right) ^{\left( k\right) }:=\tilde{u}_{\nu }^{\left(
	k\right) }=u_{\nu }^{\left( k\right) }\text{,\ }0\leq \nu \leq d-1.
\end{equation*}

Therefore, proposition \ref{TP2} and the left product of a form by
polynomial, show that 
\begin{equation*}
\tilde{u}_{k-1}\tilde{u}_{k+\nu }^{-1}=u_{k-1}u_{k+\nu }^{-1},\text{\ }0\leq
\nu \leq d-1\text{.}
\end{equation*}

\subsubsection{Co-Modified polynomials}

Now, combining the results of co-recursive and co-dilated, a new family
of polynomials might be generated by modifying the recurrence coefficients all together
\cite{Dini,Foup3,Foup4,MarcPert,Ronv90}. The new family obtained, denoted by 
$\left\{ \check{P}_{n}\right\} _{n\geq 0}$, called co-modified sequence, and
it is generated by the following recurrence relation 
\begin{equation*}
\begin{tabular}{rllll}
$\check{P}_{0}\left( x\right) $ & $=1,$ & \hspace{0.75cm} & $\check{P}
_{1}\left( x\right) $ & $=x-\beta _{0}-\mu _{0},\vspace{0.2cm}$ \\
$\check{P}_{m}\left( x\right) $ & \multicolumn{4}{l}{$=\left( x-\alpha
	_{m-1}\right) \check{P}_{m-1}\left( x\right) -\sum\nolimits_{\nu
		=0}^{m-2}\xi _{m-1-\nu }^{d-1-\nu }\check{P}_{m-2-\nu }\left( x\right)
	,\quad 2\leq m\leq d,\vspace{0.2cm}$} \\ 
$\check{P}_{d+1}\left( x\right) $ & \multicolumn{4}{l}{$=\left( x-\beta
	_{d}\right) \check{P}_{d}\left( x\right) -\sum\nolimits_{\nu =0}^{d-2}\gamma
	_{d-\nu }^{d-1-\nu }\check{P}_{d-1-\nu }\left( x\right) -\lambda \gamma
	_{1}^{0},\vspace{0.2cm}$} \\ 
$\check{P}_{m+d+1}\left( x\right) $ & \multicolumn{4}{l}{$=\left( x-\beta
	_{m+d}\right) \check{P}_{m+d}\left( x\right) -\sum\nolimits_{\nu
		=0}^{d-1}\gamma _{m+d-\nu }^{d-1-\nu }\check{P}_{m+d-1-\nu }\left( x\right)
	,\ \ m\geq 1,$}
\end{tabular}
\end{equation*}
where $\alpha _{n}$ and $\xi _{n}^{\nu },$ $1\leq \nu \leq d-1$ are given by (\ref{P13}).

In general framework, for $k\geq 0$ we could define the co-modified sequence by 
$\check{P}_{m}\left( x\right) =P_{m}\left( x\right) ,$ for $m\leq k$ and for 
$m>k$ by the following
\begin{equation*}
\begin{array}{rl}
\check{P}_{m}\left( x\right) & =\left( x-\alpha _{m-1}\right) \check{P}
_{m-1}\left( x\right) -\sum\nolimits_{\nu =0}^{m-2}\xi _{m-1-\nu }^{d-1-\nu }
\check{P}_{m-2-\nu }\left( x\right) ,\quad k+1\leq m\leq d+k,\vspace{0.2cm}
\\ 
\check{P}_{d+k+1}\left( x\right) & =\left( x-\beta _{d+k}\right) \check{P}
_{d+k}-\sum\nolimits_{\nu =0}^{d-2}\gamma _{d+k-\nu }^{d-1-\nu }\check{P}
_{d+k-1-\nu }-\lambda \gamma _{k+1}^{0}\check{P}_{k},
\end{array}
\end{equation*}
\begin{equation*}
\check{P}_{m+d+1}\left( x\right) =\left( x-\beta _{m+d}\right) \check{P}
_{m+d}\left( x\right) -\sum\nolimits_{\nu =0}^{d-1}\gamma _{m+d-\nu
}^{d-1-\nu }\check{P}_{m+d-1-\nu }\left( x\right) ,\ \ m\geq k+1.
\end{equation*}

From the previous results, the general solution of this recurrence, connects
all the above modified sequences through 
\begin{equation*}
\begin{array}{cl}
\check{P}_{n}\left( x\right) & =P_{n}\left( x\right)
-\sum\limits_{i=1}^{d}A_{i}\left( x\right) P_{n-k-i}^{\left( k+i\right)
}\left( x\right) +\gamma _{k+1}^{0}\bar{\lambda}P_{k}(x)P_{n-(d+k+1)}^{
	\left( d+k+1\right) }\left( x\right) ,\ d\geq 1,\ n,k\geq 0 \\ 
\check{P}_{n}\left( x\right) & =P_{n}\left( x\right) ,\quad n\leq k.
\end{array}
\end{equation*}
and also we have 
\begin{equation*}
\check{P}_{n}^{\left( d+k\right) }\left( x\right) =P_{n}^{\left( d+k\right)
}\left( x\right) ,\quad d\geq 1,\quad n,k\geq 0.
\end{equation*}

That is, if we denote the dual sequence of co-modified polynomials by 
$\left\{ \check{u}_{n}\right\} _{n\geq 0}$, then lemma \ref{TP2} provides
that necessarily we have 
\begin{equation*}
\check{u}_{\nu }^{\left( d+k\right) }=u_{\nu }^{\left( d+k\right) }.
\end{equation*}

In other words, proposition \ref{TP2} shows that 
\begin{equation*}
\left( x\check{u}_{\nu }\right) \check{u}_{d+k-1}^{-1}=\left( xu_{\nu
}\right) u_{d+k-1}^{-1},\ \ \nu \geq d.
\end{equation*}
We could give formally further connections between $\left\{ \check{u}_{n}\right\} _{\nu\geq 0}$ 
and $\left\{ u_{\nu}\right\}_{\nu\geq 0}$ similar to that in proposition \ref{TP5} and its corollary (see \cite{Dini} for the standard orthogonality, i.e. d=1).

\section{Determinants with co-polynomials entries}\label{sec:determ}

In this section we shed a spotlight on the theory of determinants whose
entries are $d$-OPS, and we give mild generalization as
well as few identities that characterize some Casorati determinants related to co-polynomials (discussed in \autoref{sec:3}).
That is to say, we try to give a $d$-analogue of some well known properties related to associated and co-polynomials in the standard orthogonality. To begin with, we have the following formula which leads to formula (\ref{B39}). The latter one plays in turn, a pivotal role in proving almost all the results of this section. 
It is also worthy to mention another motivation of the formula in question, that it was already used in the study of the regularity of two terms linear combinations of $d$-orthogonal polynomials \cite{Sadek4}

\begin{proposition}
	\label{S1} We have for any $d$-OPS $\left\{ P_{n}\right\} _{n}$  the following expansion
	\begin{equation}
	\begin{array}{cl}
	P_{n+m}^{\left( r\right) } & 
	=P_{m}^{\left( n+r\right) }P_{n}^{\left(
		r\right) }-\left( \sum_{i=1}^{d}\gamma _{n+r}^{d-i}P_{m-i}^{\left(
		n+r+i\right) }\right) P_{n-1}^{\left( r\right) }\vspace{0.2cm} \\ 	& 
	-\left( \sum_{i=1}^{d-1}\gamma _{n+r-1}^{d-1-i}P_{m-i}^{\left(
		n+r+i\right) }\right) P_{n-2}^{\left( r\right) }-...-\gamma
	_{n+r-d+1}^{0}P_{m-1}^{\left( n+r+1\right) }P_{n-d}^{\left( r\right) }.
	\end{array}
	\label{B36}
	\end{equation}
\end{proposition}

\begin{proof}
	The above identity can be easily proved by mathematical induction. Indeed,
	the equality is satisfied for $m=0$ and $m=1$ ($\forall n,r\geq 0$). Assume
	that it is true up to a fixed $m$. Then, by replacing (\ref{B36}) in the
	recurrence of $P_{n+m+1}^{\left( r\right) }$, and after getting 
	$P_{n-i}^{\left( r\right) }$, $0\leq i\leq d$, as a factor in the obtained
	expression, the result follows.
\end{proof}

By interchanging the role of $n$ and $m$ we find
\begin{equation}
\begin{array}{cl}
P_{n+m}^{\left( r\right) } & =P_{n}^{\left( m+r\right) }P_{m}^{\left(
	r\right) }-\left( \sum_{i=1}^{d}\gamma _{m+r}^{d-i}P_{n-i}^{\left(
	m+r+i\right) }\right) P_{m-1}^{\left( r\right) }\vspace{0.2cm} \\ 
& -\left( \sum_{i=1}^{d-1}\gamma _{m+r-1}^{d-1-i}P_{n-i}^{\left(
	m+r+i\right) }\right) P_{m-2}^{\left( r\right) }-...-\gamma
_{m+r-d+1}^{0}P_{n-1}^{\left( m+r+1\right) }P_{m-d}^{\left( r\right) }.
\end{array}
\label{B38}
\end{equation}

The latter relation gives a link between polynomials of $d+2$ levels of
association. By setting $m=1$ we get
\begin{equation}
P_{n+1}^{\left( r\right) }(x)=\left( x-\beta _{r}\right) P_{n}^{\left(
	r+1\right) }(x)-\sum\nolimits_{i=1}^{d}\gamma _{r+1}^{d-i}P_{n-i}^{\left(
	r+1+i\right) }(x),  \label{B39}
\end{equation}%
which is a dual formula of (\ref{P6}). When $d=1$, we obtain the result of
Belmehdi and Van Assche (see also \cite{Lopez1,Lopez2,Dickinson}).

As a consequence of the relation (\ref{B39}) is the following. Take 
$P_{n-k}^{\left( k\right) }$ in the place of $P_{n+1}^{\left( r\right) }$ and
expand the polynomials $P_{n-\left( k+1\right) }^{\left( k+1\right) }$ by
means of (\ref{B39}) to get
\begin{equation*}
\begin{array}{cl}
P_{n-k}^{\left( k\right) } & =P_{2}^{\left( k\right) }P_{n-\left( k+2\right)
}^{\left( k+2\right) }-\left[ \gamma _{k+2}^{d-1}P_{1}^{\left( k\right)
}+\gamma _{k+1}^{d-2}\right] P_{n-\left( k+3\right) }^{\left( k+3\right)
}-...\vspace{0.15cm} \\ & 
-\left[ \gamma _{k+2}^{1}P_{1}^{\left( k\right) }+\gamma _{k+1}^{0}\right]
P_{n-\left( k+d+1\right) }^{\left( k+d+1\right) }-\gamma
_{k+2}^{0}P_{1}^{\left( k\right) }P_{n-\left( k+d+1\right) }^{\left(
	k+d+1\right) }.
\end{array}
\end{equation*}
Proceeding in the same way $r$ times we obtain the following expression
\begin{equation}
P_{n-k}^{\left( k\right) }=P_{r}^{\left( k\right) }P_{n-\left( k+r\right)
}^{\left( k+r\right) }-q_{1,r-1}P_{n-\left( k+r+1\right) }^{\left(
	k+r+1\right) }-...-q_{d,r-1}P_{n-\left( k+r+d\right) }^{\left( k+r+d\right)
},  \label{B41}
\end{equation}%
where $q_{1,r-1}\ ...\ q_{d,r-1}\ $ are polynomials on $x$ of degree $r-1$. 

Further information on these latter polynomials are presented in the section \ref{Sub:3}.

Formula (\ref{B39}) constitutes the key ingredient in our approach of the present section.
Since each of the co-polynomials sequence discussed in  \autoref{sec:3} 
satisfies the same recurrence as $\left\{ P_{n}\right\} $ from certain level 
$k$, we could give an analogue of the expansion (\ref{B36}) as well as of 
(\ref{B38}) for any co-polynomials sequence 
by following a similar approach. For instance, $\left\{ P_{n}\right\} $ and
its corresponding co-recursive sequence $\left\{ Q_{n}\right\} $
satisfy the same recurrence relation for $n\geq d+1$, For this end, by
regarding the result in proposition \ref{S1}, we can easily obtain an
analogous expression for the co-recursive polynomials and the result is 
(which can be proved merely by mathematical induction) 
just replacing $P_{n}^{\left( r\right) }$ by $Q_{n}$ in (\ref{B36} ) for $n\geq
d+1$
\begin{equation*}
\begin{array}{cl}
Q_{n+m} & =P_{m}^{( n) }Q_{n}-\left( \sum_{i=1}^{d}\gamma _{n}^{d-i}
P_{m-i}^{( n+i)}\right) Q_{n-1}\vspace{0.15cm} \\ 
& -\left( \sum_{i=1}^{d-1}\gamma _{n-1}^{d-1-i}P_{m-i}^{( n+i) }\right)
Q_{n-2}-...-\gamma _{n-d+1}^{0}P_{m-1}^{( n+1)}Q_{n-d}.
\end{array}%
\end{equation*}

In what follows, some Casorati determinants are going to be presented according to our notation (\ref{NOT1}).
First, let us consider, for $n,r\geq 0$ and $d\geq 1$, the following
determinant 
\begin{equation}
B_{n}^{\left( r\right) }=\left\vert \mathbf{P}_{n,-1}^{r}\
\mathbf{P}_{n+1,-1}^{r}...\mathbf{P}_{n+d-1,-1}^{r}\right\vert ^{T}.  \label{B25}
\end{equation}

Now express each of the polynomials in the first column of the determinant (\ref{B25}) 
by means of the recurrence relation (\ref{B39}), 
then use the linearity of the determinant with respect to its first
column, it is not difficult to check that

\begin{proposition}
	\label{BT8} The determinant (\ref{B25}) satisfies the following linear
	recurrence%
	\begin{equation}
	\begin{array}{cl}
	B_{n}^{\left( r\right) } & =\left( -1\right) ^{d}\gamma
	_{r+1}^{1}B_{n-1}^{\left( r+1\right) }-\left( -1\right) ^{2\left( d-1\right)
	}\gamma _{r+1}^{0}\gamma _{r+2}^{2}B_{n-2}^{\left( r+2\right) }-...	\vspace{0.2cm} \\ 	&
	-\left( -1\right) ^{\left( d-1\right) \left( d-1\right) }\gamma
	_{r+1}^{0}\gamma _{r+2}^{0}...\gamma _{r+d-2}^{0}\gamma
	_{r+d-1}^{d-1}B_{n-\left( d-1\right) }^{\left( r+d-1\right) }\vspace{0.2cm}
	\\ 
	& +\left( -1\right) ^{d\left( d-1\right) }\left( x-\beta _{r+d-1}\right)
	\gamma _{r+1}^{0}...\gamma _{r+d-1}^{0}B_{n-d}^{\left( r+d\right) }	\vspace{0.2cm} \\ 	& 
	+\left( -1\right) ^{d.d}\gamma _{r+1}^{0}\gamma _{r+2}^{0}...\gamma
	_{r+d}^{0}B_{n-\left( d+1\right) }^{\left( r+d+1\right) }.
	\end{array}
	\label{BB}
	\end{equation}
\end{proposition}

Notice that if one expresses each of the polynomials in the first column in the
determinant (\ref{B25}) according to the recurrence relation (\ref{P6})
instead of the dual recurrence  (\ref{B39}), one gets a recurrence of order 
$\left(d+1\right) $ but for the same level of association $r$ in each row.
When $d=2 $, the recurrence (\ref{BB}) reduces to that of de Bruin \cite[Lem.1, p.372]{deBruin}.

Now if we add, for instance, the next row in bottom and the next column at
rightmost in the determinant $B_{n}$, we obtain a nonzero constant. This shows that
those families are linearly independent (see proposition \ref{PD1} below). 
Accordingly, let us consider the following Casorati determinants 
\begin{equation}
\begin{array}{cl}
\Delta _{n}^{\left( r\right) } & :=\left\vert \mathbb{P}_{n}^{r}\ 
\mathbb{P}_{n-1}^{r+1}...\mathbb{P}_{n-d}^{r+d}\right\vert =
\left\vert \mathbf{P}_{n}^{r}\ \mathbf{P}_{n+1}^{r}...\mathbf{P}_{n+d}^{r}\right\vert ^{T} ,
\vspace{.2cm} \\ 
\nabla _{n}^{\left( r\right) } & :=\left\vert \mathbb{Q}_{n}\ 
\mathbb{P}_{n-r}^{r}...\mathbb{P}_{n-r-d+1}^{r+d-1}\right\vert.
\end{array}
\label{B24}
\end{equation}
We are able to prove the following identity.

\begin{proposition}
	\label{PD1} \label{BT10} The determinant $\Delta _{n}^{\left( r\right) }$
	satisfies the following identity
	\begin{equation}
	\Delta _{n}^{\left( r\right) }=\left( -1\right) ^{\left( d+1\right)
		n}\prod\limits_{i=1}^{n}\gamma _{i+r}^{0} \quad \text{with} \quad \Delta
	_{0}^{\left( r\right) }=1.  \label{B27}
	\end{equation}
\end{proposition}

\begin{proof}
	The proof of (\ref{B27}) follows easily from the dual recurrence relation 
	(\ref{B39}). Indeed, we express each of the polynomials in the first column
	of determinant $\Delta _{n}^{\left( r\right) }$ by means of (\ref{B39})
	and using the linearity of the determinant with respect to its first column,
	we obtain
	\begin{equation*}
	\Delta _{n}^{\left( r\right) }=\left( -1\right) ^{d+1}\gamma
	_{r+1}^{0}\Delta _{n-1}^{\left( r+1\right) }.
	\end{equation*}
	
	Then, the result follows by induction on $n$.
\end{proof}

More generally, we shall prove the following result

\begin{theorem}
	\label{S2} For any integers $m_{1},...,m_{d}>n\geq 0$, and $r\geq 0$, we
	have 
	\begin{equation}
	F_{n}^{\left( r\right) }=\left\vert \mathbf{P}_{n}^{r}\ \mathbf{P}
	_{m_{1}}^{r}...\mathbf{P}_{m_{d}}^{r}\right\vert ^{T}=\Delta _{n}^{\left(
		r\right) }\left\vert \mathbf{P}_{m_{1}-n-1,-1}^{r+n+1}\ \mathbf{P}%
	_{m_{2}-n-1,-1}^{r+n+1}...\mathbf{P}_{m_{d}-n-d,-1}^{r+n+1}\right\vert ^{T}
	\label{B40}
	\end{equation}
\end{theorem}

\begin{proof}
	The equality (\ref{B40}) can be obtained in a similar manner as in the
	previous proposition, we express each of the polynomials in the first column
	using the dual recurrence (\ref{B39}), to get 
	\begin{equation*}
	F_{n}^{\left( r\right) }=\left( -1\right) ^{d+1}\gamma
	_{r+1}^{0}F_{n-1}^{\left( r+1\right) }.
	\end{equation*}
	
	Proceeding in a similar way $n$ times we have%
	\begin{equation*}
	F_{n}^{\left( r\right) }=\Delta _{n}^{\left( r\right) }F_{0}^{\left(
		r+n\right) },
	\end{equation*}
	which implies the required result.
\end{proof}

Having in mind that the first $d+1$ polynomials in association, i.e., the
polynomials $\{P_{n}^{\left( r\right) }\}_{0\leq r\leq d}$ are linearly
independent, then any determinant of type (\ref{B24}) of dimension $m\times
m$, with $m\geq d+1$, is identically zero.


It is worthwhile to notice that the results of this section can be proved in another way using companion matrices. That is to say, the recurrence relation (\ref{B6}) can be presented in terms of another matrix 
$C_{n}$ called companion or transfer matrix as $\mathbb{P}_{n+1}^{i}=C_{n}^{(i)}\mathbb{P}_{n}^{i}$ where 
$C_{n}^{(i)}$ is the matrix 
\begin{equation*}
C_{n}^{(i)}=\left( 
\begin{array}{cc}
\boldmath{0} & I_{d} \\ 
-V & x-\beta _{n+d+i}
\end{array}
\right) ,\hspace{0.5cm}V=\left( \gamma _{n+i+1}^{0},\gamma
_{n+i+2}^{1},...,\gamma _{n+d+i}^{d-1}\right) .
\end{equation*}

By virtue of our notation (\ref{NOT1}), let us introduce the following Casorati determinant 
\begin{equation*}
D_{n-i}^{(i)}=\left\vert \mathbb{P}_{n-i}^i\ \mathbb{P}_{n-r}^{r}... 
\mathbb{P}_{n-r-d+1}^{r+d-1}\right\vert, \ \ 0\leq i\leq r-1, \ r\geq 1.
\end{equation*}

In this case, since $\det(C_{n}^{(i)})=(-1)^{d+1}\gamma _{n+i+1}^0$, then we have 
\begin{equation*}
D_{n-i}^{(i)}=(-1)^{d+1}\gamma _{n}^0 D_{n-i-1}^{(i)}.
\end{equation*}
Therefore, we get by induction that	
\begin{proposition}
	\label{PD2} \label{BT12} For any $n\geq 0$, $r\geq 1$, and $0\leq i\leq r-1$, we have 
	\begin{equation*}
	D_{n-i}^{(i)}=(-1)^{(d+1)(n-r+1)}\prod\limits_{k=r}^{n}\gamma
	_{k}^{0}P_{r-i-1}^{(i)}(x)=\Delta _{n-r+1}^{(r-1)}P_{r-i-1}^{(i)}(x),
	\end{equation*}
\end{proposition}
which presents another proof of the proposition \ref{PD1} (see also \cite{Kalyagin94}).

By all means, this can be also checked by replacing each polynomials in
the first column of $D_n$ by the corresponding dual recurrence (\ref{B41}).
Similar identity for co-recursive polynomials can be obtained by replacing
the polynomial $P_{m}$ by $Q_{m}$.


However, this idea is not good enough to work with if one uses different
sequences because the vector $V$ is not the same. Nevertheless, it is more
convenient some times, to use the expansion (\ref{B39}) instead of the
recurrence (\ref{B6}). For instance, we could express and calculate 
Casorati determinant (\ref{B40}) using transfer matrix. We have 
\begin{equation*}
\mathbf{P}_{m_{i}}^{r}=\widetilde{C}_{r}\mathbf{P}_{m_{i}-1}^{r+1},
\end{equation*}%
where the matrix $\widetilde{C}$ is obtained from the matrix ${C}$ when the vector $V$ is replaced by the vector $\widetilde{V}=\left( \gamma
_{r+1}^{0},\gamma _{r+1}^{1},...,\gamma _{r+1}^{d-1}\right) $. Accordingly,
we get 
\begin{equation*}
F_{n}^{\left( r\right) }=\left( -1\right) ^{d+1}\gamma
_{r+1}^{0}F_{n-1}^{\left( r+1\right) }.
\end{equation*}

Thus our next task is to show analogous results for sequences obtained by
a finite modification in the recurrence coefficients. First, taking into
account the initial conditions (\ref{P14})-(\ref{P15}) and the corollary \ref{TP3}, 
we are able to prove the following identities satisfied by the
determinant $\nabla _{n}^{\left( r\right) }$.

\begin{proposition}
	\label{BT11} The determinant $\nabla _{n}^{\left( r\right) }$ satisfies 
	$\nabla _{n}^{\left( 0\right) }=\left( -1\right) ^{d+1}A_{d}\Delta
	_{n}^{\left( 0\right) }$ and for $r\geq 1$ the following identities 
	\begin{equation}
	\nabla _{n}^{\left( r\right) }=Q_{r-1}\Delta _{n-r+1}^{\left( r-1\right)
	}=\left\{ 
	\begin{array}{cc}
	\left[ P_{r-1}-A_{1}P_{r-2}^{\left( 1\right) }-...-A_{r-1}P_{0}^{\left(
		r-1\right) }\right] \Delta _{n-r+1}^{\left( r-1\right) } & \text{,\ if \ }
	r\leq d\vspace{0.2cm} \\ 
	\left[ P_{r-1}-A_{1}P_{r-2}^{\left( 1\right) }-...-A_{d}P_{r-d-1}^{\left(
		d\right) }\right] \Delta _{n-r+1}^{\left( r-1\right) } & \text{,\ if \ }r>d.
	\end{array}%
	\right.   \label{CO3}
	\end{equation}
\end{proposition}

\begin{proof}
	We express each of the polynomials in the first column of determinant 
	$\nabla _{n}^{\left( r\right) }$ according to the expansion (\ref{P17}), we get 
	\begin{equation*}
	\nabla _{n}^{\left( r\right) } =D_{n}^{(0)}-A_{1}D_{n-1}^{(1)} -...
	-A_{r-1}\Delta _{n-r+1}^{\left( r-1\right) }-...-A_{d}D_{n-d}^{(d)},
	\end{equation*}
	now proposition \ref{BT12} completes the proof.
\end{proof}

Analogous formulas for co-dilated and for co-modified are the
following. If we replace the vector $\mathbb{Q}_{n}$ in $\nabla _{n}^{\left(
	r\right) }$ at first time by $\mathbb{\tilde{P}}_{n}$ and by $\mathbb{\check{P}}_{n}$ in a second time, then the resulting determinants are denoted by $\tilde{\nabla}_{n}^{\left( r\right) }$ and $\check{\nabla}%
_{n}^{\left( r\right) }$, respectively, i.e. 
\begin{equation*}
\tilde{\nabla}_{n}^{\left( r\right) }=\left\vert \mathbb{\tilde{P}}_{n}\ 
\mathbb{P}_{n-r}^{r}...\mathbb{P}_{n-r-d+1}^{r+d-1}\right\vert,\qquad
\check{\nabla}_{n}^{\left( r\right) }=\left\vert \mathbb{\check{P}}_{n}\ 
\mathbb{P}_{n-r}^{r}...\mathbb{P}_{n-r-d+1}^{r+d-1}\right\vert
\end{equation*}
and as above, we have the following results.

\begin{proposition}
	Casorati determinants corresponding to the co-modified polynomials
	satisfy
	\begin{equation*}
	\check{\nabla}_{n}^{\left( r\right) }=\left\{ 
	\begin{array}{cl}
	\left( -1\right) ^{d+1}\left[ A_{d}+\bar{\lambda}\gamma _{1}^{1}\right]
	\Delta _{n}^{\left( 0\right) } & \text{,\ if \ }r=0,\vspace{0.2cm} \\ 
	\lambda \Delta _{n}^{\left( 0\right) } & \text{,\ if \ }r=1,\vspace{0.2cm}
	\\ 
	Q_{1}\Delta _{n-1}^{\left( 1\right) } & \text{,\ if \ }r=2,\vspace{0.2cm} \\ 
	\left[ Q_{r-1}+\bar{\lambda}\gamma _{1}^{0}P_{r-d-1}^{\left( d+1\right) }%
	\right] \Delta _{n-r+1}^{\left( r-1\right) } & \text{,\ if \ }r\geq 3.%
	\end{array}%
	\right. 
	\end{equation*}
	
	The determinants $\tilde{\nabla}_{n}^{\left( r\right) }$ could be obtained as
	a particular case from $\check{\nabla}_{n}^{\left( r\right) }$ by taking $%
	A_{i}\equiv 0$.
\end{proposition}

\begin{proof}
	Using the expression (\ref{P21}), then the determinant $\tilde{\nabla}%
	_{n}^{\left( r\right) }$ reads, when $r=0$%
	\begin{equation*}
	\tilde{\nabla}_{n}^{\left( 0\right) }=\left( -1\right) ^{d}\bar{\lambda}%
	\gamma _{1}^{0}\left\vert \mathbb{P}_{n}\ \mathbb{P}_{n-1}^{1}...\mathbb{P}%
	_{n-d+1}^{d-1}\mathbb{P}_{n-d-1}^{d+1}\right\vert .
	\end{equation*}
	
	Once again, express the polynomials in the first column by means of the dual
	recurrence (\ref{B39}) and use the fact that $\Delta _{n}^{\left( 0\right)
	}=\left( -1\right) ^{d+1}\gamma _{1}^{0}\Delta _{n-1}^{\left( 1\right) }$ to
	get the value of $\tilde{\nabla}_{n}^{\left( 0\right) }$. Then induction gives the result.
	
	For the determinant $\check{\nabla}_{n}^{\left( r\right) }$, it suffices to
	remark that we have the following relation between all the perturbed
	families and the original one 
	\begin{equation*}
	\check{P}_{n}=Q_{n}+\tilde{P}_{n}-P_{n},
	\end{equation*}%
	that is, in view of the above notation and results, 
	\begin{equation*}
	\check{\nabla}_{n}^{\left( r\right) }=\nabla _{n}^{\left( r\right) }+\tilde{%
		\nabla}_{n}^{\left( r\right) }-P_{r-1}\Delta _{n-r+1}^{\left( r-1\right) }.
	\end{equation*}
	
	This finishes the proof of the proposition.
\end{proof}

We can use the results of this section to give analogous results for the
modified sequences. Indeed, combine the proposition \ref{BT12} and \ref{BT11}, 
it is not difficult to notice that we again have

\begin{theorem}
	\label{S11} For any integers $m_{1},...,m_{d}>n\geq 0$, and $r\geq 1$, we have
	\begin{equation*}
	R_{n}^{\left( r\right) }= \left\vert 
	\begin{array}{llll}
	Q_{n} & Q_{m_1} & ... & Q_{m_d} \\ 
	\mathbf{P}_{n-r,-1}^{r} & \mathbf{P}_{m_1-r,-1}^{r} & ... & \mathbf{P}_{m_{d}-r,-1}^{r}
	\end{array}
	\right\vert ^{T} =\left\{ 
	\begin{array}{cl}
	\left( -1\right) ^{d+1}A_{d}F_{n}^{\left( 0\right) } & ,\ \ \text{if \ }r=0,
	\vspace{0.2cm} \\ 
	Q_{r-1}F_{n-r+1}^{\left( r-1\right) } & ,\ \ \text{if \ }r\geq 1.
	\end{array}
	\right. .
	\end{equation*}
\end{theorem}

By analogy, let us denote the first column of the determinant $R_{n}^{\left(
	r\right) }$ by $\mathbb{Q}$. If we replace this last vector by $\mathbb{%
	\check{P}}$, and denote the resulting determinant by $\check{R}_{n}^{\left(
	r\right) }$, we then obtain the following

\begin{corollary}
	New Casorati determinants corresponding to the co-dilated and the
	co-modified polynomials satisfy
	\begin{equation*}
	\check{R}_{n}^{\left( r\right) }=\left\{ 
	\begin{array}{cl}
	\left( -1\right) ^{d+1}\left[ A_{d}+\bar{\lambda}\gamma _{1}^{1}\right]
	F_{n}^{\left( 0\right) } & \text{,\ if \ }r=0,\vspace{0.2cm} \\ 
	\lambda F_{n}^{\left( 0\right) } & \text{,\ if \ }r=1,\vspace{0.2cm} \\ 
	Q_{1}F_{n-1}^{\left( 1\right) } & \text{,\ if \ }r=2,\vspace{0.2cm} \\ 
	\left[ Q_{r-1}+\bar{\lambda}\gamma _{1}^{0}P_{r-d-1}^{\left( d+1\right) }%
	\right] F_{n-r+1}^{\left( r-1\right) } & \text{,\ if \ }r\geq 3.%
	\end{array}%
	\right. 
	\end{equation*}
	
	The determinants $\tilde{R}_{n}^{\left( r\right) }$ are obtained from $\check{R%
	}_{n}^{\left( r\right) }$ by taking $A_{i}\equiv 0$.
\end{corollary}

Further generalization of $F_{n}^{\left( r\right) }$ are the following
determinants

\begin{corollary}
	\label{S7} The determinants $G_{n}$ satisfy the following recurrence%
	\begin{equation*}
	G_{d}\left( n\right) =\left\vert \mathbb{P}_{n-s_{0}}^{(s_{0})}\ \mathbb{P}%
	_{n-s_{1}}^{(s_{1})}...\mathbb{P}_{n-s_{d}}^{(s_{d})}\right\vert ^{T}=\left(
	-1\right) ^{d+1}\gamma _{n-d}^{0}G_{d}\left( n-1\right) .
	\end{equation*}
\end{corollary}

\subsection{A characterization of $d$-orthogonality}
\label{Sub:3}

Next we give a generalization of the characterization of the 
orthogonality pointed out by Al-Salam \cite{Al-Salam}. 
A motivation of this section comes from the number theory, mainly, when people wanted to expand three term recurrence relation in terms of three consecutive polynomials of fixed degree. 
We saw huge paper trying to prove a property that the polynomial coefficients  are as the same type as  the sequence is true for many polynomial families of classical sequence of numbers. In this section, we shall show that the above property is true for any sequence of $d$-orthogonal polynomials. 

From the general theory, 
for any linear recurrence relations of $d+1$ terms, there are $d$ linearly independent solutions, i.e. the Wronskian of these $d$ solutions is different from zero.
It follows then, that every solution is a linear combination of $d$ linearly independent solution.

We now set $d+1$ polynomials $S_{n}^{\left( i\right) },$ $1\leq i\leq d+1$
defined by the initial conditions $S_{k-1}^{k}\neq 0$ and $S_{n}^{k}=0$ when 
$n<k-1$ for $1\leq k\leq d+1$. This construction allows us to assert that the set 
$\left\{ S_{n}^{i},1\leq i\leq d+1\right\} _{n}$ forms a fundamental set of solutions of a $d+2$-term linear recurrence relation. Furthermore, we have the following results \cite{Hou}

\begin{lemma}
	\label{BT14} A necessary and sufficient condition that there exists a
	relation%
	\begin{equation*}
	\left\vert \mathbb{S}_{n}^{\left( 1\right) }\ \mathbb{S}_{n}^{\left(
		2\right) }...\mathbb{S}_{n}^{\left( d+1\right) }\right\vert =\Delta _{n}\neq
	0
	\end{equation*}%
	is that the sequence of polynomials $\left\{ S_{n}^{\left( i\right) },1\leq
	i\leq d+1\right\} _{n}$ are $d$-OPS.
\end{lemma}

Now suppose that $\left\{ f_{n}\right\} $ is $d$-OPS satisfying the
following recurrence relation%
\begin{equation*}
f_{n+d+1}\left( x\right) =\left( A_{n+d}x+B_{n+d}\right) f_{n+d}\left(
x\right) +\gamma _{n+d}^{d-1}\ f_{n+d-1}\left( x\right) +...+\gamma
_{n+1}^{0}\ f_{n}\left( x\right) .
\end{equation*}

It is evident from (\ref{B41}), that for each integer $p\geq 1$, we have 
\begin{equation}
f_{n+d+p}\left( x\right) =T_{p}^{\left( 1\right) }\left( x\right)
f_{n+d}\left( x\right) +...+T_{p}^{\left( d+1\right) }\left( x\right)
f_{n}\left( x\right) ,  \label{BC4}
\end{equation}%
where $T_{p}^{\left( i\right) }\left( x\right) \ $ are polynomials on $x$ of
degree $p$ and $p-1$ for $i=0$ and $2\leq i\leq d+1$, respectively, and where 
\begin{equation*}
\begin{array}{l}
T_{0}^{\left( 1\right) }=1,\text{\ \ }T_{0}^{\left( i\right) }\left(
x\right) =0,\ \ 1\leq i\leq d+1\vspace{0.15cm} \\ 
T_{1}^{\left( 1\right) }\left( x\right) =A_{n}x+B_{n},\ \ T_{1}^{\left(
	i\right) }\left( x\right) =\gamma _{n+d+2-i}^{d+1-i},\ \ 2\leq i\leq d+1,%
\vspace{0.15cm} \\ 
T_{2}^{\left( 1\right) }\left( x\right) =\left( A_{n+1}x+B_{n+1}\right)
\left( A_{n}x+B_{n}\right) +\gamma _{n+d+1}^{d-1},\vspace{0.15cm} \\ 
T_{2}^{\left( i\right) }\left( x\right) =\left( A_{n+1}x+B_{n+1}\right)
\gamma _{n+d+2-i}^{d+1-i}+\gamma _{n+d+3-i}^{d-i},\ \ 2\leq i\leq d+1,%
\vspace{0.15cm} \\ 
\vdots%
\end{array}%
\end{equation*}

Then we can prove the following

\begin{theorem}
	\label{BT15} The polynomials $T_{p}^{\left( i\right) }\left( x\right)$, $%
	1\leq i\leq d+1$, appeared in (\ref{BC4}) are also $d$-OPS. 
	Moreover, they satisfy the following identity 
	\begin{equation}
	\left\vert \mathbb{T}_{p}^{\left( 1\right) }\ \mathbb{T}_{p}^{\left(
		2\right) }...\mathbb{T}_{p}^{\left( d+1\right) }\right\vert ^{T}=\Delta
	_{n-d}^{-1}\Delta _{n+p}\neq 0.  \label{BC6}
	\end{equation}
\end{theorem}

\begin{proof}
	It follows from (\ref{BC4}), that 
	\begin{equation*}
	\left\vert \mathbb{S}_{n+p}^{\left( 1\right) }\ \mathbb{S}_{n+p}^{\left(
		2\right) }...\mathbb{S}_{n+p}^{\left( d+1\right) }\right\vert
	^{T}=\left\vert \mathbb{S}_{n-d}^{\left( 1\right) }\ \mathbb{S}%
	_{n-d}^{\left( 2\right) }...\mathbb{S}_{n-d}^{\left( d+1\right) }\right\vert
	^{T}\left\vert \mathbb{T}_{p}^{\left( 1\right) }\ \mathbb{T}_{p}^{\left(
		2\right) }...\mathbb{T}_{p}^{\left( d+1\right) }\right\vert ^{T}.
	\end{equation*}
	
	Hence%
	\begin{equation*}
	\Delta _{n-d}\left\vert \mathbb{T}_{p}^{\left( 1\right) }\ \mathbb{T}%
	_{p}^{\left( 2\right) }...\mathbb{T}_{p}^{\left( d+1\right) }\right\vert
	^{T}=\Delta _{n+p}.
	\end{equation*}
	
	Then, lemma \ref{BT14} completes the proof.
\end{proof}

Notice that for $d=1$, a sequence of quasi-orthogonal polynomials satisfies a three term recurrence relation with polynomial coefficients \cite{ChiharaQuasi}. In addition, any linear combination in $l$ terms ($l>d$) of a $d$-OPS with constant coefficients, could be expressed as a linear combination in terms of only $d+1$ terms with
polynomial coefficients. This attempt was treated by Joulak in \cite{Joulak}. 
For a given $d$-OPS $\left\{ P_{n}\right\} _{n\geq 0}$ defined by the recurrence (\ref{B6}), let us consider the following linear combination 
\begin{equation}
Q_{n}\left( x\right) =P_{n}\left( x\right) +a_{n}^{\left( 1\right)
}P_{n-1}\left( x\right) +...+a_{n}^{\left( r\right) }P_{n-r}\left( x\right)
,\ \ n\geq 1.  \label{J1}
\end{equation}

\begin{proposition}
	For $r>d$, the polynomials sequence $\left\{ Q_{n}\right\} $ 
	defined by (\ref{J1}) might be given in the following form
	\begin{equation}
	\begin{array}{cl}
	Q_{n}\left( x\right)  & =U_{r-1}P_{n-r+1}+\left[ a_{n}^{\left( r\right)
	}-\sum\nolimits_{i=0}^{d-1}\gamma _{n-r+1}^{d-1-i}U_{r-2-i}\right] P_{n-r}	\vspace{0.2cm} \\ 
	& -\left[ \sum\nolimits_{i=0}^{d-2}\gamma _{n-r}^{d-2-i}U_{r-2-i}\right]
	P_{n-r-1}-\left[ \sum\nolimits_{i=0}^{d-3}\gamma _{n-r-1}^{d-3-i}U_{r-2-i}
	\right] P_{n-r-2}\vspace{0.2cm} \\ 
	& -...-\gamma _{n-r-d+2}^{0}U_{r-2}P_{n-r-d+1},
	\end{array}\label{BT16}
	\end{equation}%
	where%
	\begin{equation*}
	U_{r}-a_{n}^{\left( r\right) }=\left( x-\beta _{n-r}\right)
	U_{r-1}-\sum\nolimits_{\nu =0}^{d-1}\gamma _{n-r+1}^{d-1-\nu }U_{r-2-i},
	\end{equation*}%
	with $U_{0}=1$ and $U_{-s}\equiv 0$ for $s\geq 1$.
\end{proposition}

Now as the converse of the Theorem \ref{BT15} we can say: 
In the latter expansion (\ref{BT16}) the sequence of polynomials $\left\{ U_{n}\right\} $ is
$d$-OPS iff $a_n^{(i)}\equiv0$, $1\leq i\leq r$. In this case,  
$\left\{ P_n\right\}=\left\{Q_{n}\right\}$ (quasi-orthogonal of order zero), 
i.e. $\left\{Q_{n}\right\}$ is also $d$-OPS 
and (\ref{BT16}) reduces again to  (\ref{BC4}).

\section{Christoffel-Darboux type formulas}\label{sec:ChrDarbForm}

Our next wishes are to give some formulas of Christoffel-Darboux type. Let us first
notice that, for the determinant $F_{n}^{\left( r\right) }$ (\ref{B40}) above, when $m_{i}=n+i$
for $1\leq i\leq d-1$, and if we replace $m_{d}$ by $m-i$, we readily get the
following identities%
\begin{equation*}
\left\vert \mathbf{P}_{n}^{r}...\mathbf{P}_{n+d-1}^{r}\mathbf{P}%
_{m-i}^{r}\right\vert ^{T}=\Delta _{n}^{\left( r\right) }P_{m-n-i-d}^{\left(
	n+r+d\right) },
\end{equation*}%
and in a similar way, also using the dual recurrence relation (\ref{B39}),
we have%
\begin{equation*}
\begin{array}{cl}
\left\vert \mathbf{P}_{n}^{r}...\mathbf{P}_{n+d-2}^{r}\mathbf{P}_{n+d}^{r}%
\mathbf{P}_{m-i}^{r}\right\vert ^{T} & =\Delta _{n}^{\left( r\right) }\left[
P_{1}^{\left( n+r+d-1\right) }P_{m-n-i-d}^{\left( n+r+d\right)
}-P_{m-n-i-d+1}^{\left( n+r+d-1\right) }\right] 
\vspace{0.2cm} \\ & 
=\Delta _{n}^{\left( r\right) }\left[ \sum\nolimits_{j=1}^{d}\gamma
_{r+n+d}^{d-j}P_{m-n-d-i-j}^{\left( n+r+d+j\right) }\right] .%
\end{array}%
\end{equation*}

Now we want to give a motivation of these latter identities. Especially when 
$d=2$ we can give further new type of Christoffel-Darboux formula. Indeed,
the next results, given by theorem \ref{S9} and corollary \ref{S10}, are
established for $d=2$. In this particular case, we are able to prove the
following formula

\begin{theorem}
	\label{S9} For any integers $k>m>n\geq 0$ and $r\geq 0$, we have%
	\begin{equation}
	\begin{array}{l}
	H_{m}:=\left\vert \mathbf{P}_{n}^{r}\mathbf{P}_{m}^{r}\mathbf{P}%
	_{k}^{r}\right\vert ^{T}=\left\vert 
	\begin{array}{ccc}
	P_{n}^{\left( r\right) } & P_{n-1}^{\left( r+1\right) } & P_{n-2}^{\left(
		r+2\right) }\vspace{0.15cm} \\ 
	P_{m}^{\left( r\right) } & P_{m-1}^{\left( r+1\right) } & P_{m-2}^{\left(
		r+2\right) }\vspace{0.15cm} \\ 
	P_{k}^{\left( r\right) } & P_{k-1}^{\left( r+1\right) } & P_{k-2}^{\left(
		r+2\right) }%
	\end{array}%
	\right\vert \vspace{0.15cm} \\ 
	=\Delta _{n}^{\left( r\right) }\left[ P_{k-m-1}^{\left( m+r+1\right)
	}\left\vert 
	\begin{array}{cc}
	P_{m-n-1}^{\left( r+n+1\right) } & P_{m-n-2}^{\left( r+n+2\right) }\vspace{0.15cm} \\ 
	P_{m-n}^{\left( r+n+1\right) } & P_{m-n-1}^{\left( r+n+2\right) }%
	\end{array}%
	\right\vert\right. 	
	\left.	+\gamma _{m+r}^{0}P_{k-m-2}^{\left( m+r+2\right) }\left\vert 
	\begin{array}{cc}
	P_{m-n-2}^{\left( r+n+1\right) } & P_{m-n-3}^{\left( r+n+2\right) }\vspace{%
		0.15cm} \\ 
	P_{m-n-1}^{\left( r+n+1\right) } & P_{m-n-2}^{\left( r+n+2\right) }%
	\end{array}%
	\right\vert \right] 
	\end{array}
	\label{B50}
	\end{equation}
\end{theorem}

\begin{proof}
	When $d=2$, since $k>m>n$, then expanding $P_{k-i}^{\left( r+i\right) }$ and $P_{m-i}^{\left( r+i\right) }$ for $i=0,1,2$, by means of
	proposition \ref{S1} in the following way%
	\begin{equation}
	\begin{array}{cl}
	P_{k-i}^{\left( r+i\right) } & =P_{k-m}^{\left( m+r\right) }P_{m-i}^{\left(
		r+i\right) }-\gamma _{m+r-1}^{0}P_{k-m-1}^{\left( m+r+1\right)
	}P_{m-i-2}^{\left( r+i\right) }\vspace{0.15cm} \\ 	& 
	-\left( \gamma _{m+r}^{1}P_{k-m-1}^{\left( m+r+1\right) }+\gamma
	_{m+r}^{0}P_{k-m-2}^{\left( m+r+2\right) }\right) P_{m-i-1}^{\left(
		r+i\right) },%
	\end{array}
	\label{B51}
	\end{equation}%
	and
	\begin{equation}
	\begin{array}{cl}
	P_{m-i}^{\left( r+i\right) } & =P_{m-n}^{\left(n+r\right) }P_{n-i}^{\left(r+i\right) }-\gamma _{n+r-1}^{0}P_{m-n-1}^{\left( n+r+1\right)
	}P_{n-i-2}^{\left( r+i\right) }\vspace{0.15cm} \\ 	& 
	-\left( \gamma _{n+r}^{1}P_{m-n-1}^{\left( n+r+1\right) }+\gamma_{n+r}^{0}P_{m-n-2}^{\left(n+r+2\right) }\right) P_{n-i-1}^{\left(
		r+i\right) }.
	\end{array}
	\label{B52}
	\end{equation}
	Then replacing each polynomials in the bottom row of $H_{m}$
	by the corresponding recurrence from (\ref{B51}-\ref{B52}), and by theorem \ref{S2},
	we get%
	\begin{equation*}
	\begin{array}{cl}
	H_{m} & =\left( \gamma _{m+r}^{1}P_{k-m-1}^{\left( m+r+1\right) }+\gamma
	_{m+r}^{0}P_{k-m-2}^{\left( m+r+2\right) }\right) \Delta _{n}^{\left(
		r\right) }\left\vert 
	\begin{array}{cc}
	P_{m-n-2}^{\left( r+n+1\right) } & P_{m-n-3}^{\left( r+n+2\right) }\vspace{%
		0.15cm} \\ 
	P_{m-n-1}^{\left( r+n+1\right) } & P_{m-n-2}^{\left( r+n+2\right) }%
	\end{array}%
	\right\vert \vspace{0.15cm} \\ 	& 
	+\gamma _{m+r-1}^{0}P_{k-m-1}^{\left( m+r+1\right) }\Delta _{n}^{\left(
		r\right) }\left\vert 
	\begin{array}{cc}
	P_{m-n-3}^{\left( r+n+1\right) } & P_{m-n-4}^{\left( r+n+2\right) }\vspace{%
		0.15cm} \\ 
	P_{m-n-1}^{\left( r+n+1\right) } & P_{m-n-2}^{\left( r+n+2\right) }%
	\end{array}%
	\right\vert .%
	\end{array}%
	\end{equation*}
	
	Next the following formula completes the proof%
	\begin{align*}
	\begin{array}{cl}
	\left\vert 	\begin{array}{cc}
	P_{m-n-1}^{\left( r+n+1\right) } & P_{m-n-2}^{\left( r+n+2\right) }\vspace{%
		0.15cm} \\ 
	P_{m-n}^{\left( r+n+1\right) } & P_{m-n-1}^{\left( r+n+2\right) }%
	\end{array}%
	\right\vert  & =\gamma _{m+r}^{1}
	\left\vert 	\begin{array}{cc}
	P_{m-n-2}^{\left( r+n+1\right) } & P_{m-n-3}^{\left( r+n+2\right) }\vspace{%
		0.15cm} \\ 
	P_{m-n-1}^{\left( r+n+1\right) } & P_{m-n-2}^{\left( r+n+2\right) }%
	\end{array}%
	\right\vert 
	+\gamma _{m+r-1}^{0}\left\vert 
	\begin{array}{cc}
	P_{m-n-3}^{\left( r+n+1\right) } & P_{m-n-4}^{\left( r+n+2\right) }\vspace{%
		0.15cm} \\ 
	P_{m-n-1}^{\left( r+n+1\right) } & P_{m-n-2}^{\left( r+n+2\right) }%
	\end{array}%
	\right\vert .%
	\end{array}%
	\end{align*}
\end{proof}

We are still working with the case $d=2$, now we shall give a first
Christoffel-Darboux type formula. By setting
\begin{equation*}
J_{m}=\left[ \left( -1\right) ^{m}\prod\nolimits_{l=1}^{m}\gamma _{l+r}^{0}%
\right] ^{-1}\left\vert 
\begin{array}{cc}
P_{m-n-1}^{\left( r+n+1\right) } & P_{m-n-2}^{\left( r+n+2\right) }\vspace{%
	0.15cm} \\ 
P_{m-n}^{\left( r+n+1\right) } & P_{m-n-1}^{\left( r+n+2\right) }%
\end{array}%
\right\vert ,
\end{equation*}%
we get%
\begin{equation*}
\left[ \left( -1\right) ^{m}\prod\nolimits_{l=1}^{m}\gamma _{l+r}^{0}\right]
^{-1}H_{m}=\left( -1\right) ^{n}\prod\limits_{i=1}^{n}\gamma _{i+r}^{0}\left[
P_{k-m-1}^{\left( m+r+1\right) }J_{m}-P_{k-m-2}^{\left( m+r+2\right) }J_{m-1}%
\right] ,
\end{equation*}%
with $J_{n}=0$. In addition, since $H_{n}=0$, we conclude that

\begin{corollary}
	\label{S10} The following relation holds true for any integers $k>m>n\geq 0$ and $%
	r\geq 0$%
	\begin{equation}
	\begin{array}{c}
	\sum\limits_{v=n+1}^{m}\left( \left( -1\right)
	^{v}\prod\limits_{l=1}^{v}\gamma _{l+v}^{0}\right) ^{-1}\left\vert \mathbf{P}%
	_{n}^{r}\mathbf{P}_{v}^{r}\mathbf{P}_{k}^{r}\right\vert ^{T}\vspace{0.15cm}	\\ 
	=\dfrac{\left( -1\right) ^{m-n}\prod\nolimits_{i=1}^{n}\gamma _{i+r}^{0}}{%
		\prod\nolimits_{l=1}^{m}\gamma _{l+r}^{0}}P_{k-m-1}^{\left( m+r+1\right)
	}\left\vert 
	\begin{array}{cc}
	P_{m-n-1}^{\left( r+n+1\right) } & P_{m-n-2}^{\left( r+n+2\right) }\vspace{%
		0.15cm} \\ 
	P_{m-n}^{\left( r+n+1\right) } & P_{m-n-1}^{\left( r+n+2\right) }%
	\end{array}%
	\right\vert .%
	\end{array}
	\label{B60}
	\end{equation}
\end{corollary}

We believe that there exist generalizations of (\ref{B50}) as well as of 
(\ref{B60}) for $d\geq 3$. Of course, it my be difficult to explicitly 
compute them for any $d\geq 3$  in this direction. 
Whereas, it seems that the above formula might affords an alternative way of understanding of the connection between the polynomials $\left\{K_n\right\}$ and 
$\left\{L_n(.;c)\right\}$ appeared in \autoref{sec:kern}.

The rest of this section is devoted to present further generalizations as well as other types of Christoffel-Darboux formula. From the following recurrences%
\begin{equation}
\begin{array}{rl}
x_{1}P_{n+d-1}^{\left( r\right) }\left( x_{1}\right)  & =P_{n+d}^{\left(
	r\right) }\left( x_{1}\right) +\beta _{n+r+d-1}P_{n+d-1}^{\left( r\right)
}\left( x_{1}\right) +...+\gamma _{n+r}^{0}P_{n-1}^{\left( r\right) }\left(
x_{1}\right) ,\vspace{0.15cm} \\ 
x_{2}P_{n+d-2}^{\left( r+1\right) }\left( x_{2}\right)  & =P_{n+d-1}^{\left(
	r+1\right) }\left( x_{2}\right) +\beta _{n+r+d-1}P_{n+d-2}^{\left(
	r+1\right) }\left( x_{2}\right) +...+\gamma _{n+r}^{0}P_{n-2}^{\left(
	r+1\right) }\left( x_{2}\right) ,\vspace{0.15cm} \\ 
\vdots  &  \\ 
x_{d+1}P_{n-1}^{\left( r+d\right) }\left( x_{d+1}\right)  & =P_{n}^{\left(
	r+d\right) }\left( x_{d+1}\right) +\beta _{n+r+d-1}P_{n-1}^{\left(
	r+1\right) }\left( x_{d+1}\right) +...+\gamma _{n+r}^{0}P_{n-d-1}^{\left(
	r+d\right) }\left( x_{d+1}\right) \vspace{0.15cm}%
\end{array}
\tag{S}  \label{S}
\end{equation}%
we get%
\begin{equation*}
\left( \Delta _{n}^{\left( r\right) }\right) ^{-1}\left\vert 
\begin{array}{llll}
\mathbb{P}_{n,-1}^{(r)}(x_{1}) & \mathbb{P}_{n-1,-1}^{(r+1)}(x_{2}) & ... & 
\mathbb{P}_{n-d,-1}^{(r+d)}(x_{d+1})\vspace{0.2cm} \\ 
x_{1}P_{n+d-1}^{(r)}(x_{1}) & x_{2}P_{n+d-2}^{(r+1)}(x_{2}) & ... & 
x_{d+1}P_{n-1}^{(r+d)}(x_{d+1})%
\end{array}%
\right\vert ^{T}=I_{n}-I_{n-1},
\end{equation*}%
where%
\begin{equation*}
I_{n}=\left( \Delta _{n}^{\left( r\right) }\right) ^{-1}\left\vert \mathbb{P}%
_{n}^{\left( r\right) }\left( x_{1}\right) \ \mathbb{P}_{n-1}^{\left(
	r+1\right) }\left( x_{2}\right) ...\mathbb{P}_{n-d}^{\left( r+d\right)
}\left( x_{d+1}\right) \right\vert ^{T},
\end{equation*}%
since $I_{0}=1$, we have the following generalized Christoffel-Darboux type
formula%
\begin{equation*}
\begin{array}{l}
\displaystyle\sum\limits_{v=1}^{n}\dfrac{\Delta _{n}^{\left( r\right) }}{\Delta
	_{v}^{\left( r\right) }}\left\vert 
\begin{array}{llll}
\mathbb{P}_{v,-1}^{(r)}(x_{1}) & \mathbb{P}_{v-1,-1}^{(r+1)}(x_{2}) & ... & 
\mathbb{P}_{v-d,-1}^{(r+d)}(x_{d+1})\vspace{0.2cm} \\ 
x_{1}P_{v+d-1}^{(r)}(x_{1}) & x_{2}P_{v+d-2}^{(r+1)}(x_{2}) & ... & 
x_{d+1}P_{v-1}^{(r+d)}(x_{d+1})%
\end{array}%
\right\vert ^{T}\vspace{0.15cm} \\ 
=\left\vert \mathbb{P}_{n}^{r}\left( x_{1}\right) \ \mathbb{P}%
_{n-1}^{r+1}\left( x_{2}\right) ...\mathbb{P}_{n-d}^{r+d}\left(
x_{d+1}\right) \right\vert ^{T}-\Delta _{n}^{\left( r\right) }.%
\end{array}%
\end{equation*}

This is a generalization of the formula given in \cite{MaroniFSh} and
of the formula in \cite[Prop. 2.10]{Belmehdi}. Notice also that when 
$x_{1}=x_{2}=...=x_{d+1}$, the identity (\ref{B27}) is found.

In the remainder of this section, the notation $X^{\left[ n\right] }$ is used to indicate the n$^{\text{th}}$ derivative of $X$.

Now, replace in the system (\ref{S}) the recurrence of $P_{n+d-i}^{\left(
	r+i\right) }\left( x_{i+1}\right) $ by that of $P_{n+d}^{\left( r\right)
}\left( x_{i+1}\right) $ for $0\leq i\leq d$ respectively. Next taking the
(i-1)$^{\text{th}}$ derivative of the i$^{\text{th}}$ equation, and
eliminate the coefficients $\beta _{n}$ and $\gamma _{n}^k$. Then after
dividing by $\Delta _{n}^{\left( r\right) }$ and taking sum, we obtain%
\begin{equation*}
\begin{array}{l}
\displaystyle\sum\limits_{v=1}^{n}\dfrac{\Delta _{n}^{\left( r\right) }}{\Delta
	_{v}^{\left( r\right) }} \left\vert 
\begin{array}{llll}
\mathbb{P}_{v,-1}^{(r)}(x_1) & \left[\mathbb{P}_{v,-1}^{(r)}(x_2)\right]%
^{\prime } & \cdots & \left[\mathbb{P}_{v,-1}^{(r)}(x_{d+1})\right]^{\left[ d%
	\right] }\vspace{.2cm} \\ 
x_1P_{v+d-1}^{(r)}(x_1) & \left[x_2P_{v+d-1}^{(r)}(x_2)\right]^{\prime } & 
\cdots & \left[x_{d+1} P_{v+d-1}^{(r)}(x_{d+1})\right]^{\left[ d\right] }%
\end{array}%
\right\vert^{T} \vspace{0.15cm} \\ 
=\left\vert \mathbb{P}_{n}^{r}\left( x_{1}\right) \ \mathbb{P}_{n}^{r}\left(
x_{2}\right) ^{\prime }\cdots\mathbb{P}_{n}^{r}\left( x_{d+1}\right) ^{\left[
	d\right] }\right\vert ^{T}-\prod\limits_{k=1}^{d}k!\Delta _{n}^{\left(
	r\right) },
\end{array}%
\end{equation*}
and when $x_{1}=...=x_{d+1}$ we get 
\begin{equation}
\displaystyle\sum\limits_{v=1}^{n}\dfrac{\Delta _{n}^{\left( r\right) }}{\Delta
	_{v}^{\left( r\right) }} \left\vert 
\begin{array}{llll}
\mathbb{P}_{v,-1}^{(r)} & \left[\mathbb{P}_{v,-1}^{(r)}\right]^{\prime } & 
\cdots & \left[\mathbb{P}_{v,-1}^{(r)}\right]^{\left[ d\right] }\vspace{.2cm}
\\ 
0 & P_{v+d-1}^{(r)} & \cdots & \left[P_{v+d-1}^{(r)}\right]^{\left[ d-1%
	\right] }%
\end{array}%
\right\vert^{T} 
=\left\vert\ \mathbb{P}_{n}^{r}\ \ \left[\mathbb{P}_{n}^{r}\right]^{\prime
}\ \cdots \ \left[\mathbb{P}_{n}^{r}\right]^{\left[ d\right] }\right\vert
^{T}-\prod\limits_{k=1}^{d}k!\Delta _{n}^{\left( r\right) }.%
\label{B65}
\end{equation}

Similarly by taking in the system (\ref{S}), the (i-1)$^{\text{th}}$
derivative of the i$^{\text{th}}$ equation we obtain, using the similar
approach above, the following 
\begin{equation*}
\begin{array}{l}
\displaystyle\sum\limits_{v=1}^{n}\dfrac{\Delta _{n}^{\left( r\right) }}{\Delta
	_{v}^{\left( r\right) }} \left\vert 
\begin{array}{llll}
\mathbb{P}_{v,-1}^{(r)}(x_1) & \left[\mathbb{P}_{v-1,-1}^{(r+1)}(x_2)\right]%
^{\prime } & \cdots & \left[\mathbb{P}_{v-d,-1}^{(r+d)}(x_{d+1})\right]^{%
	\left[ d\right] }\vspace{.2cm} \\ 
x_1P_{v+d-1}^{(r)}(x_1) & \left[x_2P_{v+d-2}^{(r+1)}(x_2)\right]^{\prime } & 
\cdots & \left[x_{d+1} P_{v-1}^{(r+d)}(x_{d+1})\right]^{\left[ d\right] }%
\end{array}%
\right\vert^{T} \vspace{0.15cm} \\ 
=\left\vert \mathbb{P}_{n}^{r}\left( x_{1}\right) \ \left[\mathbb{P}%
_{n-1}^{r+1}\left( x_{2}\right)\right] ^{\prime }\cdots \left[\mathbb{P}%
_{n-d}^{r+d}\left(x_{d+1}\right)\right] ^{\left[ d\right] }\right\vert ^{T},
\end{array}%
\end{equation*}%
and when $x_{1}=...=x_{d+1}$ we infer that 
\begin{equation*}
\displaystyle\sum\limits_{v=1}^{n}\dfrac{\Delta _{n}^{\left( r\right) }}{\Delta
	_{v}^{\left( r\right) }}\left\vert 
\begin{array}{llll}
\mathbb{P}_{v,-1}^{(r)} & \left[\mathbb{P}_{v-1,-1}^{(r+1)}\right]^{\prime }
& \cdots & \left[\mathbb{P}_{v-d,-1}^{(r+d)}\right]^{\left[ d\right] }%
\vspace{.2cm} \\ 
0 & P_{v+d-2}^{(r+1)} & \cdots & \left[P_{v-1}^{(r+d)}\right]^{\left[ d-1%
	\right] }%
\end{array}%
\right\vert^{T} 
=\left\vert\ \mathbb{P}_{n}^{r}\ \ \left[%
\mathbb{P}_{n-1}^{r+1}\right]^{\prime }\ \cdots \ \left[\mathbb{P}%
_{n-d}^{r+d}\right]^{\left[ d\right] }\right\vert ^{T}.%
\end{equation*}

\section{Darboux transformations} \label{sec:lu}

This section deals with $LU$ as well as $UL$ decomposition of the lower Hessenberg
matrix $J_{d}$. A motivation of this decomposition comes out in the study of Kostant-Toda lattice  \cite{BarriosBranq}, where operators in the commutator are banded matrices. 
The authors show further that the matrix $L$ could be written as a product of $d$ bi-diagonal matrices  $L=L_{1}L_{2}...L_{d}$ with a full description in case $d=2$. 
Moreover, two of the authors considered latter $d$ Darboux transformations of $J_{d}$ and defined 
$d$ new matrices through
\begin{equation}
\begin{array}{cl}
J_{d}^{\left( 0\right) } & =J_{d}, \\ 
J_{d}^{\left( i\right) } & =L_{i+1}...L_{d}UL_{1}...L_{i}+\lambda I,\text{
	for }i=1,2,..,d.
\end{array}
\label{DT}
\end{equation}
Therein, they have shown that the above transformations generate $d$
solutions denoted $\{P_{n}^{(i)}\} $ of $(d+2)$-term recurrence relation, i.e., 
any circular permutation between the matrices $L_{i}$ for $1\leq i\leq d$ 
and the matrix $U$ brings forth another solution of the
recurrence (see \cite[p.123]{Barrios} for more details). Furthermore, by
denoting $v^{(i)}(z)=(P_0^{(i)}(z),P_1^{(i)},...)^T$, they have obtained the following
connection 
\begin{equation*}
L_{j+1}L_{j+2}...L_{i}v^{(i)}(z)=v^{(j)}(z), \ \ 0\leq j <i\leq d,
\end{equation*}
\begin{equation}
\begin{array}{cl}
P_{m+1}^{\left( i\right) } & =P_{m+1}^{\left( i+1\right)
}+l_{(d+1)m+i+2}^{(i)}P_{m}^{\left( i+1\right) },\ \ m=0,1,... \\ 
P_{0}^{\left( i+1\right) } & =1,
\end{array}
\label{DTB}
\end{equation}%
where $l_{(d+1)m+i+2}^{(i)}$ are the entries at position $(m+1,m)$ of the
matrix $L_{i}$. 
We would like to point out that the latter transformations 
have also been investigated in \cite{BranqSymmetriz}.

The $d$-orthogonality of the above polynomials gives evidence of the following question: what kind these polynomials are? We refer to the above recursion in the next two sections, we establish that the above polynomials are in fact the
($d+1$)-decomposition of some $d$-symmetric sequence.

Let us denote the matrices $U$ and $L$ as follows 
\begin{equation}
U=\left( 
\begin{array}{cccc}
m_{1} & 1 &  &  \\ 
& m_{2} & 1 &  \\ 
&  & \ddots & \ddots%
\end{array}%
\right) ,\ \ L=\left( 
\begin{array}{ccccc}
1 &  &  &  &  \\ 
l_{11} & 1 &  &  &  \\ 
\vdots & \ddots & \multicolumn{1}{l}{\ddots} & \multicolumn{1}{l}{} &  \\ 
l_{d1} & \cdots & l_{dd} & 1 &  \\ 
0 & \ddots &  &  & \ddots%
\end{array}
\right) .  \label{DT1}
\end{equation}

First let us express the matrix $J_{d}$ as the product of $U$ times $L$. The following results generalizing those in \cite{Bueno}

\begin{proposition}
	\label{T11} Let $\{P_{n}(x)\}_{n\geq 0}$ be a d-OPS defined by the $J_{d}$
	given in (\ref{Y9}). Assume that $P_{n}\left( 0\right) \neq 0,$ $n\geq 1$.
	Then, for the $LU$ decomposition of the matrix $J_{d}$, we have 
	\begin{equation}
	\begin{array}{l}
	m_{1}=\beta _{0} \vspace*{1mm}\\ 
	m_{n}=\beta _{n-1}-l_{n-1,n-1},\text{\ \ for \ }n\geq 2,\vspace*{1mm} \\ 
	l_{n+d-1,n}=\gamma _{n}^{0}/m_{n},,\text{\ \ for \ }n\geq 1, \vspace*{1mm}\\ 
	l_{n+i,n}+l_{n+i,n+1}m_{n+1}=\gamma _{n+1}^{d-i},\text{\ \ for \ }1\leq
	i\leq d,%
	\end{array}
	\label{DT2}
	\end{equation}%
	where the elements $l_{n,k}$ can be computed recursively in the following manner 
	\begin{equation}
	\begin{array}{l}
	l_{i,1}m_{1}=\gamma _{1}^{d-i},\text{\ \ for \ }1\leq i\leq d,\vspace*{1mm} \\ 
	l_{n+i,n+1}=\gamma _{n+1}^{d-i}/m_{n+1}-l_{n+i,n}/m_{n+1},\text{\ \ for \ }%
	1\leq i\leq d.%
	\end{array}
	\label{DT3}
	\end{equation}
	
	Moreover, we have 
	\begin{equation}
	m_{n}=-P_{n}(0)/P_{n-1}(0).  \label{DT4}
	\end{equation}
\end{proposition}

\begin{proof}
	The product of $L$ times $U$ gives 
	\begin{equation*}
	\begin{array}{rl}
	\beta _{0} & =m_{1}, \vspace*{1mm}\\ 
	\gamma _{1}^{d-i} & =l_{i1}m_{1}, \vspace*{1mm}\\ 
	\gamma _{n+1}^{d-i} & =l_{n+i,n}+l_{n+1,n+1}m_{n+1},
	\end{array}%
	\end{equation*}%
	whence the recursions (\ref{DT2})-(\ref{DT3}).
	
	The equality (\ref{DT4}) can be checked by induction on $n$. 
	Since $P_{1}\left( 0\right) =-\beta _{0},$ then 
	\begin{equation*}
	m_{1}=\beta _{0}=-P_{1}\left( 0\right) /P_{0}\left( 0\right) .
	\end{equation*}
	
	Assume that $m_{k}=-P_{k}\left( 0\right) /P_{k-1}\left( 0\right) $ for 
	$k\leq n$. Then from the recurrence relation (\ref{B6}) we get
	\begin{equation*}
	P_{n+1}\left( 0\right) =-\beta _{n}P_{n}\left( 0\right) -\gamma
	_{n}^{d-1}P_{n-1}\left( 0\right) -...-\gamma _{n-d+1}^{0}P_{n-d}\left(
	0\right) ,
	\end{equation*}
	hence
	\begin{equation*}
	-\frac{P_{n+1}\left( 0\right) }{P_{n}\left( 0\right) }=\beta _{n}+\frac{
		\gamma _{n}^{d-1}}{P_{n}\left( 0\right) /P_{n-1}\left( 0\right) }+...+\frac{
		\gamma _{n-d+1}^{0}}{\left[ P_{n}\left( 0\right) /P_{n-1}\left( 0\right) 
		\right] ...\left[ P_{n-d+1}\left( 0\right) /P_{n-d}\left( 0\right) \right] },
	\end{equation*}
	using the induction hypothesis as well as (\ref{DT3}), we infer that
	\begin{equation*}
	-\frac{P_{n+1}\left( 0\right) }{P_{n}\left( 0\right) }=\beta _{n}-\frac{
		\gamma _{n}^{d-1}}{m_{n}}+...+\left( -1\right) ^{n+d-1}\frac{\gamma
		_{n-d+1}^{0}}{m_{n}m_{n-1}...m_{n-d+1}}.
	\end{equation*}
	
	Now, from the first equality in (\ref{DT3}), we remark that the
	last two terms  can be written as
	\begin{equation*}
	\begin{array}{cl}
	\displaystyle \left( -1\right) ^{n+d-2}\left[ \frac{\gamma _{n-d+2}^{1}-\frac{\gamma
			_{n-d+1}^{0}}{m_{n-d+1}}}{m_{n}m_{n-1}...m_{n-d+2}}\right] & \displaystyle=\left(
	-1\right) ^{n+d-2}\frac{\gamma _{n-d+2}^{1}-l_{n,n-d+1}}{%
		m_{n}m_{n-1}...m_{n-d+2}}\vspace{2mm} \\ 	&
	\displaystyle 
	=\left( -1\right) ^{n+d-2}\frac{\frac{\gamma _{n-d+2}^{1}}{m_{n-d+2}}-%
		\frac{l_{n,n-d+1}}{m_{n-d+2}}}{m_{n}m_{n-1}...m_{n-d+3}} \vspace{2mm}\\ 
	&\displaystyle =\left( -1\right) ^{n+d-2}\frac{l_{n,n-d+2}}{m_{n}m_{n-1}...m_{n-d+3}}.%
	\end{array}%
	\end{equation*}
	
	By induction we get at end 
	\begin{equation*}
	m_{n+1}=-\frac{P_{n+1}\left( 0\right) }{P_{n}\left( 0\right) }=\beta _{n}-%
	\left[ \frac{\gamma _{n}^{d-1}}{m_{n}}-\frac{l_{n,n-1}}{m_{n}}\right] =\beta
	_{n}-l_{n,n}.
	\end{equation*}
\end{proof}

Now, for the $UL$ decomposition we have

\begin{proposition}
	\label{T12} Assume that $J_{d}=UL$ denotes the $UL$ factorization of the
	lower Hessenberg matrix $J_{d}$. We have for $1\leq j\leq d$, the following initial
	conditions 
	\begin{equation}
	\begin{array}{l}
	l_{jj}=\beta _{j-1}-\mu _{j-1},\vspace*{1mm} \\ 
	l_{j,i}=\gamma _{i}^{d-j+i}-\eta _{i}^{d-j+i},\text{\ \ for \ }1\leq i\leq
	j-1,%
	\end{array}
	\label{DT5}
	\end{equation}%
	where%
	\begin{equation}
	\begin{array}{rl}
	\mu _{j-1} & =m_{j},\ \ \text{for \ }j\geq 1, \vspace*{1mm}\\ 
	\eta _{i}^{d-j+i} & =m_{j}l_{j-1,i},\ \ \text{for\ \ }1\leq i\leq j-1\leq
	d-1,%
	\end{array}
	\label{DT6}
	\end{equation}%
	are free parameters, and for $n\geq d+1$, the following
	\begin{equation}
	\begin{array}{l}
	m_{d+n}=\gamma _{n}^{0}/l_{d+n-1,n},\ \ \text{for\ \ }n\geq 1, \vspace*{1mm}\\ 
	l_{d+n,i}=\gamma _{i}^{i-n}-m_{d+n}l_{d+n-1,i},\ \ \text{for\ \ }n+1\leq i\leq n-2.
	\end{array}
	\label{DT7}
	\end{equation}
	
	In addition, the free parameters $\mu _{i}$ and $\eta _{i}$ define a new
	sequence of co-recursive polynomials which can be used to determine $l_{ij}$. 
	Furthermore, for $1\leq n\leq d$ we have the following 
	\begin{equation}
	\begin{array}{rl}
	-l_{nn} & =Q_{1}^{\left( n-1\right) }(0)=\mu _{n-1}-\beta _{n-1}, \vspace*{1mm}\\ 
	-l_{n,n-1} & =Q_{2}^{\left( n-2\right) }(0)+l_{nn}Q_{1}^{\left( n-2\right)
	}(0), \\ 
	\vdots &  \\ 
	-l_{n,n+1-d} & =Q_{d}^{\left( n-d\right) }(0)+l_{nn}Q_{d-1}^{\left(
		n-d\right) }(0) +...+l_{n,n+2-d}Q_{1}^{\left( n-d\right) }\left( 0\right) ,
	\end{array}
	\label{DT8}
	\end{equation}
	and the recursion (\ref{DT7}) for $n\geq d+1$.
\end{proposition}

\begin{proof}
	The product of $U$ times $L$ gives, for $1\leq j\leq d$
	\begin{equation*}
	\begin{array}{rl}
	\beta _{j-1} & =l_{jj}+m_{j}, \vspace*{1mm}\\ 
	\gamma _{i}^{d-j+i} & =l_{ji}+m_{j}l_{j-1,i},\text{\ \ with \ }1\leq i\leq j-1,
	\end{array}
	\end{equation*}
	which shows that there are exactly $d$ free parameters $\left\{m_{j}\right\} _{1\leq j\leq d-1}$, 
	and when $n=d+1$, we get the following
	\begin{equation*}
	\begin{array}{rl}
	m_{d+1} & =\gamma _{1}^{0}/l_{d1}, \\ 
	l_{d+1,i} & =\gamma _{i}^{i-1}-m_{d+1}l_{di},\text{\ for \ }2\leq i\leq d,
	\end{array}
	\end{equation*}
	and for $n\geq d+1$, we get (\ref{DT7}).
	
	The proof of (\ref{DT8}) follows readily by combining the co-recursive's recurrence relation 
	and the associated polynomials. Notice that this is just a simple idea on how to compute the coefficients $l_{n,k}$. 
	Indeed, let us denote the sequence of co-recursive polynomials generated by perturbing the
	recurrence of $\left\{ P_n\right\}$ through the free parameters $\mu_i$ and 
	$\eta_i$ by $\left\{ Q_n\right\}$. In this case, we have for $1\leq n\leq d$ 
	\begin{equation*}
	-Q_{n}\left( 0\right) =l_{nn}Q_{n-1}\left( 0\right) +l_{n,n-1}Q_{n-2}\left(
	0\right) +...+l_{n,n+1-d}Q_{n-d}\left( 0\right) + +\gamma_{n-d}^{0}
	Q_{n-d-1}\left( 0\right) .
	\end{equation*}
	
	Now we determine the coefficients $l_{n,i}$. Remark first that%
	\begin{equation*}
	-Q_{1}^{\left( n-1\right) }\left( 0\right) =\beta _{n-1}-\mu _{n-1}=l_{n,n},
	\end{equation*}%
	and 
	\begin{equation*}
	-Q_{2}^{\left( n-2\right) }\left( 0\right) =l_{n,n}Q_{1}^{\left( n-2\right)
	}(0)+l_{n,n-1}Q_{0}^{\left( n-2\right) }(0),
	\end{equation*}
	hence the proof follows by induction on $n$.
\end{proof}

\section{Kernel polynomials and quasi-orthogonality}\label{sec:kern}

The following question indicates just how little we know about kernel
polynomials.

Darboux transformation allows us to deduce a new family of polynomials for
which our sequence $\{P_n\}$ is $d$-quasi-orthogonal of order one \cite{Sadek2}.
In the standard orthogonality, the expressions given in the following
proposition define Kernel polynomials from the quasi-orthogonality's point of view. 
Furthermore, we have the following
result

\begin{proposition}
	\label{T13} Let $\left\{ P_{n}\right\} $ be a $d$-OPS and $J_{d}$ the
	corresponding lower Hessenberg matrix, and $\left\{ K_{n}\right\} $ denotes the
	sequence of polynomials generated by $J_{d}^{d}=UL$. Then 
	\begin{equation}
	P_{n}=K_{n}+l_{n,n}K_{n-1}+l_{n,n-1}K_{n-2}+...+l_{n,n-d+1}K_{n-d},\ n\geq 0
	\label{DT9}
	\end{equation}%
	and 
	\begin{equation}
	xK_{n}(x)=P_{n+1}(x)-\frac{P_{n+1}(0)}{P_n(0)}P_n(x),\ \ n\geq 0.
	\label{DT10}
	\end{equation}
\end{proposition}

\begin{proof}
	Define a monic polynomials sequence $\left\{ R_{n}\right\} $ by 
	\begin{equation*}
	R_{n+1}=P_{n+1}+m_{n+1}P_{n}.
	\end{equation*}
	
	Then, 
	\begin{equation*}
	x\mathbb{P}=J_{d}\mathbb{P}=LU\mathbb{P}=L\left( R_{1},R_{2},...\right) ^{T},
	\end{equation*}
	that is 
	\begin{equation*}
	xP_{n}=R_{n+1}+l_{n,n}R_{n}+l_{n,n-1}R_{n-1}+...+l_{n,n-d+1}R_{n-d+1}.
	\end{equation*}
	
	Remark that $R_{n}\left( 0\right) =0$ because at least $l_{n,n-d+1}\neq 0$.
	That is, $R_{n}\left( x\right) =xS_{n-1}\left( x\right) $ whence (\ref{DT9}%
	). On the other hand, we get 
	\begin{equation*}
	x\mathbb{S}=U\mathbb{P}=UL\mathbb{S}
	\end{equation*}%
	which means that the sequence $\left\{ S_{n}\right\} $ is $d$-OPS
	corresponding to the Darboux transformation $J_{d}^{d}=UL$, i. e. $\{S_n\}=\{K_n\}$.
\end{proof}

The expression (\ref{DT9}) means that $\{P_n\}$ is $d$-quasi-orthogonal of
order exactly one with respect to the corresponding vector of linear forms of $\{K_n\}$ (\ref{ES}) \cite{Sadek2}. In the usual orthogonality, recurrences (\ref{DT9})
and (\ref{DT10}) reduce, respectively, to the formulas (9.5) and (9.4) in 
\cite[p.45]{Chihara} (see also exercise 9.6 p.49). 
Then these define kernel polynomials in the $d$-orthogonality sense.

Notice also that the polynomials generated by Darboux transformations $J_d$
and $J_d^d$ can be related through the matrix of change of basis $L$ in the
form $\mathbb{P}=L\mathbb{K}$.

It is obvious that the recurrence of kernel polynomials 
$\{K_n\}$ as well as of $\{P_n\}$ could be extremely determined by using only
the two recurrences (\ref{DT9})-(\ref{DT10}). 
Indeed, suppose that $J_{d}=UL$ and define $\mathbb{P}=L\mathbb{K}$, i.e., that is by 
(\ref{DT9}). Then 
\begin{equation*}
x\mathbb{K}=UL\mathbb{K}=U\mathbb{P}
\end{equation*}%
hence%
\begin{equation*}
LU\mathbb{P}=xL\mathbb{K}=x\mathbb{P}
\end{equation*}%
which means that $\left\{ P_{n}\right\} $ is $d$-OPS generated by $LU$. Using
once again the recurrence (\ref{DT9})-(\ref{DT10}) we get 
\begin{align*}
xP_{n} & =P_{n+1}+\left( l_{nn}+m_{n+1}\right) P_{n}
+\sum\nolimits_{i=0}^{d-2}\left( l_{n,n-1-i}+l_{n,n-i}m_{n-i}\right)
P_{n-i-1}\\&+l_{n,n-d+1}m_{n-d+1}P_{n-d}%
\end{align*}
and 
\begin{align*}
xK_{n} & =K_{n+1}+\left( l_{n+1,n+1}+m_{n+1}\right) K_{n}+\sum\nolimits_{i=0}^{d-2}\left( l_{n+1,n-i}+l_{n,n-i}m_{n+1}\right)
K_{n-i-1}\\&+l_{n,n-d+1}m_{n+1}K_{n-d}%
\end{align*}%
then according to propositions \ref{T11} and \ref{T12} we have respectively
the recurrence of $\left\{P_{n}\right\} $ as well as that of $\left\{K_{n}\right%
\} $.

To determine the dual sequence of $\left\{K_n\right\}$ which we denote by $%
\mathcal{V}=(v_0,...,v_{d-1})^T$, we use the $d$-quasi-orthogonality as it was
already pointed out in \cite{Sadek2}. Since 
\begin{equation*}
\begin{array}{l}
\left\langle v_{r},P_{n}\right\rangle =\left\langle v_{r},k_{n}\right\rangle
+l_{n,n}\left\langle v_{r},K_{n-1}\right\rangle +...+l_{n,n-d+1}\left\langle
v_{r},K_{n-d}\right\rangle =0,\ \ n\geq r+d+1,\vspace{0.2cm} \\ 
\left\langle v_{r},P_{r+d}\right\rangle =l_{r+d,r+1}\left\langle
v_{r},K_{r}\right\rangle \neq 0,%
\end{array}%
\end{equation*}
then, there exists $r\leq t_{r}\leq r+d$ such that 
\begin{equation*}
\begin{array}{l}
\left\langle v_{r},P_{n}\right\rangle =0,\ \ n\geq t_{r}+1,\vspace{0.2cm} \\ 
\left\langle v_{r},P_{t_{r}}\right\rangle \neq 0.%
\end{array}%
\end{equation*}

According to theorem \ref{T6}, there exist $d$ polynomials $\phi _{r}^{\mu }$, 
$0 \leq r,\mu \leq d-1$, such that 
\begin{equation*}
v_{r}=\sum\nolimits_{\mu =0}^{d-1}\phi _{r}^{\mu }u_{\mu }.
\end{equation*}
Set $t_{r}=q_{r}d+p_{r},\ \ 0\leq p_{r}\leq d-1$. Since $r\leq t_{r}\leq r+d$
and $p_r\leq d-1$, then $q_r\leq 1$. Further, if $q_{r}=1$, then $p_{r}\leq
r $. Hence, the above expression of $v_r$ takes the following form 
\begin{equation}
v_r=\sum_{i=0}^{r}(a_r^ix-b_r^i)u_i-\sum_{j=r+1}^{d-1}b_r^ju_j.  \label{L2}
\end{equation}

Now, applying both sides of $v_r$ recursively on the polynomials 
$P_{0},...,P_{d-1}$, and making use of (\ref{R1}), we obtain the following
expressions

\begin{equation}
\begin{array}{rll}
b_{0}^{0} & =a_0^0\beta _{0}-1, \hspace{.3cm} \text{and} \hspace{.3cm}
b_{r}^{0} =a_r^0\beta _{0}+a_r^1, & \text{for} \ \ 1\leq r\leq d-1, \vspace{%
	0.2cm} \\ 
b_{r}^{r} & =a_r^0\gamma _{1}^{d-r}+a_r^1\gamma
_{2}^{d-r+1}+...+a_r^{r-1}\gamma _{r}^{d-1}+a_r^r\beta _{r}-1, & \text{for}
\ \ 1\leq r\leq d-1, \vspace{0.2cm} \\ 
b_{r}^{i} & =a_r^0\gamma _{1}^{d-i}+a_r^1\gamma_{2}^{d-i+1}+...
+a_r^{i-1}\gamma_{i}^{d-1}+a_r^i\beta _{i}+a_r^{i+1}, \ \  & \text{for} \ \
1\leq i< r\leq d-1,\vspace{0.2cm} \\ 
b_{r}^{i} & =a_r^0\gamma _{1}^{d-i}+a_r^1\gamma _{2}^{d-i+1}+...
+a_r^r\gamma_{r+1}^{d-i+r}-l_{i,r+1},\ \  & \text{for} \ \ 1\leq r<i\leq d-1.%
\end{array}
\label{L3}
\end{equation}

The $a_i$'s, are easily obtained from the following 
\begin{equation*}
\left\langle v_{r},P_{d+i}\right\rangle =l_{d+i,r+1} =a_r^i\gamma
_{i+1}^{0}+a_r^{i+1}\gamma_{i+2}^{1}+...+a_r^{r}\gamma_{r+1}^{r-i}, \ \ 
\text{for} \ \ 0\leq i\leq r.
\end{equation*}

It may be worthwhile to consider analogous problems of the usual
orthogonality in which kernel polynomials appeared as particular or as a
solution of the whole problem. It could be then of interest to study 
Uvarov modification of the measure since the regularity as well as 
the recurrence coefficients corresponding to the new sequence which are expressed in terms of kernel polynomials. For this end,  the following problem is considered.

Given a $d$-OPS $\left\{P_{n}\right\} $ with respect to some regular $\mathcal{%
	U}=(u_0,\dots,u_{d-1})^T$, and define a new vector form $\mathcal{V}%
=(v_0,\dots,v_{d-1})^T$ as 
\begin{equation}
v_r=u_r+\lambda \delta_c, \ \ 0\leq r\leq d-1.  \label{UT1}
\end{equation}
When $\mathcal{V}$ is regular, we denote the corresponding $d$-OPS by $%
\left\{Q_{n}\right\} $.

Now suppose that $\mathcal{V}$ is regular and write 
\begin{equation*}
Q_m=P_m+\sum_{i=0}^{m-1}a_{m,i}P_{i}.
\end{equation*}
Then, by taking $m=dn+k$ with $0\leq k\leq d-1$, we get 
\begin{equation*}
\left\langle u_{r},Q_{dn+r}P_j\right\rangle=\left\{ 
\begin{tabular}{ll}
$\left\langle u_{r},P_{dn+r}P_n\right\rangle=\prod\nolimits_{\nu
	=1}^{n}\gamma _{d(\nu-1)+r+1}^0$, & \ for \ $j=n$, \vspace{.2cm} \\ 
$a_{dn+r,dj+r}\left\langle u_{r},P_{dj+r}P_j\right\rangle$, & \ for $0\leq
j\leq n-1$,%
\end{tabular}%
\right.
\end{equation*}
and by (\ref{UT1}) we also have 
\begin{equation*}
\left\langle u_{r},Q_{dn+r}P_j\right\rangle=\left\{ 
\begin{tabular}{ll}
$\left\langle v_{r},Q_{dn+r}P_n\right\rangle-\lambda Q_{dn+r}(c)P_n(c)$, & \
for \ $j=n$, \vspace{.2cm} \\ 
$-\lambda Q_{dn+r}(c)P_j(c)$, & \ for $0\leq j\leq n-1$.%
\end{tabular}%
\right.
\end{equation*}

Accordingly, we obtain 
\begin{equation}
Q_{dn+k}(x)=P_{dn+k}(x)-\lambda Q_{dn+k}(c)L_{dn+k-1}(x;c),  \label{UT3}
\end{equation}
where $L_n(x;.)$ is the polynomial defined as 
\begin{equation}
L_{dn+k}(x;c)=\sum_{j=0}^{n-1}P_j(c)\left\{ \sum_{r=0}^{d-1}\frac{P_{dj+r}(x)%
}{\left\langle u_{r},P_{dj+r}P_j\right\rangle} \right\}+P_n(c)\left\{
\sum_{r=0}^{k}\frac{P_{dn+r}(x)}{\left\langle u_{r},P_{dn+r}P_n\right\rangle}
\right\}.  \label{UT4}
\end{equation}

Set $x=c$ in (\ref{UT3}) to get 
\begin{equation*}
Q_{dn+k}(c)\left[1+\lambda L_{dn+k-1}(c;c)\right]=P_{dn+k}(c),
\end{equation*}
where necessarily $1+\lambda L_{dn+k-1}(c;c)\neq 0$, otherwise, we also have 
$P_{dn+k}(c)=0$. It now follows by induction that $c$ should be a common zero
for more than $d$ consecutive polynomials of the sequence $%
\left\{P_{n}\right\} $ which is impossible (see corollary \ref{Z2} bellow). Hence 
\begin{equation*}
Q_{dn+k}(c)=P_{dn+k}(c)\left[1+\lambda L_{dn+k-1}(c;c)\right]^{-1}
\end{equation*}
and then 
\begin{equation}
Q_{m}(x)=P_{m}(x)-\lambda \frac{P_m(c)}{1+\lambda L_{m-1}(c;c)}L_{m-1}(x;c).
\label{UT5}
\end{equation}

While it is obvious that the polynomials $L_n(x;c)$ given by the
formula (\ref{UT4}) are exactly kernel polynomials in the case $d=1$, it is not yet clear how this sequence is connected to the polynomials $K_n$
appeared in proposition \ref{T13}. Moreover, the correct expression of 
kernel polynomials as a sum in terms of $\left\{P_{n}\right\} $ that provides 
such more general properties is not yet come over.
In the usual orthogonality ($d=1$), the polynomials $L_n(x;c)$ are extremely 
used to characterize the regularity of the corresponding linear form $\mathcal{V}$, 
they were never the main heroes here, but we leave regularity questions of 
the Uvarov transform outside the scope of this paper. Notice also that other perturbations are considered recently in \cite{BarriosGarcia,BarriosGarciaManrique}.

Let us now focus on the $d$-quasi-orthogonality appeared in the proposition \ref{T13}.
In fact, the above results (on the $d$-quasi-orthogonality), allow us to
extract another characterization of the $d$-quasi-orthogonality. Actually, the
$d$-quasi-orthogonality between two sequences of $d$-OPS subject to some
conditions \cite{Sadek2}.

\begin{theorem}\label{NCQO2}
	Let $\left\{P_{n}\right\} $ and $\left\{Q_{n}\right\} $ be two $d$-OPS with
	respect to $\mathcal{U}$ and $\mathcal{V}$ respectively. Then the following assertions are equivalent
	\begin{enumerate}
		\item[(1)]  
		$\left\{P_{n}\right\} $ is $d$-quasi-orthogonal of order at most $l$ with respect to 
		$\mathcal{V}$,
		
		\item[(2)] There exists a matrix polynomial $\Phi=\left( \phi _{t}^{s}\right)$ such that $\mathcal{V}=\Phi \mathcal{U}$ with
		\begin{equation}
		\hspace{-5mm}	\begin{tabular}{lll}
		\text{Zone A}\hspace{5mm}	$\deg \phi _{t}^{s}=l,$ & $0\leq s\leq t+r-d, $ & $\text{for\ \ }d-r\leq t\leq d-1, \vspace{0.2cm}$ \\ 
		\text{Zone B}\hspace{5mm}	 $\deg \phi _{t}^{s}\leq l-1,$ & $0\leq s\leq t+r, $ & $\text{for} \ \ 0\leq t\leq d-r-1, \vspace{0.2cm}$ \\
		\text{Zone C}\hspace{5mm}	 $\deg \phi _{t}^{s}\leq l-2,$ & $r+t+1\leq s\leq d-1, $ & $\text{for} \ \ 0\leq t\leq d-r-2.$
		\end{tabular}
		\label{E25}
		\end{equation}
		
		\item[(3)] there exist an integer $1\leq r\leq d$ and
		complex numbers $a_{n,k}$ such that 
		\begin{align}
		P_n(x)&=Q_{n}(x)+a_{n,1}Q_{n-1}(x)+...+
		a_{n,(l-1)d}Q_{n-(l-1)d}(x)\nonumber\\ &+\sum_{i=1}^{r}a_{n,(l-1)d+i}Q_{n-(l-1)d-i}(x),  \label{QO3}
		\end{align}
		
		\item[(4)] There exist a polynomial $\pi_{\sigma} $ of degree $\sigma \leq d(l-1)+r$ 	and a polynomial matrix $\Phi ^{\ast }=\left(\chi_{t}^{s}\right)$ 	such that
		$\pi_{\sigma} \mathcal{U}=\Phi ^{\ast }\mathcal{V}$ with 	
		\begin{equation}
		\begin{tabular}{lll}
		$\deg \chi_{t}^{s}\leq(d-1)(l-1)+r,$ & $0\leq s\leq d-r-1, $ & $\text{for\ \ }r\leq t\leq d-1, \vspace{0.2cm}$ \\ 
		$\deg \chi _{t}^{s}\leq (d-1)(l-1)+r-1,$ & $0\leq s\leq d-r, $ & $\text{for} \ \ 0\leq t\leq r-1, \vspace{0.2cm}$ \\
		$\deg \chi _{t}^{s}\leq (d-1)(l-1)+r-2,$ & $d-r+1\leq s\leq d-1, $ & $\text{for} \ \ r\geq 2.$
		\end{tabular}\label{QO5}
		\end{equation}
		
		\item[(5)] there exist a polynomial $\pi_{\sigma} $ and integers $\sigma \leq d(l-1)+r$ and $\rho\leq d\big((d-1)(l-1)+r\big)-r$ such that 
		\begin{align}
		\pi_{\sigma}(x)Q_n(x)&=
		\sum_{v=n-\rho}^{n+\sigma}a_{n,v}P_v(x). \label{QO4}
		\end{align}
	\end{enumerate}
\end{theorem}

\begin{remark}\label{RBand}
	The shape of the matrix $\Phi$ is as follows
	
	\NiceMatrixOptions{transparent}
	$\Phi=\begin{pmatrix}
	\phi_0^0 	    & \cdots & \phi_0^r    & \phi_{0}^{r+1}      &  \cdots   \rotatebox{-35}{\color{green} Zone C}\cdots   &   \phi_{0}^{d-1}	     \\
	\vdots 	&        &      & \ddots & \ddots   &     \vdots \\
	\phi_{d-r-1}^{0}       &        & \hspace*{5mm}\rotatebox{-35}{\color{blue} Zone B}    &        &          \ddots  &    \phi_{d-r-2}^{d-1} \\
	\phi_{d-r}^0       &\ddots  &      &     &     &     \phi_{d-r-1}^{d-1} \\
	\vdots\rotatebox{-35}{\color{red} Zone A} 	&  \ddots       &   \ddots   &  &    &     \vdots \\
	\phi_{d-1}^0& \cdots &        \phi_{d-1}^{r-1} & \phi_{d-1}^r     & \cdots &      \phi_{d-1}^{d-1} 
	\end{pmatrix}$
	
	Let us remark in the above theorem that if $r=d-1$, then the third line in (\ref{E25}), i.e. zone C, disappears and we get polynomials with degree $l$ bellow the main diagonal and of degree $l-1$ at the main diagonal and above. Yet when $r=d$ we increase the degree of the main diagonal by one and we obtain the $d$-quasi-orthogonality of order exactly $l$  then we recover \cite[Propo. 3.1]{Sadek2} and also \cite[Th. 3.1]{Maronid}. 
	Illustrative examples  of the equivalence between (2) and (3) are given in \cite{Sadek6} with $l=1$ and $r=1$. The next example provides an explanation on how the degrees of $\phi_{t}^s$ are related to the number of terms in the linear combination (\ref{QO3}).
\end{remark}

\begin{example}
	Let us look at degree of the entries of the matrix polynomial $\Phi$ corresponding to the following linear combination 
	\begin{equation}
	P_{n}(x)=Q_{n}(x) +a_{n,1}Q_{n-1}(x)+...+a_{n,s}Q_{n-s}(x),\ \ 1\leq s\leq d-1, \ n\geq 0. \label{STLC1}
	\end{equation}
	As in the case $r=d$ (kernel polynomials), following the same resonance we deduce that the matrix $\Phi $ yields 
	\begin{align*}
	v_t&=\sum_{i=0}^{s+t-1}b_t^iu_i, \hspace{5mm} 0\leq t\leq d-r, \hspace{5mm} b_t^i\equiv 0, \ \ s+t\leq i\leq d-1,\\ 
	v_t&=\sum_{i=0}^{t+r-d-1}(a_r^ix-b_r^i)u_i-\sum_{j=t+r-d}^{d-1}b_r^ju_j, \ \  d-r+1\leq t\leq d-1. 
	\end{align*}
\end{example}

\begin{proof}[Proof of Theorem \ref{NCQO2}] 
	$(1)\implies (2)$ Assume that $\left\{P_{n}\right\} $ is $d$-quasi-orthogonal of order at most $l$ with respect to $\mathcal{V}$, that is to say, 
	\begin{equation}
	\left\langle v_{t},P_n\right\rangle=0, \ \ n\geq d(l-1)+r+t+1, \ \ 1\leq r\leq d, \ \ 0\leq t\leq d-1.\label{QOL1}
	\end{equation}
	Then, we have \cite[Propo. 3.1]{Sadek2} $\mathcal{V}=\Phi\mathcal{U}$, i.e.
	\begin{equation}
	v_t=\sum_{i=0}^{d-1}\phi_t^iu_i, \ 0\leq t\leq d-1, \ \ \max\limits_{t,i}\deg( \phi_t^i)= l.\label{QOL2}
	\end{equation}
	Let us have a close look at the degree of $\phi_t^i$. We shall replace (\ref{QOL2}) in (\ref{QOL1}) and discuss the degree of $\phi_t^i$ for each $t$. Therefore, for $t=d-1$ we have from the left hand side of     
	\begin{equation*}
	\left\langle v_{d-1},P_n\right\rangle=
	\sum_{i=0}^{r-1}
	\left\langle u_i,\phi_{d-1}^iP_n\right\rangle
	+\sum_{i=r}^{d-1}
	\left\langle u_i,\phi_{d-1}^iP_n\right\rangle=0, \ n\geq d(l-1)+r+d,
	\label{QOL3}
	\end{equation*}
	and from the first sum of the right hand side $n\geq dl+r$ which agrees with the left hand side. While in the second sum if $\deg\phi_t^i=l$ for some $r\leq i\leq d-1$, we get zero for $n\geq dl+i+1>dl+r$. Whence $\deg\phi_t^i\leq l-1$ for $r\leq i\leq d-1$.
	
	Repeat the same resonance on $t$ until $t=d-r$ for which we write 
	\begin{equation*}
	\left\langle v_{d-r},P_n\right\rangle=
	\left\langle u_i,\phi_{d-r}^0P_n\right\rangle
	+\sum_{i=1}^{d-1}
	\left\langle u_i,\phi_{d-r}^iP_n\right\rangle
	\label{QOL4}
	\end{equation*}
	then from the left hand side we have zero for $n\geq dl+1$ and from the right hand side, the first term gives zero for $n\geq dl+1$ while the degree of  $\phi_{d-r}^i$ under the sum, i.e. for $1\leq i\leq d-1$,  should be less or equal to $l-1$.
	
	From these we can see, for $t=d-r-1$, that $\deg \phi_{d-r-1}^i\leq l-1$ for $0\leq i\leq d-1$.
	
	Therefore, for  $t=d-r-2$ we have
	\begin{equation*}
	\left\langle v_{d-r-2},P_n\right\rangle=\sum_{i=0}^{d-2}
	\left\langle u_i,\phi_{d-r-2}^iP_n\right\rangle
	+\left\langle u_{d-1},\phi_{d-r-2}^{d-1}P_n\right\rangle.
	\label{QOL5}
	\end{equation*}
	The left hand side gives zero for $n\geq dl-1$. Accordingly, from the right hand side the degree of $\phi_{d-r}^i$ under the sum is $\leq l-1$. While the degree of the last term should be less or equal to $l-2$. Whence the result.

	$(2)\implies (3)$ Trivial from (\ref{QOL1}).  
	
	$(3)\implies (1)$	Suppose that we have the formula (\ref{QO3}), then 
	\begin{equation*}
	\left\langle v_{t},P_n\right\rangle=\left\langle
	v_{t},Q_{n}\right\rangle +...+a_{n,(l-1)d+r} \left\langle
	v_{t},Q_{n-(l-1)d-r}\right\rangle =0,\ 
	n\geq d(l-1)+r+t+1.
	\end{equation*}
	Now since $d(l-1)<d(l-1)+r\leq dl$, it follows that  $\left\{P_{n}\right\} $ is $d$-quasi-orthogonal of order at most $l$ with respect to $\mathcal{V}$.
	
	$(2)\implies (4)$ Consider the matrix $\Phi ^{\ast }=\left(\chi_{t}^{s}\right)$ such that
	$\Phi ^{\ast }=\left( \det \Phi \right) \Phi ^{-1}$.
	Then, multiplying both sides of $\mathcal{V}=\Phi \mathcal{U}$ by $\Phi ^{\ast }$ and put $\det \Phi =\pi_{\sigma}$, we obtain
	$\pi_{\sigma} \mathcal{U}=\Phi ^{\ast }\mathcal{V}$,
	with (\ref{QO5}).
	
	$(4)\Leftrightarrow (5)$ From \ $\pi_{\sigma} \mathcal{U}=\Phi ^{\ast }\mathcal{V}$ it is easily seen that
	\begin{equation*}
	\left\langle \pi_{\sigma}u_{t},P_mQ_n\right\rangle=\sum_{i=0}^{d-1}
	\left\langle v_i,\chi_{t}^iP_mQ_n\right\rangle=0, \ n\geq d\big(m+(d-1)(l-1)+r\big)+i+1.
	\end{equation*}
	Therefore, taking into account (\ref{QO5}), we have from 
	\cite[Cor. 2.16]{Sadek3} the equivalence between (4) and (5) with $\rho=d\big((d-1)(l-1)+r\big)-r$ which corresponds to the max of degree, i.e. for $t=d-1$ and then $i=d-r-1$.
	
	$(4)\implies (1)$ From $\pi_{\sigma} \mathcal{U}=\Phi ^{\ast }\mathcal{V}$ we get
	\begin{equation}
	\left\langle u_{t},\pi_{\sigma}P_n\right\rangle=
	\sum_{i=0}^{d-1}
	\left\langle v_i,\chi_{t}^iP_n\right\rangle =0, \ n\geq d\big(d(l-1)+r\big)+t+1.\label{QOL6}
	\end{equation}
	Now, for the $\max$ of $t$, we have $\deg \chi_{t}^i\leq (d-1)(l-1)+r$ and $i\leq d-r-1$, then (\ref{QOL6}) is equivalent, with $t=d-1$ and $i=d-r-1$, to 
	\begin{align}
	n\geq 	d\big(d-1)(l-1)+r\big)+\bm{d(l-1)+r}+d-r,
	\end{align}
	therefore, from the fact $1\leq r\leq d$ we get  $d(l-1)<d(l-1)+r\leq dl$ which shows that $\left\lbrace P_n \right\rbrace_n $ is $d$-quasi-orthogonal of order at most $l$.
\end{proof}

Using dual sequences and far from the quasi-orthogonality, as it was shown in Maroni \cite[Prop. 6.2, Prop. 6.3]{MaroniFT} the if of an analogue of the first structure relation for classical OPS (in the standard orthogonality) known as Al-Salam-Chihara's characterization \cite{Al-Salam2} in the particular case $d=2$, i.e. every Hahn-classical $2$-OPS satisfies the first structure relation! Now Theorem \ref{NCQO2} enables us to extract  an analogue of the first structure relation for any $d\geq 1$.
In fact, $\left\{P_{n}\right\} $ is Hahn-classical (semi-classical of class s=0) $d$-OPS, 
means that the derivative sequence $\left\{Q_{n}=P_{n+1}^{\prime }/(n+1)\right\} $ is also $d$-orthogonal. 
In this case, the sequence $\left\{ P_{n}\right\} $ itself is $d$-quasi-orthogonal
of order two at most with respect to the vector linear form of $\left\{Q_{n}\right\} $ \cite{Sadek3}. And from \cite[Th. 3.1]{DouakCar} we can notice that the degree of the entries of the matrix  $\Phi$ corresponds to the case $l=2$ and $r=1$ at (\ref{E25}) of the Theorem  \ref{NCQO2} above which is explicitly constructed by Douak and Maroni at Pearson equation using the dual sequences \cite[Th. 3.1]{DouakCar}. Consequently, we shall summarize all of these in the following corollary

\begin{corollary}
	When the sequence $\left\{ P_{n}\right\} $ is Hahn-classical $d$-OPS  the above two structure relations reduce to the following expressions
	\begin{enumerate}
		\item[(1)] There exit an integer $1\leq r\leq d+1$ and complex numbers $a_{n,k}$ such that 
		\begin{equation}
		P_n(x)=Q_{n}(x)+\sum_{i=1}^{r}	a_{n,i}Q_{n-i}(x),  \label{CSR2}
		\end{equation}

		\item[(2)] There exist polynomial $\pi_{l} $ of degree  $l\leq d+1$ and $\rho\leq d^2-1$ such that 
		\begin{align}
		\pi_{l}(x)Q_n(x)&=
		\sum_{v=n-\rho}^{n+l}a_{n,v}P_v(x). \label{CSR3}
		\end{align} 
	\end{enumerate}
\end{corollary}

\begin{remark}
	Let us remark that in the case $d=1$ we get exactly the first structure relation of Al-Salam and Chihara \cite{Al-Salam2} and when $d=2$ we get the structure relation constructed by Maroni using the dual sequences \cite[Prop. 6.3]{MaroniFT} and as we mentioned above Maroni was not able following this approach to close the implication, i.e. to show that every solution of the first structure relation is Hahn-classical.
\end{remark}

For the standard orthogonality ($d=1$), it is well known that if a sequence $\left\{P_n\right\}$ is classical orthogonal with respect to some linear form satisfying Pearson equation $\left( \alpha U\right)'=\beta U$, then their derivative sequences 
of any order $k\ge 1$ are again classical and orthogonal with respect to $V=\alpha^k U$, whereas, this is not direct conclusion if $d\geq 2$. 
In fact, the normalized derivative sequence denoted  $\left\{P_{n}^{[1]}\right\} $ is $d$-orthogonal with respect to $\mathcal{V}=\Phi\mathcal{U}$ \cite{DouakCar,Sadek3}. Suppose that there exists matrix $\Upsilon$ such that
$\Upsilon\Phi\Psi=\Psi\Upsilon\Phi$, then $\left(\Phi_1\mathcal{V}\right)'=\Psi_1\mathcal{V}$ 
with $\Phi_1=\Upsilon\Phi$ and $\Psi_1=\Phi_1'+\Psi\Upsilon$.
Thus, the $d$-orthogonality of second derivative sequence is judged according to the matrix $\Upsilon$. 

As conclusion, at the beginning we thought that we should distinguish whether the derivative sequence is $d$-OPS, then we  proposed to call a $d$-OPS families for which their derivatives of any order are still $d$-OPS, a very classical $d$-OPS. Although, we could give a first characterization of the very classical $d$-OPS families (see \cite[cor.13]{Sadek4}).

\begin{corollary}
	A $d$-OPS $\left\{P_n\right\}$ is very classical if and only if there exist complex numbers $\lambda _{n,\nu}$ not all zero such that 
	\begin{equation}
	P_{n}^{[m]}(x)=\sum_{i=0}^{d+1}\tilde{\lambda} _{n,\nu}P_{n-\nu}^{[m+1]}(x),\ \forall n,m\geq 0.\label{CSR4}
	\end{equation}
\end{corollary}

Furthermore, by differentiating (\ref{CSR2}) $m$ times we find (\ref{CSR4}) which characterizes the very classical $d$-OPS. Accordingly, all Hahn-classical $d$-OPS are in fact very  classical $d$-OPS, and then there is no need to introduce this appellation.

\begin{example}
	First of all, the structure relation (\ref{CSR2})  is satisfied by the
	$d$-Appell ($d$-Hermite) \cite{DouakAppell} $\lambda _{n,\nu}=0$, $\nu \neq n$. 
	The $2$-Laguerre polynomials satisfy (\ref{CSR2}) with $r=d=2$ \cite[(3.9)]{CheikhDouakBateman} (quai-orthogonality of order exactly 1) and (\ref{CSR3}) 
	with $\pi(x)=x$ \cite[(2.27)]{CheikhDouakBateman} (See also  \cite[\S 4.1]{Sadek6} for more general case). 
	The $d$-analogue of $q$-Meixner and big $q$-Laguerre \cite[prop. 3.4]{LamiriqMeixqLag} 
	as well as little $q$-Laguerre \cite[prop.3.2]{CheikhqLag}  are Hahn classical $d$-OPS, then by following \cite[7.2.13]{Gasper-Rahman} we can show that they satisfy structure relations of type (\ref{CSR2}) as well as (\ref{CSR3}).
\end{example}

The question now is: can we determine all Hahn classical $d$-OPS families (for fixed $d$) as in the work of Al-Salam-Chihara \cite{Al-Salam2}?
The next section provides further information about Hahn's classical $d$-OPS.

It is worthy to point out that our characterization (the second structure relation) of Hahn-classical $d$-OPS (\ref{CSR2}) is quite useful, and a very natural way to construct new families of $d$-OPS with Hahn's property by looking for their corresponding (exponential) generating functions. Indeed, we started this direction by considering at first time a linear combination of two consecutive terms with different operators and we obtained very interesting results (see \cite{Sadek6} for more details).

\section{$(d+1)$-decomposition and $d$-symmetrization}\label{sec:sym}

Starting from the classical problem of symmetrization \cite{Chihara}, Douak and Maroni introduced a natural generalization, i.e., the $d$-symmetrization
as well as the ($d+1$)-decomposition of a sequence of polynomials.

A sequence of polynomials $\left\{B_n\right\}$ is called $d$-symmetric, if it
fulfills $B_n(\xi_kx)=\xi_k^nB_n(x)$, for each $0\leq k\leq d$ and $n\geq 0$
where $\xi_k=\exp\{2ik\pi/(d+1)\}$. Notice that, when $d=1$, we get the
definition of a symmetric sequence $B_n(-x)=(-1)^nB_n(x)$.

Besides, a vector form $\mathcal{V}=( v_{0},...,v_{d-1}) ^{T}$ is called $d$-symmetric, if for each $0\leq r\leq d-1$. The moments of the linear form $v_r$ satisfy 
\begin{equation*}
(v_r)_{(d+1)n+s}=0 \ \ \text{whenever} \ \ r\neq s, \ \ 0\leq s\leq d, \ \
n\geq 0.
\end{equation*}

As a result, $\left\{B_n\right\}$ is $d$-symmetric, if and only if it could be written as
\begin{equation}
B_{(d+1)n+s}=x^s\sum_{p=0}^{n}a_{(d+1)n+s,(d+1)p+s}x^{(d+1)p}, \ \ 0\leq
s\leq d, \ \ n\geq 0.  \label{SD1}
\end{equation}

The components of the sequence $\left\{B_n\right\}$ are $d+1$ sequences
denoted by $\left\{B_n^{i}\right\}$, for $0\leq i\leq d$, and defined as follows 
\begin{equation}
B_{(d+1)n+s}(x)=x^sB_n^s(x^{d+1}), \ \ 0\leq s\leq d, \ \ n\geq 0.
\label{SD2}
\end{equation}

In the same paper, they additionally proved that a sequence of polynomials is $d$-symmetric, if and only if the corresponding vector of linear forms is $d$-symmetric. 
Furthermore, a $d$-symmetric sequence of polynomials 
$\left\{B_n\right\}$ is $d$-OPS, if and only if it satisfies a ($d+1$)-order
linear recurrence relation of the form 
\begin{equation}
\left\{ 
\begin{array}{l}
B_{n+d+1}(x)=xB_{n+d}(x)-\rho_{n+1}B_n(x), \ \ n\geq 0,\vspace{.2cm} \\ 
B_n(x)=x^n, \ \ 0\leq n\leq d.%
\end{array}
\right.  \label{SD3}
\end{equation}

As it was pointed out by Douak and Maroni \cite[p.85-86]{Maroni2classic}, that first each component of a $d$-symmetric $d$-OPS is again $d$-OPS and, second, there exist some links between the components. Besides that links, we can also
give a connection between any two components either consecutive or not
adjacent. Further, the last component is connected to the first one which presents, in fact, kernel polynomials as it was shown by Chihara \cite[p.45]{Chihara}, i.e., the second component in the quadratic decomposition of a symmetric OPS defines kernel polynomials.

Indeed, replace first $n+d+1$ in (\ref{SD3}) by $(d+1)n$ we get 
\begin{equation*}
B_{(d+1)n}(x)=xB_{(d+1)(n-1)+d}(x)-\rho_{(d+1)(n-1)+1}B_{(d+1)(n-1)}(x),
\end{equation*}
and according to (\ref{SD2}), we find 
\begin{equation}
xB_{n}^d(x)=B_{n+1}^0(x)+\rho_{(d+1)n+1}B_{n}^0(x).  \label{DT11}
\end{equation}

Furthermore, since the components are not $d$-symmetric (see (\ref{SD8}) below), then it results that $B_n^0(0)\neq 0$, $\forall n\geq 0$, we thus obtain 
\begin{equation*}
B_{(d+1)(n+1)}(0)=B_{n+1}^0(0)=-\rho_{(d+1)n+1}B_n^0(0),
\end{equation*}
i.e., 
\begin{equation*}
\rho_{(d+1)n+1}=-B_{n+1}^0(0)/B_n^0(0)=m_{n+1}.
\end{equation*}

This shows, taking into consideration the result of \autoref{sec:kern}, that the $(d+1)^{th}$ and the first component are respectively the kernel polynomials $\left\{K_n\right\}$, 
and the original sequence $\left\{P_n\right\}$ 
(see Chihara's book for a comparative results).

Next, by replacing $n+d+1$ in (\ref{SD3}) by $(d+1)n+s$ for $s=1$ to $d$ recursively, and making use of (\ref{SD2}) once again, we get an analogue result of that mentioned in (\ref{DTB}) 
\begin{equation}
B_{n+1}^i=B_{n+1}^{i+1}+\rho_{(d+1)n+i+2}B_n^{i+1}, \ \ 0\leq i\leq d-1. \label{SD4}
\end{equation}

The latter connection shows that the component $B_n^i$ can be expressed as a linear combination in terms of any other component $B_n^k$. 
Indeed, replace recursively the system obtained above, we get the following recursion 
\begin{align}
&B_{n+1}^i=B_{n+1}^{i+k}+\label{SD6}\\\nonumber         &
\sum\limits_{\substack{ t=1  \\ 2\leq
		j_1<j_2<\dots<j_t\leq k+1}}^{k}
\rho_{(d+1)n+i+j_1}\rho_{(d+1)(n-1)+i+j_2}\dots\rho_{(d+1)(n-t+1)+i+j_t}\
B_{n+1-t}^{i+k},  
\end{align}
which shows, by taking $i=0$ and $k=d$ therein, that the
coefficients $l_{n,k}$ in (\ref{DT8}) (or equivalently in the expansion (\ref{DT9})) are given explicitly in terms of $\rho_n$, and if we replace (\ref{SD6}) with $i=0$ in (\ref{DT11}) we find the recursion of Douak and Maroni which is written in the following form 
\begin{align}
&xB_n^d=B_{n+1}^i+ \label{SD5}\\&
\sum\limits_{\substack{ t=0  \\ 1\leq j_0<j_1<\dots<j_t\leq
		i+1}}^{i} \rho_{(d+1)n+j_0}\rho_{(d+1)(n-1)+j_1}\dots\rho_{(d+1)(n-t)+j_t}\
B_{n-t}^{i}, \ \ 0\leq i\leq d. \nonumber
\end{align}

In other words, recursions (\ref{SD4})-(\ref{SD5}) all together, could
define kernel polynomials and the first component as a linear
combination in ($i+1$)-term of the component $\left\{B_n^i\right\}$.

As a matter of fact, the system (\ref{SD6}) is also quite useful to determine the coefficients of the recurrence relation for each component \cite[thm. 2.3]{Blel} 
(see also \cite[p. 88]{Maroni2classic}). 
First off, the recurrence relation of kernel polynomials is already obtained from (\ref{SD5}) with $i=d $. Whilst, the recurrence of any other component can be obtained by combining (\ref{SD6}) and (\ref{SD4}). 
Instead of following this long way, the recurrence of the ongoing sequences are already obtained in the proof of 
\cite[thm. 2.3]{Blel} implicitly. In what follows, the explicit
form is needed. Hence as a much deeper, with extra computation we explicitly have 
\begin{equation}
x^{r}B_{n}=B_{n+r}+\sum_{\substack{ k=1  \\ 1\leq i_{1}\leq \dots \leq
		i_{k}\leq r-k+1}}^{r}\rho _{n+i_{1}-d}\rho _{n+i_{2}-2d}\dots \rho
_{n+i_{k}-kd}B_{n+r-k(d+1)}.  \label{SD7}
\end{equation}

Now, with $r=d+1$ and $n\rightarrow n(d+1)+s$, we get using (\ref{SD2}) 
\begin{align}
&xB_{n}^{s}=B_{n+1}^{s}+
\sum_{\substack{ k=1  \\ 1\leq i_{1}\leq \dots \leq
		i_{k}\leq d-k+2}}^{d+1}\rho _{(n-1)(d+1)+i_{1}+1+s}\dots \rho
_{(n-k)(d+1)+i_{k}+k+s}B_{n+1-k}^{s},  \label{SD8}
\end{align}%
when $d=2$ the recursion (\ref{SD8}) reduces to that given in \cite[p.88]{Maroni2classic}. 

It is easily concluded from the preceding that the above recurrence of 
components (\ref{SD8}) could be used to give an analogous of Chihara's
theorem \cite[thm.9.1, p.46]{Chihara}. Actually, we have already pursued
this viewpoint here to show out the recursion (\ref{SD8}). Accordingly, the
relation (\ref{SD1}) is satisfied whenever (\ref{SD8}) is satisfied and
vice-versa.

Back to \autoref{sec:lu}, to show that sequences
generated by Darboux transformations (\ref{DT}) are in fact our $\left\{
B_{n}^{i}\right\} $. For this end, we shall formulate sufficient condition in
terms of the coefficients which makes them easy to check.

According to (\ref{SD4}), the components are defined recursively as 
$\mathbb{B}^{i}=R_i\ \mathbb{B}^{i+1}$ where $R_i$ are bidiagonal matrices with 1 in
the main diagonal and $\rho_{m(d+1)+i+2}$ at the position $(m+2,m+1)$ for $m\geq 0$. 
Since the polynomials are monic, it suffice then to show
that $L_{i+1}=R_i$ where $L_i$ are the matrices given in (\ref{DT})-(\ref%
{DTB}). We proceed by induction on $i$ and we shall show that $J_d^i$ is the
corresponding lower Hessenberg matrix of $\left\{ B_{n}^{i}\right\} $. On one
hand, since $J_d^0$ is the lower Hessenberg matrix for the first component, it follows
then that $xR_0\mathbb{B}^{1}=J_d^0R_0\mathbb{B}^{1}$, that is $%
J_d^1=R_0^{-1}J_d^0R_0$. On the other hand, from $L_1J_d^1=J_d^0L_1$ in (\ref%
{DTB}), it follows by comparing the entries in both main diagonal using the
recursion (\ref{SD8}) that $L_1=R_0$. From which the result follows readily
by induction.

Consequently, this is another easy way to prove that the lower matrix in (%
\ref{DTB}) is the product $L=R_0R_1\dots R_{d-1}$.

A word about the measure of orthogonalities. Let us first denote the
corresponding orthogonality's vector form of $\left\{ B_{n}^{i}\right\} $ by 
$\mathcal{U}^{i}=\left( u_{0}^{i},...,u_{d-1}^{i}\right) ^{T}$. It has been
already shown in \cite[thm. 3]{Sadek4} that $\mathcal{U}^{i+1}=\Phi _{i}%
\mathcal{U}^{i}$ where $\Phi _{i}=\left\{ \phi _{v}^{\mu }(i)\right\} _{\mu
	,v=0}^{d-1}$ is a $d\times d$ matrix polynomial with entries 
\begin{equation*}
\begin{tabular}{ll}
$\phi _{r}^{r+1}(i)=\rho _{(d+1)r+i+2},$ & $\phi _{r}^{r}(i)=1,\ \ $for $%
\text{\ \ }0\leq r\leq d-2,\vspace{0.2cm}$ \\ 
$\phi _{r}^{\mu }\equiv 0,$ & for $r+2\leq \mu \leq d-1\ $ and $\ 0\leq
r\leq d-3,\vspace{0.2cm}$ \\ 
$\phi _{r}^{\mu }\equiv 0,$ & for $0\leq \mu \leq r-1\ $ and \ $1\leq r\leq
d-2,\vspace{0.2cm}$ \\ 
$\phi _{d-1}^{0}(i)=\left( \rho _{d^{2}+i+1}/\gamma _{1}^{0}(i)\right)
\left( x-\beta _{0}(i)\right) ,\ \ $ & $\phi _{d-1}^{d-1}(i)=1-\left( \rho
_{d^{2}+i+1}\ \gamma _{1}^{1}(i)\right) /\gamma _{1}^{0}(i)$, \vspace{0.2cm}
\\ 
$\phi _{d-1}^{\mu }(i)=-\left( \rho _{d^{2}+i+1}\ \gamma _{1}^{d-\mu
}(i)\right) /\gamma _{1}^{0}(i),\ \ $ & for $1\leq \mu \leq d-2$,%
\end{tabular}
\label{Q4}
\end{equation*}
where $\beta _{n}(i)$ and $\gamma _{n}^{r}(i)$ are the coefficients of the ($d+2$)-term recurrence relation of $\left\{ B_{n}^{i}\right\} $ which could be all expressed in terms of $\rho _{n}$ from (\ref{SD8}). For instance, we have 
\begin{equation*}
\gamma _{n}^{0}(i)=\prod_{\nu =n}^{d+n}\rho _{(\nu -1)d+n+i},\ \ 0\leq i\leq
d-1,\ \ n\geq 1,
\end{equation*}
this shows further that the brackets (\ref{CO1}) could be expressed in terms
of $\rho _{n}$ as follows 
\begin{equation*}
\left\langle u_{r}^{i},P_{dn+r}^{i}P_{n}^{i}\right\rangle =\prod_{\nu
	=1}^{n-1}\ \prod_{j=(\nu -1)d+r+1}^{\nu d+r+1}\rho _{d(j+\nu -2)d+r+i+1},\ \
0\leq i\leq d-1,\ \ n\geq 1.
\end{equation*}

It thus follows, that all $\mathcal{U}^{i}$, $0\leq i\leq d$ could be
determined whenever one of them is explicitly known, and according to our
notation in the previous section, we infer that 
\begin{equation*}
\mathcal{V}=\mathcal{U}^{d}=\Phi_{d-1}\Phi_{d-2}\dots\Phi_{0} \mathcal{U}%
^0=\Phi \mathcal{U},
\end{equation*}
which is another expression and of course another way to determine the
matrix $\Phi$ given in (\ref{L2})-(\ref{L3}). It is well known that when $%
d=1 $, the orthogonality's measure of kernel polynomials is just the
corresponding measure of $B_n^0$ multiplied by $x$ \cite[p. 35]{Chihara}, or
more generally times $x-c$ if one consider the transformation $UL+cI$ instead
(in our case $c=0$).

Remark further that, the corresponding matrix Pearson equation of a Hahn classical $d$-OPS \cite{DouakCar} is very nice in the $d$-symmetric case. Indeed, if we
denote the vector form of the sequence $\left\{ B_{n}\right\} $ by 
$\mathcal{W}=(w_{0},\dots ,w_{d-1})^{T}$ and by adding tilde to the corresponding
vector form as well as to the recurrence coefficients for its derivative sequence,
then the connection between the two forms might be written as 
\begin{equation}
\begin{array}{rllll}
\widetilde{w}_{d-1} & =a_{d-1}^{d}x^{2}w_{0}+b_{d-1}^{d-1}w_{d-1},%
\vspace{2mm} &  &  \\ 
\widetilde{w}_{r} & =b_{r}^{r}w_{r}+a_{r}^{r+1}xw_{r+1}, &  & & \text{for} \ \ 0\leq
r\leq d-2,
\vspace{2mm} \\ 
b_{r}^{r} & =\displaystyle\frac{(d+r+1)\widetilde{\rho }_{r+1}^{0}-r\rho
	_{r+2}^{0}}{\rho _{r+1}^{0}+\rho _{r+2}^{0}}, & & b_{r}^{r}+a_{r}^{r+1}=1,
& \text{for} \ \ 0\leq r\leq d-1.
\end{array} \label{SD9}
\end{equation}
It should be noted that from the connection $u_{r}^{i}=\sigma _{d+1}\left(
x^{i}w_{r(d+1)+i}\right) $, for $0\leq i\leq d$, the dual sequence of components of the derivative sequence of $\left\{ B_{n}\right\} $ 
\cite[(5.4)]{Maroni2classic} could be also expressed explicitly in terms of $w_{r}$
using (\ref{SD9}).

Let us now move on to Hahn's property. When the sequence $\left\{B_n\right\}$ is $d$-symmetric and Hahn classical, 
then, first from \cite[cor. 4.7]{CheikhRomdSym} we know that its derivative is again $d$-symmetric 
and Hahn classical $d$-OPS (which can be obtained differently from (\ref{CSR2})), and
second from the paper of Blel \cite[thm 2.4]{Blel}, that all the components
are also Hahn classical $d$-OPS. In other words, these results mean that 
\begin{theorem}
The derivative of any order of a $d$-symmetric Hahn classical $d$-OPS and its components
are again $d$-symmetric Hahn classical $d$-OPS and Hahn classical $d$-OPS, respectively.
\end{theorem}

Hence, an interesting question to think about is the converse, i.e., is a $d$-symmetric sequence possesses Hahn's property if their components are all of Hahn type (Hahn classical), or at
least one of them is needed to be of Hahn type? 
Further, does the kernel sequence of Hahn classical one also of Hahn type?

It is not at all difficult to see why the inverse situation is also true.
Indeed, given a $d$-symmetric sequence $\left\{ B_{n}\right\} $ and denote its
components by $\left\{ B_{n}^{s}\right\} $. Then from \cite[thm. 5.3]{Maroni2classic}
we know that the sequence $\left\{ A_{n}=(n+1)^{-1}B_{n+1}^{\prime }\right\} $ is also
$d$-symmetric, we denote its components as $\left\{ A_{n}^{s}\right\} $. Now,
we assume that the components $\left\{ B_{n}^{s}\right\} $ are Hahn classical $d$-OPS, 
and we try to prove that $\left\{ B_{n}\right\} $ is Hahn classical $d$-OPS,
namely that $\left\{ B_{n}\right\} $ and $\left\{ A_{n}\right\} $ are both
$d$-OPS. Actually, it suffice to show that if the components $\left\{
B_{n}^{s}\right\} $ are $d$-OPS, then $\left\{ B_{n}\right\} $ should be $d$-OPS too. 
In fact, from (\ref{SD8}) with $s=0$, we get a recurrence of the type (\ref{SD7}) 
\begin{align*}
&x^{d+1}B_{n}=B_{n+d+1}+\\&
\sum_{\substack{ k=1  \\ 1\leq i_{1}\leq \dots \leq
		i_{k}\leq d-k+2}}^{d+1}\rho _{n-(d+1)+i_{1}+1}\dots \rho
_{n-k(d+1)+i_{k}+k}B_{n+(1-k)(d+1)}.
\end{align*}
From this, by taking $n=d(d+1)+r$, and from the definition of dual sequences it follows that
\begin{equation*}
\left\langle w_{r},x^{d+1}B_{d(d+1)+r}\right\rangle =\prod_{\nu
	=1}^{d+1}\rho _{(d-\nu )(d+1)+\nu +r+1}\left\langle w_{r},B_{r}\right\rangle
\neq 0,
\end{equation*}
whence the desired result.

Remark that we only needed to suppose that the first component is $d$-OPS. It is sufficient enough to prove the $d$-orthogonality of its derivative to suppose only that the first component is Hahn classical. 
Indeed, it is well understood that $(n+1)^{-1}\left(B_{n+1}^0\right)^{\prime }=A_n^d$ 
\cite[thm. 5.3]{Maroni2classic}. Hence, the result follows in the same easy way by taking care now of $A_n^d$ as above.

Consequently, this answers affirmatively to the two previous questions. In other words, it results from the above discussion,  that if the first component is Hahn 
classical $d$-OPS, then all the components and their corresponding
$d$-symmetric sequence are also Hahn classical $d$-OPS as well.

As a conclusion of this section, it results that under the umbrella of each
$d$-symmetric Hahn classical $d$-OPS there are ($d+1$) Hahn classical (nonsymmetric) $d$-OPS
families. And since there are  $2^d$ $d$-symmetric Hahn classical $d$-OPS, then 
the $d$-symmetric sequences and their components constitute a set of $(d+1)2^d$ 
families of Hahn classical $d$-OPS.

This gives evidence to think about the following question: Can we distribute
all Hahn classical $d$-OPS onto $d$ sets? Namely, is it possible to consider $d$ sets
of specific $d$-OPS families in which the recurrence coefficients are all zero except 
one parameter ($d$-symmetric case), two parameter, ..., $d$ parameter and repeat 
the above study in order to generate the maximum number possible of Hahn classical $d$-OPS 
into which their recurrence coefficients are determined by only one parameter 
($d$-symmetric case), two parameter, ..., $d$ parameter? 
Second: if the previous situation possible, is there any $d$-OPS of Hahn type which could not be a component 
of any of the above $d$ sets, i.e., under no umbrella?

\begin{example}
	A good example for a $d$-symmetric sequence and its components, is presented
	in \cite{CheikhRomdSymBrenke}. 
	The authors showed that there are only two $d$-symmetric $d$-OPS families of Brenke
	type. Furthermore, the corresponding generating functions of the components
	are all explicitly determined. See also the components of $d$-Hermite
	polynomial \cite[p. 287-288]{DouakAppell} and \cite[sec. 5]{CheikhDouakBateman} for $2$-Laguerre.
	
	Let us consider $d$-Chebyshev polynomials of second kind \cite{MaroniChyI}
	with $\rho_n=\rho$. Accordingly, from (\ref{SD8}) it is readily seen that
	the recurrence coefficients differ only for the initial conditions, i.e., 
	$\beta_n(i)=\beta_n(j)$ and $\gamma_{n+1}^k(i)=\gamma_{n+1}^k(j) $ for $n\geq 1$.
	
	For instance, to check (\ref{DT10}) we need only to look at the initial conditions. 
	Remark that $\beta_0(i)=(i+1)\rho $ and $\gamma_{1}^{d-1}(i)=(i+1)\left(d-(i/2)\right)\rho^2$. 
	Hence, $B_1^d(x):=K_1(x)=x-(d+1)\gamma$ and $B_2^0(x):=P_2(x)=(x-\beta_1)(x-\rho)-d\rho^2$. 
	Consequently, we obtain 
	\begin{equation*}
	P_2(x)-\left(
	P_2(0)/P_1(0)\right)P_1(x)=(x-\beta_1)(x-\rho)-d\rho^2+(\beta_1-d\rho)(x-\rho)=xK_1.
	\end{equation*}
\end{example}

\section{Some properties of zeros}\label{sec:zero}

Facing now the set of zeros, and recall that a zero of a polynomial $\pi(x)$
at an interior point of $\left[a,b\right]$ is said to be nodal or nonnodal
according as $\pi(x)$ changes or does not change sign in the neighborhood of
the zero.

Let $\left\{ P_{n}\right\} _{n\geq 0}$ be $d$-OPS with respect to $U=\left(
u_{0},...,u_{d-1}\right) ^{T}$. The following theorem given by Maroni \cite%
{MaroniFSh} in the sense of $1/p$ orthogonality

\begin{proposition}
	\label{Z1} Suppose that $\gamma _{m+1}^{0}>0$, $m\geq 0$. Then each
	polynomial $P_{dn+q},$ $1\leq q\leq d$ has at least $n+1$ distinct nodal
	zeros.
\end{proposition}

The previous proposition stated without proof, but it is readily proved
using only the recurrence relation (see \cite[p. 56]{Iseghem} for a such
proof). Recall now the following definition

\begin{definition}
	\cite{Chihara} A moment functional $u$ is called positive definite if 
	$\langle u,\pi(x)\rangle>0$ for every polynomial $\pi(x)$ that is not
	identically zero and is non-negative for all real $x$.
\end{definition}

Since the moment of linear form may be expressed in terms of the
recurrence coefficients of the corresponding OPS, then it is straightforward that the respective OPS as well as the recurrence coefficients should be real in the positive definite case. 

Let $\left\{x_{i}\right\} _{i=1}^{k}$ be all the nodal zeros of $P_{m}$ and
set 
\begin{equation*}
\pi _{k}\left( x\right) =\left( x-x_{1}\right) ...\left( x-x_{k}\right),
\end{equation*}
then $\pi _{k}\left( x\right)P_{dk+r}\left( x\right) \geq 0$. In
addition, from (\ref{CO1}) and the definition of the $d$-orthogonality, we
have

\begin{equation*}
\langle u_{r},\pi _{k}(x)P_{dk+r}\rangle =\left\langle
u_{r},x^{k}P_{dk+r}\right\rangle =\prod\nolimits_{\nu =0}^{k-1}\gamma _{\nu
	d+r+1}^{0},
\end{equation*}
which gives explicitly the determinants $H_{md}$ in \cite[eq. (2.7), (2.9)]{Maronid}, 
whereas \cite[p.878]{Sadek2} 
\begin{equation*}
\begin{array}{ll}
\left\langle u_{r},x^{k}P_{dk+r-i}\right\rangle & =\displaystyle 
\sum_{j=0}^{i}\gamma _{d(k-1)+r+1-j}^{i-j}\left\langle
u_{r},x^{k-1}P_{d(k-1)+r-j}\right\rangle, \quad 1\leq i\leq d-1, 
\vspace{.2cm} \\ 
\left\langle u_{r},x^{k}P_{dk+r-d}\right\rangle & =\displaystyle 
\beta_{d(k-1)+r}\left\langle u_{r},x^{k-1}P_{d(k-1)+r}\right\rangle 
\vspace{.2cm} \\ &
+\displaystyle \sum_{j=0}^{d-1}\gamma _{d(k-1)+r+1-j}^{d-1-j}\left\langle
u_{r},x^{k-1}P_{d(k-1)+r-1-j}\right\rangle, \vspace{.2cm} \\ 
\left\langle u_{r},x^{k}P_{dk+r-l}\right\rangle & =\displaystyle 
\left\langle u_{r},x^{k-1}P_{dk+r-l+1}\right\rangle+
\beta_{dk+r-l}\left\langle u_{r},x^{k-1}P_{dk+r-l}\right\rangle \vspace{.2cm}\\ 
& +\displaystyle \sum_{j=0}^{d-1}\gamma _{dk+r-l-j}^{d-1-j}\left\langle
u_{r},x^{k-1}P_{dk+r-l-1-j}\right\rangle, \quad d+1\leq l\leq dk+r.
\end{array}%
\end{equation*}

Hence, on account of \cite[eq. (2.6)-(2.9)]{Maronid} and \cite[p.876-878]{Sadek2} we have the following result
\begin{proposition}
	$u_{r}$ is positive definite if and only if $\beta_{\nu},\ \gamma _{\nu +1}^{s}$,
	$1\leq s\leq d-1$, are real and $\gamma _{\nu +1}^{0}>0$,  $\forall \nu \geq 0 $.
\end{proposition}

Further, suppose that there are $s$ nonnodal zeros and let 
\begin{equation*}
\phi _{s}\left( x\right) =\left( x-x_{1}\right) ...\left( x-x_{s}\right),
\end{equation*}
then $\phi_sP_{dk+r}$ has only nodal zeros. Hence,
\begin{equation*}
\langle u_r,\phi_sP_{dn+q}\rangle = 0, \ \ dn+q\geq ds+r+1,
\end{equation*}
this means that $s\leq n-1$. We have then proved the following consequence of the
previous proposition

\begin{corollary}
	If there exist nonnodal zeros
	for the polynomial $P_{dn+q}$, $1\leq q\leq d$, then there are $n-1$
	distinct zeros at most.
\end{corollary}

Accordingly, from formulae (\ref{B65}) as well as (\ref{B27}) which shows
that $\Delta _{n}^{\left( r\right) }\neq 0$, we readily deduce that

\begin{corollary}
	\label{Z2} The multiplicity of zeros of any $d$-OPS is at most $d$.
	Moreover, any $d+1$ consecutive polynomials and any $d+1$ consecutive
	polynomials from the $r$-associated sequence $\left\{ P_{n}^{\left( r\right)
	}\right\} $, have no common zero. And for any $r\geq 0$, the polynomials $%
	P_{n}^{\left( r\right) }$, $P_{n}^{\left(r+1\right) }$,..., $P_{n}^{\left(
		r+d\right) }$ have no common zero.
\end{corollary}

Corollaries \ref{Z1} and \ref{Z2} show again that the zeros are simple when $d=1$.

\subsection{Chebyshev systems}

Zeros of OPS interlace as a consequence generally from
the recurrence relation.

A system of real functions $\left\{\mu_i\right\}_{i=0}^{d}$ defined on an
abstract set $E$ is called a Chebyshev system (T-system) of order $d$ on $E$
if any polynomial (any linear combination) 
\begin{equation*}
P(t)=\sum_{i=0}^{d}c_i\mu_i(t), \ \ \ \text{with} \ \ \ \sum c_i^2\neq 0,
\end{equation*}
has at most $d$ zeros on $E$ \cite{Krein}.

It is readily seen that $\left\{\mu_i\right\}_{i=0}^{d}$ is a T-system on $E$
if and only if the determinant 
\begin{equation*}
\det\left(\mathbb{P}_{0}(t_0)\ \mathbb{P}_{0}(t_1)...\mathbb{P}%
_{0}(t_d)\right)
\end{equation*}
does not vanish for any pairwise distinct $t_0,...,t_d\in E$. This follows
at once by considering a system of $n+1$ homogeneous equations 
\begin{equation*}
\sum_{i=0}^{d}c_i\mu_i(t_j)=0, \ \ \ j=0,1,...,d,
\end{equation*}
in $c_1,...,c_d$.

The interlacing property for the zeros of polynomials orthogonal with
respect to a Markov system, proved by Kershaw, with respect to Lebesgue measure in \cite{Kershaw}, and under a weak condition, with respect to the Borel measure in \cite{Lopez3}. 
The same argument used in \cite{AsscheZero}
to prove the interlacing property for the type II multiple OPS with respect to measures that form an AT system.

Recall that a system of measures $(\mu_1,...,\mu_r)$ forms an AT system for
the set of integers $(n_1,...,n_r)$ on $\left[a,b\right]$ if the measures $%
\mu_j$ are absolutely continuous with respect to a measure $\mu$ on $\left[%
a,b\right]$, with $d\mu_j(x)=\omega_j(x)d\mu(x)$ and 
\begin{equation*}
\left\{\omega_1,
x\omega_1,...,x^{n_1-1}\omega_1,\omega_2,...,x^{n_r-1}\omega_r\right\}
\end{equation*}
is a Chebyshev system on $\left[a,b\right]$ of order $n=n_1+...+n_r-1$ \cite%
{Krein,Nikish82}.

In view of all the above results, we are able to announce and prove the
following result

\begin{proposition}
	Let $\left\{P_n\right\}$ be an $d$-OPS with respect to an AT system $\mathcal{U%
	}=(\mu_0,...,\mu_{d-1})$. Then the zeros of $P_{dn+r}$ and $P_{dn+r+1}$
	interlace for $n\geq 0$ and $0\leq r\leq d-1$.
\end{proposition}

\begin{proof}
	The proof is analogue of that used to prove \cite[theorem 2.1]{AsscheZero}.
	Replace the polynomials $P_{\overrightarrow{n}}$ and 
	$P_{\overrightarrow{n}+	\overrightarrow{e}_k}$ by $P_{dn+r}$ and 
	$P_{dn+r+1}$ respectively and use the following determinant
	\begin{equation*}
	\begin{array}{l}
	W_{dn}(x_1,...,x_{dn+r-1})= \\ 
	\left\vert  
	\begin{array}{cc}
	\omega_1(x_1)\ ...\ x_1^{n}\omega_1(x_1) & \omega_2(x_1) \ ... \
	x_1^{n}\omega_{r-1}(x_1) \ ... \ x_1^{n-1}\omega_{r}(x_1) \ ... \
	x_1^{n-1}\omega_{d}(x_1) \\ 
	\omega_1(x_2)\ ... \ x_2^{n}\omega_1(x_2) & \omega_2(x_2) \ ... \
	x_2^{n}\omega_{r-1}(x_2) \ ... \ x_2^{n-1}\omega_{r}(x_2) \ ... \
	x_2^{n-1}\omega_{d}(x_2) \\ 
	\vdots \hspace{.5cm} \dots &  \\ 
	\omega_1(x_{dn})\ ... \ x_{dn}^{n}\omega_1(x_{dn}) & \omega_2(x_{dn}) \ ...
	\ x_{dn}^{n}\omega_{r-1}(x_{dn}) \ ... \ x_{dn}^{n-1}\omega_{r}(x_{dn}) \
	... \ x_{dn}^{n-1}\omega_{d}(x_{dn}) \\ 
	& 
	\end{array}
	\right\vert%
	\end{array}%
	\end{equation*}
	for the point $x_1,...,x_{dn+r-1}$ on $\left[a,b\right]$ instead. 
	
	Following the same resonance in \cite[theorem 2.1]{AsscheZero} we conclude
	that the zeros $x_k$ and $y_i$ of $P_{dn+r}$ and $P_{dn+r+1}$ respectively,
	are in the following situation 
	\begin{equation*}
	y_i<x_i<y_{i+1} \ \ \ \text{for} \ \ \ i=1,...,dn+r \ \ \ \text{and }\ \ \
	n\geq 0.
	\end{equation*}
\end{proof}

\subsection*{Remarks}

First of all, we are not working onto empty set. The zeros of $d$-Laguerre OP 
\cite{CheikhDouakII} are real, positive and simple for $\alpha _{i}+1>0$, $i=0,...,d$ \cite{HyperBessel}. This gives evidence to think whether is it possible to avoid such condition, i.e.,  to look if there are extra informations on the zeros that could be found out from the recurrence coefficients. Such sufficient conditions are termed out for type II multiple OPS \cite[thm. 2.2]{AsscheZero}. For the $d$-orthogonality, 
we believe that we could give an analogous of the latter condition using the
recurrence (\ref{SD8}) (resp. and maybe some analogue of it). Which 
provides conditions only on one parameter $\rho_{n}$ (resp. on few parameters).
Although next section shows, with the aid of totally nonnegative matrices, 
that such sufficient conditions on the recurrence coefficients are available.

\subsection{Totally positive matrix}

In this section, we provide new approach based on totally positive matrices
to show that zeros of $d$-OPS could be real and simple. Let us first recall some terminologies and definitions.

\begin{definition}
	A $n\times m$ matrix $A$ is said to be:
	\begin{enumerate}
		\item [(1)] totally nonnegative (TN) if all its minors are nonnegative. 
		\item [(2)] totally positive (TP) if all its minors are strictly positive. 
		\item [(3)] an oscillation matrix if $A$ is TN and some power of $A$ are TP.
	\end{enumerate}
\end{definition}

Oscillation matrix is an interesting class between TP and TN matrices was introduced by Gantmakher and Krein which share the spectral properties of TP matrices. It is more convenient to consider this class of non symmetric matrices with the oscillatory properties. Actually, there are relatively simple criteria for determining if a 
TN matrix is an oscillation matrix.

\begin{theorem}\label{TPZ1}
	\cite{GantmakherKrein37} A $n\times n$ matrix $A=\left( a_{i,j}\right)
	_{i,j=0}^{n}$ is an oscillation matrix if and only if $A$ is TN,
	nonsingular, and $a_{i,i+1}$,$a_{i+1,i}>0$, $i=1,...,n-1$. Furthermore, if $%
	A $ is oscillation matrix, then $A^{n-1}$ is TP.
\end{theorem}

If a matrix $A$ is TP (resp. TN), then $A^{T}$ (transpose of $A$) is TP (resp. TN) as well as every submatrix of $A$ and $A^{T}$ is TP (resp. TN). 	
Furthermore, since the product of TN matrices is TN matrix, then the
following proposition is with important interest in our study of zeros. It
could be also proved readily using planar network. 

\begin{proposition}
	\cite[p.155]{Pinkus} A bi-diagonal lower triangular matrix is TN if and
	only if all its elements are nonnegative.
\end{proposition}

The eigenvalues of oscillation matrices are simple and positive. Although,
the following theorem shows that the eigenvalues of the two principal
submatrices obtained form $A$ by deleting either the first row and column,
or the last row and column, strictly interlace the eigenvalues of $A$

\begin{proposition}\label{TPZ2}
	\cite[p.136]{Pinkus} Let $A$ be an $n\times n$ TP. Then its eigenvalues are
	positive and simple. In addition, if these eigenvalues are denoted by $%
	\lambda _{1}>...>\lambda _{n}>0$, and $\mu _{1}^{(k)}>...>\mu _{n-1}^{(k)}>0$
	are the eigenvalues of the principal submatrix of $A$ obtained by deleting
	its k$^{\text{th}}$ row and column, then%
	\begin{equation*}
	\lambda _{j}>\mu _{j}^{(k)}>\mu _{j+1}^{(k)},\hspace{0.5cm}j=1,...,n-1,
	\end{equation*}%
	for $k=1$ and $k=n.$
\end{proposition}

However, TP matrices are dense in the class of TN matrices \cite[th. 2.6]{Pinkus}, i.e., for a $n\times m$ TN matrices $A$ there exists a sequence of $n\times m$ TP matrices $\left\{A_k\right\}_{k\geq 1}$ such that $\lim\limits_{k\rightarrow\infty}A_k=A$. The latter fact allows us to assert that the eigenvalues of TN matrices are both real and nonnegative.

It results now from the above discussion the following conclusion. For a 
$d$-symmetric $d$-OPS $\left\{ B_{n}\right\} $, if we assume that 
$\left\{ \rho _{n}\right\} $ is a sequence of positive numbers, then it is readily seen
that the recurrence coefficients in (\ref{SD9}) are all positive. Furthermore,
according to the factorization $J_{d}^{i}=R_{i+1}...R_{d-1}UR_{0}...R_{i}$, \autoref{TPZ1} shows that 
the $n\times n$ leading lower Hessenberg submatrices  of the components are all oscillation matrices, 
and then their eigenvalues are positive and simple. This result has already proved for Faber polynomials  
in \cite{VargaZeroFaber} and for 4 term recurrence relation in \cite{AsscheAsympt08}.

Now we want to show that the zeros of any $d$-OPS are positive and simple whenever the 
recurrence coefficients are strictly positive. In this case, according to \autoref{TPZ1} 
it is enough to show that  $n\times n$ lower Hessenberg matrices  are TN for any integer $n$. 
That is to say, it is always possible to write the matrix $L$ as a product of $d$ bi-diagonal lower TN matrices. 

From \cite[thm.2]{Barrios}, we construct recursively our matrices $L_i$ such that $L=L_1\dots L_d$ 
where $L_i$ are bidiagonal matrices with 1 in the main diagonal and 
$l_{m+1}^i>0$ at the position $(m+2,m+1)$ for $m\geq 0$. 
Let us begin by constructing $L_1$. In other words, we look for  two matrices $L_1$ and $T_1$ 
with strict positive entries such that $L=L_1T_1$ where 
\begin{equation}
T_1=\left( 
\begin{array}{ccccc}
1 &  &  &  &  \\ 
t_{11} & 1 &  &  &  \\ 
\vdots & \ddots & \multicolumn{1}{l}{\ddots} & \multicolumn{1}{l}{} &  \\ 
t_{d-1,1} & \cdots & t_{d-1,d-1} & 1 &  \\ 
0 & \ddots &  &  & \ddots
\end{array}
\right) .  \label{DTZ1}
\end{equation}

Now by equating both sides ($L$ with the product $L_1T_1$), we get from the 
first line $l_{11}=t _{11} +l_1^1>0$. 
We choose the entries of the matrix $L_1$ recursively.
Suppose that $l_{m}^1>0$ are chosen up to some integer $k-1$. 
Then, the entries of $L$ at line $k+1$ show that $l_{k}^1$ could be 
chosen strict positive and satisfy the following inequalities
\begin{equation}
\begin{array}{l}
t_{k1}=l_{k1}-l_1^1\left(l_{k2}-l_2^1\left(l_{k3}-\dots 
\left(l_{k,k-1}-l_{k-1}^1\left(l_{kk}-l_k^1\right)\right)\right)\right) >0, \vspace{.2cm}\\
t_{kk}=l_{kk}-l_k^1 >0,\vspace{.2cm} \\
t_{ki}=l_{ki}-l_i^1 t_{k,i+1} >0, \ \ 1\leq i\leq k-1,
\end{array} \label{DTZ2}
\end{equation}
for $1\leq k\leq d-1$.

For $k\geq d$, the entries at line $d+i$, $i\geq 0$, show that $l_{i}^1$ and $t_{i,j}$ 
could be chosen to be strict positive in the following manner
\begin{equation}
\begin{array}{rl}
l_{d-1+i,i}& =  l_{i}^1 t_{d-1+i,i+1} ,  \ \ \ i\geq 1, \vspace{.2cm}\\ 
l_{d-1+i,j}& = t_{d-1+i,j} +  l_{j}^1 t_{d-1+i,j+1} , \  \ \ i+1\leq j \leq i+d-1.  
\end{array} \label{DTZ3}
\end{equation}

Repeatedly, we construct $L_2$, ..., $L_d$ with strict positive entries $l_{m}^i$ and, then, 
the leading submatrix of $L$ is TN. 

According to \autoref{TPZ1}, it terms out that our lower Hessenberg matrix is oscillation matrix.  Now Proposition \ref{TPZ2} asserts that the eigenvalues of lower Hessenberg matrix are positive and simple whenever its entries $\left\{\beta_n\right\}_n$ and 
$\left\{\gamma_n^i, \ 0\leq i\leq d-1\right\}_n$, as well as the entries $\left\{m_n\right\}_n$ and $\left\{l_{n,m}\right\}_{n,m}$ 
of the matrices $U$ and $L$ respectively \autopageref{sec:lu} are strict positive (see proposition \ref{T11}). 

Denoting by $(J)_n$ the leading principal submatrix of $J_d$ (see \ref{Y9}) of size $n\times n$, and by $(\mathbb{P})_n
=\left( P_{0}\left( x\right) ,...,
P_{n}\left( x\right)\right) ^{T}$, we get
\begin{equation}
x(\mathbb{P})_{n-1}=(J)_{n} (\mathbb{P})_{n-1} 
\end{equation}
if and only if $x$ is a zero of $P_n$. This identifies the zeros of $P_n$ as eigenvalues of the matrix $(J)_{n}$. This can also be seen by expanding 
the determinant $det\left(x(I)_n-(J)_n\right)$ along the last row to get that this determinant is $P_n(x)$. In the same way, it is readily seen that the zeros of $P_{n-1}$ (resp. $P_{n-1}^{(1)}$) are the eigenvalues of the principal submatrix of $(J)_{n}$ obtained by deleting its last (resp. first) row and column. Hence, according to Proposition \ref{TPZ2}, the zeros of $P_n$ and $P_{n-1}$ and that of $P_n$ and $P_{n-1}^{(1)}$ interlace.

However, this condition is too strong (see examples bellow) and one needs to look for weaker condition that ensures zero's simplicity and interlacing. 

\subsection{Examples}

Let us look at zeros of some $d$-OPS families.  Firstly, in \cite{HyperBessel} the authors tell us that zeros of $d$-Laguerre polynomials are positive and simple. Whereas the recurrence coefficients are not all positive (see \cite[p.597]{CheikhDouakII} for $d$=2). 
Accordingly, in account of this result, the strict positivity of the recurrence coefficients is sufficient but not necessary.

For $q$-Appell OPS ($d$=1), Al-Salam \cite{Al-Salamq-Appell} gives explicitly the 
recurrence coefficients. We can mimic him to get the recurrence coefficients for $d>2$ as follows
\begin{equation*}
\beta_n=q^n\beta_0, \ \gamma_{n+1}^{i}=\left[\begin{array}{c} n+d-i \\ d-i \end{array}\right]_qq^n\gamma_1^i, 
\ \  0\leq i\leq d-1, \ \ \forall n\geq 0.
\end{equation*}
Then, zero's interlacing as well as simplicity are guaranteed for $d$-analogue of $q$-Appell whenever $\beta_0$ and $\gamma_1^i$ are strict positive. 
When $q\rightarrow 1$ we find $d$-analogue of Appell ($d$-Hermite) studied by Douak \cite{DouakAppell} with the same conclusion.
$d$-Charlier polynomials \cite{CheikhZag1} are also of Appell type (known as $\Delta_w$-Appell or discrete Appell), defined by 
their recurrence coefficients $\beta_n=wn-\beta_0$ and $\gamma_{n+1}^i=-\beta_i(n+1)_{d-1}$ for $n\geq 0$. 
Accordingly, by $\beta_i<0$ for $0\leq i\leq d-1$, the interlacing property is satisfied.
The same conclusion for Dunkl-Appell $d$-OPS studied in \cite{CheikhGeid1} where the recurrence relation is
\begin{equation*}
xP_{n}(x)=P_{n+1}(x)-\sum\nolimits_{k=1}^{d}\beta_k \frac{\gamma_{\mu}(n)}{\gamma_{\mu}(n-k)} P_{n-k}(x).
\end{equation*}
Hence, a sufficient condition for simplicity of zeros is $\beta_k<0$, $k\geq 0$.

Humbert polynomials defined by the following generating function given in terms of hypergeometric function 
\begin{equation*}
\left(1-xt+t^{d+1}\right)^{-\alpha}=\left(1+t^{d+1}\right)^{\alpha}\ _1F_0\left(\alpha,-;\frac{xt}{1+t^{d+1}}\right)
=\sum_{n\geq 0}H_n^{\alpha}(x)t^n,
\end{equation*}
are $d$-symmetric. Their components denoted by $\left\{\frac{(\alpha)_r}{r!}B_n^{\alpha+r}(x,\left(\theta_r\right))\right\}$, are explicitly given by \cite{LamiriHumbert}
\begin{equation*}
B_n^{\alpha+r}(x,\left(\theta_r\right))=\frac{(-1)^n(\alpha+r)_n}{n!}\ 
_{d+1}F_d \left(\begin{array}{cc}
\begin{array}{c}
-n,\Delta(d,n+\alpha+r),\\
\left(\theta_r\right)
\end{array} &\bigg|q,\,z\end{array}\right) ,
\end{equation*}
where $\left(\theta_r\right)$ designates the set $\left\{\frac{\alpha+1+i}{d+1}; i=0,\dots,d \ \text{and} \ i\neq d-\alpha\right\}$ and $\Delta(p,a)$ abbreviates the array of p parameters $(a+i-1)/p$, for  $i=1,\dots,p$.

Monic Humbert polynomials satisfy the recurrence (\ref{SD3}) with
\begin{equation*}
\rho_{n+1}=\frac{(n+1)_{d+1}}{(\alpha+n)_{d+1}}\left(\frac{(d+1)(\alpha-1)}{n+d+1}+1\right).
\end{equation*}
Then, in order that the recurrence coefficients in (\ref{SD8}) be strict positive it suffices to take $\alpha>0$. Hence the interlacing properties are satisfied for components of Humbert polynomials. The same conclusion for component's zeros of $d$-symmetric Dunkl $d$-OPS \cite[p.213]{CheikhGeid2} since $\rho_n>0$ \cite[p.201]{CheikhGeid2}.

Let us consider the classical $d$-OPS generated by \cite{Sadek4}
\begin{equation*}
\exp \left\{ \frac{xt}{1-at} + 
\sum\limits_{k=0}^{d-1}b_{k}\ \frac{t^k}{k!}\right\}:=\sum\limits_{n=0}^{\infty }P_{n}\left( x\right) \frac{t^{n}}{n!}.
\end{equation*}
Notice that when $a=0$ the above generating function reduces to Appell ones \cite{DouakAppell}. Now, denote by $Q_{r}(x)=P_{r+1}^{\prime }(x)/\left( r+1\right) $. In this case, upon writing $b_{i}\equiv 0$ if $i\geq d$, we have
\begin{equation*}\begin{array}{rl}
P_{n+1}(x)&=\left(x+2an+b_1\right)P_n(x)-n\left[a^2(n-1)+2ab_1-b_2\right]P_{n-1}(x)
\vspace*{1.5mm}\\&\displaystyle+\sum\limits_{k=2}^d\binom{n}{k}\left(b_{k+1}-2akb_{k}+a^2k(k-1)b_{k-1}\right)P_{n-k}(x).
\end{array}
\end{equation*}
and 
\begin{equation*}\begin{array}{rl}
Q_{n+1}(x)&=\left(x+a(2n+1)+b_1\right)Q_n(x)
-n\left[a^{2}n+2ab_1-b_2\right]Q_{n-1}(x)
\vspace*{1.5mm}\\&\displaystyle+\sum\limits_{k=2}^d\binom{n}{k}\left(b_{k+1}-2akb_{k}+a^2k(k-1)b_{k-1}\right)Q_{n-k}(x).
\end{array}
\end{equation*}
Accordingly, the following conditions $a<0$ and $b_i<0$, $i=1,\dots,d$, are sufficient for the zeros to be positive and distinct for both of the latter sequences.

\section*{acknowledgements}
	The author would like to heartily thank Professor Jiang Zeng for his generosity in time, ideas and useful conversations.
	Part of this work was performed while the author stayed at 
	Institut Camille Jordan, Universit\'{e} Claude Bernard Lyon 1, 
	and the author kindly thanks this institution for hospitality.


\end{document}